\documentclass{article}

\usepackage{floatrow}

\newfloatcommand{capbtabbox}{table}[][\FBwidth]

\usepackage{blindtext}

\usepackage{hyperref}
\usepackage{epsfig}
\usepackage{subfigure}
\usepackage{amssymb,amsmath,amsfonts,amsthm}

\usepackage{algorithmic}
\usepackage{graphicx}
\usepackage{subfigure}
\usepackage{array}
\usepackage{enumerate}
\usepackage{enumitem}
\usepackage{rotating}
\usepackage{multirow}
\usepackage[T1]{fontenc}
\usepackage{algorithm}
\AtBeginDocument{
    \hypersetup{colorlinks,urlcolor=black}
}

\usepackage{pgfplots,tikz}

\newdimen\prevdp
\def\leftlabel#1{\noalign{\prevdp=\prevdepth
		\kern-\prevdp\nointerlineskip\vbox to0pt{\vss\hbox{#1}}\kern\prevdp}}

\DeclareMathOperator*{\argmax}{arg\,max}

\newtheorem{theorem}{Theorem}[section]

\newtheorem{lemma}{Lemma}[section]
\newtheorem{corollary}{Corollary}[section]
\newtheorem{example}{Example}[section]{}
\newtheorem{definition}{Definition}[section]
\newtheorem{remark}{Remark}[section] 
\newtheorem{assumption}{Assumption}[section]

\begin{document}
\title{Approximation of the Pseudospectral Abscissa via Eigenvalue Perturbation Theory}
\author{Waqar Ahmed\footnote{Data Scientist, Kent Kart Ege Elektronik Sanayi Ticaret A.\c{S}, M\"{u}rselpa\c{s}a Bulvar{\i} 35230, Konak, Izmir, Turkey. \texttt{waqarahmed695@gmail.com}.} 
				 $\;$and 
		Emre Mengi\footnote{Ko\c{c} University, Department of Mathematics, Rumeli Feneri Yolu 34450, Sar{\i}yer, Istanbul, Turkey. \texttt{emengi@ku.edu.tr}.}}
\date{}
\maketitle

\begin{abstract}
\noindent
Reliable and efficient computation of the pseudospectral
abscissa in the large-scale setting is still not settled. Unlike the small-scale 
setting where there are globally convergent criss-cross algorithms,
all algorithms in the large-scale setting proposed to date
are at best locally convergent. We first describe how eigenvalue perturbation
theory can be put in use to estimate the globally rightmost point 
in the $\epsilon$-pseudospectrum if $\epsilon$ is small. Our treatment
addresses both general nonlinear eigenvalue problems, and the standard
eigenvalue problem as a special case. For small $\epsilon$, the estimates
by eigenvalue perturbation theory are quite accurate. In the standard 
eigenvalue case, we even derive a formula with an ${\mathcal O}(\epsilon^3)$ error. 
For larger $\epsilon$, the estimates can be used to initialize the locally convergent
algorithms. We also propose fixed-point iterations built on the the perturbation
theory ideas for large $\epsilon$ that are suitable for the large-scale setting. 
The proposed fixed-point iterations initialized by using eigenvalue perturbation theory
converge to the globally rightmost point in the pseudospectrum in a vast
majority of the cases that we experiment with. \\[.05em]



\end{abstract}

\section{Introduction}
The $\epsilon$-pseudospectrum of a matrix $A \in {\mathbb C}^{n\times n}$ for a
prescribed real number $\epsilon > 0$, denoted as $\Lambda_{\epsilon}(A)$, 
is the set consisting of eigenvalues 
of all matrices at a distance not more than $\epsilon$ with respect to the matrix 
2-norm from $A$. Formally, 
\begin{equation}\label{eq:defn_ps}
	\Lambda_{\epsilon}(A)	\;	:=	\;	
		\bigcup_{\Delta \in {\mathbb C}^{n\times n} \text{ s.t. } \| \Delta \|_2 \leq \epsilon}	
				\:	\Lambda(A + \Delta)	\: 
\end{equation}
with $\| \Delta \|_2$ denoting the 2-norm of the matrix $\Delta$; 
see, e.g., \cite{Tre99, TreM05} and references therein.
The real part of the rightmost point in the set $\Lambda_{\epsilon}(A)$, that is
\begin{equation}\label{eq:defn_pspa}
	\alpha_{\epsilon}(A)
		\;	:=	\;
	\max \{ \text{Re}(z)	\;	|	\;	z \in \Lambda_{\epsilon}(A)	\}	\:	,
\end{equation}
is referred as the $\epsilon$-pseudospectral abscissa of $A$. It carries significance
to gain information about the transient behavior of the autonomous system $x'(t) = A x(t)$,
in particular to have an estimate of $\sup_{t > 0}	\:	\| x(t) \|_2$;
see, e.g., \cite[Section 2.1]{Men06}. Moreover, if the system $x'(t) = A x(t)$ is asymptotically stable,
that is if the real parts of the eigenvalues of $A$ are all negative, $\alpha_{\epsilon}(A)$ indicates 
whether the system is robustly stable or not. More specifically, $\alpha_{\epsilon}(A) < 0$
implies that all nearby systems $x'(t) = \widetilde{A} x(t)$ with systems matrices 
$\widetilde{A}$ such that $\| A - \widetilde{A} \|_2 \leq \epsilon$ remain 
asymptotically stable. One of our aims here is to use eigenvalue perturbation theory for the 
estimation of $\alpha_{\epsilon}(A)$. As we illustrate below, if $\epsilon > 0$ is small,
eigenvalue perturbation theory can be put in use to approximate 
$\alpha_{\epsilon}(A)$ quite accurately with little computation. 


The quantity $\alpha_{\epsilon}(A)$ can be posed as the solution of a nonsmooth
and nonconvex optimization problem. Based on this characterization, 
when $A$ is of small size, there is a very reliable 
and globally convergent algorithm to compute $\alpha_{\epsilon}(A)$, namely the criss-cross 
algorithm \cite{BurLO03}, which is no longer practical when $A$ has large size. 
For the large-scale setting, 
there are, for instance, a fixed-point iteration \cite{GugO11}, a gradient-flow based approach \cite{GugL13}
and a subspace approach \cite{KreV14}. 
Even though these approaches are suitable to compute $\alpha_{\epsilon}(A)$
for much larger matrices $A$, 
they suffer from local convergence, i.e., they converge to a locally rightmost point,
which is not necessarily rightmost globally. If $\epsilon > 0$ is not so small,
in this work we still provide estimates for a globally rightmost point in $\Lambda_{\epsilon}(A)$
using eigenvalue perturbation theory. Locally convergent algorithms, 
e.g., the approaches in \cite{GugO11}, \cite{GugL13}, \cite{KreV14}, can possibly 
be initialized with these estimates so that with high probability they converge to a
rightmost point globally. Moreover, we describe a new fixed-point iteration, which,
when initialized with these estimates, converges typically to a 
globally rightmost point in $\Lambda_{\epsilon}(A)$. The new fixed-point iteration 
usually seems to be faster than the existing fixed-point iteration \cite{GugO11}.


In a more general setting, the $\epsilon$-pseudospectrum $\Lambda_{\epsilon}(T)$
of a matrix-valued function
\begin{equation}\label{eq:mat_fun}
	T(\lambda)	\;	=	\;	t_1(\lambda) T_1 + \dots + t_{\kappa}(\lambda) T_{\kappa}
\end{equation}
for given square matrices $T_j \in {\mathbb C}^{n\times n}$
and holomorphic functions $t_j : {\mathcal D} \rightarrow {\mathbb C}$ for $j = 1, \dots , \kappa$
with ${\mathcal D}$ denoting an open subset of ${\mathbb C}$ can be defined similarly. 
The set $\Lambda_{\epsilon}(T)$ still consists of eigenvalues of all matrix-valued functions
at a distance at most $\epsilon$ from $T$. We refrain from a formal definition of 
$\Lambda_{\epsilon}(T)$ for now, but a formal definition is given later in Section \ref{sec:nle_psa},
in particular in (\ref{eq:defn_NEP_pspec}). The $\epsilon$-pseudospectral abscissa 
$\alpha_{\epsilon}(T)$ is the real part of a rightmost point in $\Lambda_{\epsilon}(T)$,
assuming such a point exists. The criss-cross algorithm from the matrix setting \cite{BurLO03} 
generalizes at least for some specific instances of $T$; see in particular \cite[Section 2.3.4]{Men06}
for a criss-cross algorithm for matrix polynomials and \cite{MehM24} for its specialization 
to quadratic matrix polynomials. It is again globally convergent, yet not suitable 
for large problems.
A fixed-point iteration, in essence a generalization of \cite{GugO11}, is proposed to 
compute $\alpha_{\epsilon}(T)$ in \cite{MicG12}, but converges to locally rightmost 
points in $\Lambda_{\epsilon}(T)$, that are not necessarily rightmost globally. 
A subspace approach in the context of a general matrix-valued function
presented in \cite{MeeMMV17} is again prone to local convergence.
The number of iterations of the criss-cross algorithm when applicable, 
and situation with local convergence deteriorate for some important classes
of matrix-valued functions, including matrix polynomials.
We discuss how $\alpha_{\epsilon}(T)$ can be estimated accurately
for small $\epsilon$ using eigenvalue perturbation theory at little cost. 
For larger $\epsilon$, we derive estimates for a globally rightmost
point in $\Lambda_{\epsilon}(T)$, and introduce two fixed-point iterations to compute
$\alpha_{\epsilon}(T)$ based on these eigenvalue perturbation theory arguments.
The fixed-point iterations are simpler than the one in \cite{MicG12}, and appear to be
considerably more efficient than the criss-cross algorithms, at least on quadratic matrix 
polynomials. Even though they are locally convergent, when the fixed-point iterations 
are initialized using eigenvalue perturbation theory, they seem to converge to the globally 
rightmost point in $\Lambda_{\epsilon}(T)$ with high probability.


We start our treatment in the next section with the approximation of the 
$\epsilon$-pseudospectral abscissa using eigenvalue perturbation theory  
for the general setting of a nonlinear eigenvalue problem, which also 
encompasses the setting of a matrix. In Section \ref{sec:fp_mat_val},
partly based on the approximation result deduced via perturbation 
theory, we tailor fixed-point iterations for the $\epsilon$-pseudospectral abscissa 
of a nonlinear eigenvalue problem. Then, in Section \ref{sec:psaA}, we specialize 
into $\alpha_{\epsilon}(A)$ for a matrix $A$. Section \ref{sec:glob_opts} 
delves into the estimation of the globally rightmost points in $\Lambda_{\epsilon}(T)$ 
and $\Lambda_{\epsilon}(A)$. MATLAB implementations of the fixed-point
iterations derived that are initialized based on eigenvalue perturbation theory
are publicly available \cite{AhmM25}. Some details of these implementations are
spelled out in Section \ref{sec:software}. In Section \ref{sec:num_examp},
we report the results of numerical experiments conducted with these implementations,
illustrating the accuracy and efficiency of the proposed approaches based on perturbation theory.

\section{First-order approximation of the pseudospectral abscissa
		of a nonlinear eigenvalue problem}\label{sec:fo_nlevp}
We silently assume throughout that the matrix-valued function $T$ as in (\ref{eq:mat_fun})
is regular, i.e., $\text{det}(T(\lambda))$ is not identically zero. 
Suppose $\mu_0 \in {\mathcal D}$ is a simple eigenvalue of $T$ as in (\ref{eq:mat_fun}), 
and is isolated, i.e., there is a neighborhood of $\mu_0$ 
such that $\mu_0$ is the unique eigenvalue of $T(\lambda)$ in this neighborhood.

A right eigenvector $x \in {\mathbb C}^n \backslash \{ 0 \}$  
corresponding to $\mu_0$  satisfies $T(\mu_0) x = 0$, while a corresponding left eigenvector 
$y \in {\mathbb C}^n \backslash \{ 0 \}$ satisfies $y^\ast T(\mu_0) = 0$.


We first aim to come up with an upper bound for the real part of any eigenvalue
originating from $\mu_0$ when $T$ is subject to perturbations with norms 
not exceeding $\epsilon$, which is a prescribed positive real number. 

Formally, suppose the coefficients $T_j$ are subject to perturbations  
$\Delta T_j \in {\mathbb C}^{n\times n}$, $j = 1, \dots , \kappa$.
For given nonnegative real numbers $w_j$ representing weights on the 
perturbations of $T_j$, we consider 
\[
	\Delta T(\lambda)		\; := \;		t_1(\lambda) w_1 \Delta T_1 
						+ \dots	+ 	t_{\kappa}(\lambda) w_{\kappa} \Delta T_{\kappa}	\: ,
\]
which we refer as $\Delta T(\lambda)$ corresponding to 
$\Delta := (\Delta T_1, \dots , \Delta T_{\kappa} )$ in some occasions in the subsequent discussions.
A silent assumption regarding the weights and scalar functions $t_j(\lambda)$, $j = 1, \dots , \kappa$
that we keep throughout is that
$\sum_{j=1}^\kappa w_j | t_j(\lambda) | \neq 0$ for all $\lambda \in {\mathcal D}$.
The next theorem, a corollary of Rouch\'{e}'s theorem \cite[Theorem 10.43(b)]{Rud87}, states that there is a unique 
simple eigenvalue of any matrix-valued function sufficiently close to $T$ corresponding to the
eigenvalue $\mu_0$ of $T$. In the theorem and its proof, 
$B_r(z_0) := \{ z \in {\mathbb C} \; | \; |z - z_0| < r \}$,
$\overline{B}_r(z_0) := \{ z \in {\mathbb C} \; | \; |z - z_0| \leq r \}$
denote the open ball, closed ball, respectively, of radius $r$ centered at $z_0 \in {\mathbb C}$ 
in the complex plane. The notation $\partial B_r(z_0)$ is used for the boundary of $B_r(z_0)$.  
\begin{theorem}\label{thm:eig_well_posed}
	Suppose $\mu_0 \in {\mathcal D}$ is a simple and isolated eigenvalue of $T$.
	There is an open interval $U \subset {\mathbb R}$ containing 0
	and a real number $r > 0$ satisfying $B_r(\mu_0) \subset {\mathcal D}$
	such that for every $\Delta = ( \Delta T_1 , \dots , \Delta T_{\kappa})$
	with
	$
	\left\|
		\left[
			\begin{array}{ccc}
				\Delta T_1		&	\dots		&	\Delta T_{\kappa}
			\end{array}
		\right]	
	\right\|_2	\leq	1
	$
	and every $\eta \in U$,
	the matrix-valued function $(T + \eta \Delta T)(\lambda)$ 
	for $\Delta T(\lambda)$ corresponding to $\Delta$ has only one
	eigenvalue in $B_r(\mu_0)$, which is simple.
\end{theorem}
\begin{proof}
For any prescribed $\Delta = ( \Delta T_1 , \dots , \Delta T_{\kappa})$
	with
	$
	\left\|
		\left[
			\begin{array}{ccc}
				\Delta T_1		&	\dots		&	\Delta T_{\kappa}
			\end{array}
		\right]	
	\right\|_2	\leq	1
	$,
let $g_{\Delta}(\eta,\lambda) := \text{det}((T + \eta \Delta T)(\lambda))$. Since $\mu_0$
is an isolated eigenvalue of $T$ and ${\mathcal D}$ is open, there is a real number $r > 0$ such
that $\overline{B}_r(\mu_0) \subset {\mathcal D}$ and $\mu_0$ is the only
eigenvalue of $T$ in $\overline{B}_r(\mu_0)$. Hence, the only root of $g_{\Delta}(0,\lambda)$ 
in $\overline{B}_r(\mu_0)$ is $\mu_0$, and $\mu_0$ is a simple root (i.e., with multiplicity one). 
Now let $\rho := \min \big\{ |g_{\Delta}(0,\lambda)| \; | \; \lambda \in \partial B_r(\mu_0) \big\} \, > \, 0$.
By continuity of $g_{\Delta}(\eta, \lambda)$ with respect to $\eta$, for every $\lambda \in \partial B_r(\mu_0)$,
there is an open interval $U_{\lambda, \Delta}$ containing 0
such that $|g_{\Delta}(\eta, \lambda) - g_{\Delta}(0,\lambda) | < \rho$
for all $\eta \in U_{\lambda, \Delta}$. Indeed, by the compactness of $\partial B_r(\mu_0)$,
the intersection $U_{\Delta} := \cap_{\lambda \in \partial B_r(\mu_0)} U_{\lambda, \Delta}$
is also an open interval containing 0. It follows that $|g_{\Delta}(\eta, \lambda) - g_{\Delta}(0,\lambda) | < \rho$
for all $\eta \in U_{\Delta}$ and all $\lambda \in \partial B_r(\mu_0)$.
Consequently, $|g_{\Delta}(0,\lambda)| > |g_{\Delta}(\eta, \lambda) - g_{\Delta}(0,\lambda) |$ 
for all $\lambda \in \partial B_r(\mu_0)$ for all $\eta \in U_{\Delta}$. 
By Rouch\'{e}'s theorem \cite[Theorem 10.43(b)]{Rud87}, 
the following assertion holds for every $\eta \in U_{\Delta}$:
there is only one root of $g_{\Delta}(\eta, \lambda)$ in $B_r(\mu_0)$, which is simple.

Finally, by the compactness of 
${\mathcal S} := \{ ( \Delta T_1 , \dots , \Delta T_{\kappa})	\;	|	\;	
	\left\|
		\left[
			\begin{array}{ccc}
				\Delta T_1		&	\dots		&	\Delta T_{\kappa}
			\end{array}
		\right]	
	\right\|_2	\leq	1	\}$,
the intersection $U := \cap_{\Delta \in {\mathcal S}} U_\Delta$ is also an open interval
containing 0. The function $g_\Delta(\nu,\lambda)$ has only one root, that is also simple,
equivalently $(T + \eta \Delta T)(\lambda)$ for $\Delta T(\lambda)$ corresponding to $\Delta$
has only one eigenvalue in $B_r(\mu_0)$, that is also simple, for every $\eta \in U$ and 
every $\Delta \in {\mathcal S}$.
\end{proof}
Throughout, we assume that $\epsilon > 0$ is sufficiently small so that
$[0,\epsilon] \subset U$ for $U$ as in Theorem \ref{thm:eig_well_posed}. Moreover,
for $U$ as in Theorem \ref{thm:eig_well_posed} and
for a particular $\Delta := ( \Delta T_1 , \dots , \Delta T_{\kappa})$ with
$
	\left\|
		\left[
			\begin{array}{ccc}
				\Delta T_1		&	\dots		&	\Delta T_{\kappa}
			\end{array}
		\right]	
	\right\|_2	\leq	1
$,
we define $\mu : U \rightarrow {\mathcal D}$ as follows: $\mu(\eta)$
for $\eta \in U$ is the unique simple eigenvalue of $(T + \eta \Delta T)(\lambda)$
for $\Delta T(\lambda)$ corresponding to $\Delta$
in $B_r(\mu_0)$ asserted in Theorem \ref{thm:eig_well_posed}.
By applying the analytic implicit function theorem to 
$g(\eta,\lambda) = \text{det}((T + \eta \Delta T)(\lambda))$, it follows that
the function $\mu(\eta)$ is analytic, i.e., the real and imaginary parts
of $\mu(\eta)$ are real-analytic functions.
Clearly, $\mu$ is continuous on $[0,\epsilon]$,
analytic on $(0, \epsilon)$ and satisfies $\mu(0) = \mu_0$.
To make the dependence of the eigenvalue function $\mu(\eta)$ 
on $\Delta$ explicit, 
we may also write it as $\mu(\eta ; \Delta)$. In addition, sometimes, we may 
also want to make the dependence of $\mu(\eta)$ on $\mu_0$ and $T$ explicit, in which
case we write $\mu(\eta ; \Delta , \mu_0)$ or $\mu(\eta ; \Delta , \mu_0, T)$.


Here, for a prescribed positive real number $\epsilon$ supposedly small 
(in particular such that $[0,\epsilon] \subset U$), we would like to estimate 
\begin{equation}\label{eq:defn_Reps}
\begin{split}
	&
	{\mathcal R}(\epsilon; \mu_0)
			\;	:=	\;	
	\max_{\Delta \in {\mathcal S}}	\:\:	\text{Re} \left\{ \mu(\epsilon ; \Delta) \right\}		,
	\\[.3em]
	&
	\text{where}	\quad
	{\mathcal S}	:=	
	\{
		 (\Delta T_1 , \dots , \Delta T_{\kappa})	\;\:	|	\;\:
		 \left\|
		\left[
			\begin{array}{ccc}
				\Delta T_1		&	\dots		&	\Delta T_{\kappa}
			\end{array}
		\right]	
		\right\|_2	\leq	1	
	\}	\:	,
\end{split}	
\end{equation}
i.e., the real part of the rightmost eigenvalue that can be attained
from the eigenvalue $\mu_0$ of $T(\lambda)$ by applying perturbations with norms not
exceeding $\epsilon$.


For a given $\Delta \in {\mathcal S}$, the function
\begin{equation}\label{eq:defn_L}
	{\mathcal L}_\Delta(\eta) \;		:=	\;	 \text{Re} \left\{ \mu(\eta ; \Delta) \right\}
\end{equation}
is continuous on $[0,\epsilon]$ and real-analytic on $(0,\epsilon)$, so Taylor's theorem implies
\begin{equation}\label{eq:L_expand}
	{\mathcal L}_\Delta(\epsilon)
		\;	=	\;
	{\mathcal L}_\Delta(0)	+
	{\mathcal L}'_\Delta(0) \epsilon		+
	{\mathcal O}(\epsilon^2)		
		\;	=	\;
	\text{Re}(\mu_0)	+
	{\mathcal L}'_\Delta(0) \epsilon		+
	{\mathcal O}(\epsilon^2).
\end{equation}
Consequently, it follows from the definition of ${\mathcal R}(\epsilon; \mu_0)$
in (\ref{eq:defn_Reps}) that
\begin{equation}\label{eq:defR}
	{\mathcal R}(\epsilon; \mu_0)
		\;	=	\;
	\text{Re}(\mu_0)		+
	\epsilon	\left\{		\max_{\Delta \in {\mathcal S}}	\: {\mathcal L}'_\Delta(0)	\right\}		+
	{\mathcal O}(\epsilon^2).
\end{equation}

\subsection{Expression for the derivative of $\mu(\eta ; \Delta)$ at $\eta = 0$}\label{sec:der_mu}
For a prescribed $\Delta \in {\mathcal S}$ and the corresponding $\Delta T(\lambda)$,
there is an analytic vector-valued function $x(\eta; \Delta)$ such that
$\| x(\eta; \Delta) \|_2 = 1$ and
\begin{equation}\label{eq:eig_eqn}
		\big\{  (T + \eta \Delta T)(\mu(\eta ; \Delta))  \big\} \, x(\eta ; \Delta) 	\;	=	\;	0
\end{equation}
for all $\eta \in U$ \cite[pages 32-33]{Rel69}. Here and elsewhere, $\| v \|_2$ denotes
the Euclidean norm for $v \in {\mathbb C}^n$. 
The vector $x(\eta; \Delta)$ is a unit right eigenvector 
of $(T + \eta \Delta T)(\lambda)$ corresponding to its eigenvalue $\mu(\eta ; \Delta)$.
In the subsequent derivations, we simply write $\mu'$, $x'$ for the derivatives
$\mu'(0 ; \Delta)$, $x'(0 ; \Delta)$. Without loss of generality, we let $x = x(0; \Delta)$,
a unit right eigenvector of $T(\lambda)$ corresponding to $\mu_0$.
To derive an expression for the
derivative of $\mu(\eta ; \Delta)$ at $\eta = 0$, we differentiate both sides of (\ref{eq:eig_eqn})
at $\eta = 0$, which yields
\begin{equation}\label{eq:eig_eqn2}
	\mu' T'(\mu_0) x	+	\Delta T(\mu_0) x	+	T(\mu_0) x'	\;	=	\;	0  \:  .
\end{equation}
Recall that $y$ denotes a left eigenvector of $T(\lambda)$ corresponding to $\mu_0$
that satisfies $y^\ast T(\mu_0) = 0$. By multiplying (\ref{eq:eig_eqn2}) by $y^\ast$ from left, 
exploiting also $y^\ast T(\mu_0) = 0$, we deduce
\begin{equation}\label{eq:der_mu}
	\mu'		\;	=	\;		- \frac{ y^\ast \Delta T(\mu_0) x}{y^\ast T'(\mu_0) x}	\:	,
\end{equation}
where $y^\ast T'(\mu_0) x \neq 0$ due to the assumption that $\mu_0$ is
a simple eigenvalue of $T(\lambda)$. This also gives rise to
\begin{equation}\label{eq:derL}
	{\mathcal L}'_\Delta(0)
			\;	=	\;
	\text{Re} \left\{ \mu'(0; \Delta) \right\}
			\;	=	\;
	- \text{Re}	\left\{	   	\frac{ y^\ast \Delta T(\mu_0) x}{y^\ast T'(\mu_0) x}	\right\}.
\end{equation}

\subsection{Maximizing ${\mathcal L}'_\Delta(0)$ over $\Delta \in {\mathcal S}$}
For the estimation of ${\mathcal R}(\epsilon; \mu_0)$ using the expansion in (\ref{eq:defR}),
it suffices to maximize ${\mathcal L}'_\Delta(0)$ over $\Delta \in {\mathcal S}$,
which we carry out in this subsection by exploiting the formula in (\ref{eq:derL}).  
Our derivation here has some similarities with those in \cite{PapH93, GenVD99, TisH01}
that concern computable characterizations of complex stability radii and pseudospectra
of matrix polynomials, but the context here is different than those works.  
Without loss of generality, we assume that the right and left eigenvectors 
$x$ and $y$ in (\ref{eq:derL}) are unit, i.e., $\| x \|_2 = \| y \|_2 = 1$. Moreover, we can 
assume without loss of generality that $x$ is such that $y^\ast T'(\mu_0) x$ is real 
and negative (i.e., supposing $y^\ast T'(\mu_0) x = \rho e^{{\rm i} \theta}$ for some 
positive real number $\rho$ and $\theta \in (0,2\pi)$, we may then replace
$x$ with $\widetilde{x} = -x e^{-{\rm i}\theta}$ so that $y^\ast T'(\mu_0) \widetilde{x} = -\rho$).
It is then evident from (\ref{eq:derL}) that
\begin{equation*}
\begin{split}
	{\mathcal L}'_\Delta(0)
		\;	&  \leq	\;
	\frac{1}{| y^\ast T'(\mu_0) x |} \| \Delta T(\mu_0) \|_2	\\
		\;	&	=	\;
	\frac{1}{| y^\ast T'(\mu_0) x |} 
	\left\|
		\left[
			\begin{array}{ccc}
				\Delta T_1		&	\dots		&	\Delta T_\kappa
			\end{array}
		\right]
		\left[
			\begin{array}{c}
				w_1 \, t_1(\mu_0) I		\\
					\vdots		\\
				w_\kappa  \, t_\kappa(\mu_0) I
			\end{array}
		\right]	
	\right\|_2	\\[.4em]
			&	\leq	\;
	\frac{1}{| y^\ast T'(\mu_0) x |} 
	\left\|
		\left[
			\begin{array}{ccc}
				\Delta T_1		&	\dots		&	\Delta T_\kappa
			\end{array}
		\right]
	\right\|_2
	\sqrt{w_1^2 | t_1(\mu_0) |^2 +  \dots  + w_\kappa^2 | t_\kappa(\mu_0) |^2}			\\[.4em]
			&	\leq	\;
	\frac{1}{| y^\ast T'(\mu_0) x |} 
	\sqrt{w_1^2 | t_1(\mu_0) |^2 +  \dots  + w_\kappa^2 | t_\kappa(\mu_0) |^2}
\end{split}
\end{equation*}
for all $\Delta = ( \Delta T_1 , \dots , \Delta T_\kappa ) \in {\mathcal S}$. Let us consider the specific
perturbations
\begin{equation}\label{eq:opt_perturbation}
	\underline{\Delta T}_{\, j}
			\;	=	\;
	\frac{w_j \, \overline{t_j(\mu_0)} y x^\ast}{\sqrt{w_1^2 | t_1(\mu_0) |^2 +  \dots  + w_\kappa^2 | t_\kappa(\mu_0) |^2}}	\;	,	\;\;\;	j	=	1,	\dots , \kappa \: ,
\end{equation}
which satisfies
\begin{equation*}
\begin{split}
	&
		\left\|
		\left[
			\begin{array}{ccc}
				\underline{\Delta T}_{\, 1}	
							&	\dots		&	
				\underline{\Delta T}_{\, \kappa}
			\end{array}
		\right]
		\right\|_2	=	1	\;	,		\quad	\text{and}		\\[.45em]
	&
	\;\;
	\underline{\Delta T}(\mu_0)	\;  := \;	
	t_1(\mu_0) w_1 \underline{\Delta T}_{\, 1}	
			+	\dots		+	t_\kappa(\mu_0) w_\kappa \underline{\Delta T}_{\, \kappa}	\\
	&
	\phantom{\;\;
	\underline{\Delta T}(\mu_0)  :}
				= \;
	\left\{  \sqrt{w_1^2 | t_1(\mu_0) |^2 +  \dots  + w_\kappa^2 | t_\kappa(\mu_0) |^2}  \right\} \, y x^\ast	\:  .
\end{split}
\end{equation*}
For this perturbation, 
setting $\underline{\Delta} 
		:= ( \underline{\Delta T}_{\, 1} , \dots , \underline{\Delta T}_{\, \kappa}) \in {\mathcal S}$,
we have
\begin{equation*}
\begin{split}
	{\mathcal L}'_{\underline{\Delta}}(0)
			\,	=	\,
	- \text{Re}	\left\{	   	\frac{ y^\ast \underline{\Delta T} (\mu_0) x}{y^\ast T'(\mu_0) x}	\right\}
			\,	&	=	\,
	-\left\{ \sqrt{w_1^2 | t_1(\mu_0) |^2 +  \dots  + w_\kappa^2 | t_\kappa(\mu_0) |^2} \right\} \text{Re} \left\{ \frac{1}{y^\ast T'(\mu_0) x} \right\}	\\[.2em]
		&	=	\,
	\frac{1}{| y^\ast T'(\mu_0) x |}  \sqrt{w_1^2 | t_1(\mu_0) |^2 +  \dots  + w_\kappa^2 | t_\kappa(\mu_0) |^2}   \:	,
\end{split}
\end{equation*}
where in the last equality we employ that $y^\ast T'(\mu_0) x$ is real and negative
so that $-\text{Re}\{y^\ast T'(\mu_0) x\} = | y^\ast T'(\mu_0) x |$.


Hence, we arrive at the following result.
\begin{theorem}\label{thm:max_LDelta}
Suppose $\mu_0 \in {\mathcal D}$ is a simple and isolated eigenvalue of $T(\lambda)$.
Then the following holds:
\vskip -.5ex
\begin{equation}\label{eq:exp_der}
	\max_{\Delta \in {\mathcal S}}	\: {\mathcal L}'_\Delta(0)	
			\;	=	\;
	\frac{1}{| y^\ast T'(\mu_0) x |}   
	\sqrt{w_1^2 | t_1(\mu_0) |^2 +  \dots  + w_\kappa^2 | t_\kappa(\mu_0) |^2} 	\:	.
\end{equation}
\vskip 1ex
\noindent
Moreover, 
$\max_{\Delta \in {\mathcal S}}	\: {\mathcal L}'_\Delta(0) = 
				\: {\mathcal L}'_{\underline{\Delta}}(0)$, where 
$\underline{\Delta} = ( \underline{\Delta T}_{\, 1} , 
	\dots , \underline{\Delta T}_{\, \kappa}) \in {\mathcal S}$	
for $\underline{\Delta T}_{\, 1} , \dots ,$ $\underline{\Delta T}_{\, \kappa}$ as in (\ref{eq:opt_perturbation})
with $x, y$ denoting unit right, left unit eigenvectors of $T$ corresponding to $\mu_0$ normalized
so that $y^\ast T'(\mu_0) x$ is real and negative.
\end{theorem}
By combining (\ref{eq:exp_der}) with (\ref{eq:defR}), we also
deduce the following result, which is helpful for the estimation of ${\mathcal R}(\epsilon; \mu_0)$.
\begin{corollary}
Suppose that $\mu_0 \in {\mathcal D}$ is a simple, isolated eigenvalue of $T$, and
$\epsilon > 0$ is sufficiently small so that $[0, \epsilon] \subset U$ 
for $U$ as in Theorem \ref{thm:eig_well_posed}. Then, we have
\begin{equation}\label{eq:R_approx}
	{\mathcal R}(\epsilon; \mu_0)
		\;	=	\;
	\mathrm{Re}(\mu_0)		+
	\epsilon	\left\{	  \frac{ \sqrt{w_1^2 | t_1(\mu_0) |^2 +  \dots  + w_\kappa^2 | t_\kappa(\mu_0) |^2} }{| y^\ast T'(\mu_0) x |} 	\right\}		+
	{\mathcal O}(\epsilon^2)	\:	
\end{equation}
with $x$, $y$ denoting a unit right eigenvector, a unit left eigenvector, respectively, of $T$
corresponding to the eigenvalue $\mu_0$.
\end{corollary}

\subsection{Estimation of the $\epsilon$-pseudospectral abscissa}\label{sec:nle_psa}
For a prescribed real number $\epsilon > 0$, 
we define the $\epsilon$-pseudospectrum of $T$ by
\begin{equation}\label{eq:defn_NEP_pspec}
 \begin{split}
	&
	\Lambda_{\epsilon}(T)
		\;	:=	\;
	\bigcup_{\Delta T \in {\mathcal P}_\epsilon} \Lambda(T + \Delta T)	\:	,	\quad \text{where}		 \\
	&
	{\mathcal P}_\epsilon
		\;	=	\;
	\left\{
			\Delta T(\lambda)  =  t_1(\lambda) w_1 \Delta T_1	
						+	\dots 	+	t_\kappa(\lambda) w_\kappa \Delta T_\kappa
			\;\:	|	\;\:	(\Delta T_1 , \dots , \Delta T_{\kappa})	\in {\mathcal S}_\epsilon
	\right\}		,
 \end{split}
\end{equation}
the notation 
$\Lambda(T + \Delta T)$ is for the set of finite eigenvalues of the matrix-valued
function $(T + \Delta T)(\lambda)$ and
$
	{\mathcal S}_{\epsilon}	:=	
	\{
		 (\Delta T_1 , \dots , \Delta T_{\kappa})	\;\:	|	\;\:
		 \left\|
		\left[
			\begin{array}{ccc}
				\Delta T_1		&	\dots		&	\Delta T_{\kappa}
			\end{array}
		\right]	
		\right\|_2	\leq	\epsilon	
	\}	\:	
$.
There are slight variations of this definition of the $\epsilon$-pseudospectrum 
for a matrix-valued function considered in the literature \cite{TisH01, TisM01, MicGWN06, MicG12, MeeMMV17}. 
They differ in a minor way only by the way ${\mathcal P}_\epsilon$ in (\ref{eq:defn_NEP_pspec}) 
is defined, in particular the choice of the norm in the definition of ${\mathcal S}_\epsilon$.
The derivations throughout this section can be adapted to deal with such slight variations 
in a straightforward manner
by modifying ${\mathcal S}$ in the definition of ${\mathcal R}(\epsilon; \mu_0)$ in (\ref{eq:defn_Reps})
accordingly.


Our derivations in this subsection rely on the following assumption.
\begin{assumption}\label{ass:psa_nlevp}
Throughout, we assume that the matrix-valued function $T$ and $\epsilon$ are such that the
suprema of
\begin{center}
\textbf{(i)}  
	 $\{ \mathrm{Re}(z)	\;	|	\;	z \in \Lambda_{\epsilon}(T)	\} \;$ 	$\;\;$		and	$\;\;\,$	
\textbf{(ii)} 
	$\, \{ {\mathcal R}(\epsilon ; \mu_0) \; | \; \mu_0 \in \Lambda(T)	\}$
\end{center}
are attained.
\end{assumption}
\noindent
We remark that part (ii) of Assumption \ref{ass:psa_nlevp} is trivially satisfied if $\Lambda(T)$ is finite, 
which for instance is the case for any regular matrix polynomial.
The $\epsilon$-pseudospectral abscissa $\alpha_{\epsilon}(T)$ of $T$ 
is the real part of the rightmost point in $\Lambda_{\epsilon}(T)$, i.e.,
\[
	\alpha_{\epsilon}(T)
		\;	:=	\;
	\max \: \{ \text{Re}(z)	\;	|	\;	z \in \Lambda_{\epsilon}(T)	\}	\:	.
\]
The next result suggests a way to estimate $\alpha_{\epsilon}(T)$ using
eigenvalue perturbation theory.
\begin{theorem}\label{thm:fo_estimate}
Suppose that Assumption \ref{ass:psa_nlevp} holds, every $\mu_0 \in \Lambda(T)$ is simple
and isolated, and $\epsilon$ is sufficiently small so that $[0, \epsilon] \subset U$ for $U$
as in Theorem \ref{thm:eig_well_posed} for every $\mu_0 \in \Lambda(T)$. 
Then
\begin{equation}\label{eq:main_char_psa_nle}
\begin{split}
	\alpha_{\epsilon}(T) \; & = \;  \max \{ {\mathcal R}(\epsilon ; \mu_0) \; | \; \mu_0 \in \Lambda(T)	\} \\
		 	&	=	\; 
						\max_{\mu_0 \in \Lambda(T)}
	\left\{
		\mathrm{Re}(\mu_0)		+
		\epsilon	\left\{	  
					\frac{ \sqrt{w_1^2 | t_1(\mu_0) |^2 +  \dots  + w_\kappa^2 | t_\kappa(\mu_0) |^2}}
						{| y_{\mu_0}^\ast T'(\mu_0) x_{\mu_0} |} 	
				\right\}		
	\right\}
							\:	+	\:
						{\mathcal O}(\epsilon^2), \\[.2em]
\end{split}
\end{equation}
with $x_{\mu_0}$, $y_{\mu_0}$ denoting a unit right eigenvector, a unit left eigenvector, respectively,
corresponding to the eigenvalue $\mu_0$ of $T$.
\end{theorem}
\begin{proof}
We start by proving the first equality in (\ref{eq:main_char_psa_nle}).
Let ${\mathcal R}(\epsilon) :=  \max \{ {\mathcal R}(\epsilon ; \mu_0) \; | \; \mu_0 \in \Lambda(T)	\}$.
We first show $\alpha_{\epsilon}(T) \leq {\mathcal R}(\epsilon)$. To this end, let
$z_\ast \in {\mathbb C}$ be such that (1) $z_\ast \in \Lambda_{\epsilon}(T)$ and $\text{Re}(z_\ast) = \alpha_{\epsilon}(T)$. 
Moreover, by the definition of $\Lambda_{\epsilon}(T)$, there must be 
$(\underline{\Delta T}_{\, 1} , \dots , \underline{\Delta T}_{\, \kappa}) \in \mathcal{S}_\epsilon$ 
such that $z_\ast \in \Lambda_{\epsilon}(T + \underline{\Delta T})$ for 
$\underline{\Delta T}(\lambda) = t_1(\lambda) w_1 \underline{\Delta T}_{\, 1} 
						+ 	\dots 	+ 
				    t_\kappa w_{\kappa} \underline{\Delta T}_{\, \kappa}$.
Let 
\[
	\underline{\widehat{\Delta}}  \; = \;
	(\underline{\widehat{\Delta T}}_{\, 1} , \dots , \underline{\widehat{\Delta T}}_{\, \kappa})
		\; := \;
	\frac{1}{  \epsilon }
	(\underline{\Delta T}_{\, 1} , \dots , \underline{\Delta T}_{\, \kappa})
\]
so that $\underline{\widehat{\Delta}} \in {\mathcal S}$, as 
$(\underline{\Delta T}_{\, 1} , \dots , \underline{\Delta T}_{\, \kappa}) \in \mathcal{S}_\epsilon$ 
or equivalently $\left\|
		\left[
			\begin{array}{ccc}
				\underline{\Delta T}_1		&	\dots		&	\underline{\Delta T}_{\kappa}
			\end{array}
		\right]	
		\right\|_2 \leq \epsilon$. 
There must be an analytic eigenvalue function
$\mu(\eta; \underline{\widehat{\Delta}})$ as defined in the opening of Section \ref{sec:fo_nlevp} 
corresponding to an eigenvalue of 
$(T + \eta \underline{\widehat{\Delta T}})(\lambda)$ with 
$\underline{\widehat{\Delta T}}(\lambda) = t_1(\lambda) w_1 \widehat{\underline{\Delta T}}_{\, 1} + \dots + w_\kappa \widehat{\underline{\Delta T}}_{\, \kappa}$
such that $\mu(\epsilon; \widehat{\underline{\Delta}}) = z_\ast$. 
Letting $\mu_\ast := \mu(0 ; \widehat{\underline{\Delta}})$,
which must be an eigenvalue of the unperturbed polynomial $T(\lambda)$, we have
\begin{equation*}
\begin{split}
	{\mathcal R}(\epsilon)
		\;	&	\geq		\;
	{\mathcal R}(\epsilon ; \mu_\ast)
		=
	\max_{\Delta \in {\mathcal S}}	\:\:	\text{Re} \left\{ \mu(\epsilon ; \Delta , \mu_\ast) \right\}		\\
		\;	&	\geq		\;
	\text{Re} \left\{ \mu(\epsilon ; \widehat{\underline{\Delta}} , \mu_\ast) \right\}
		\;	=	\;
	\text{Re}(z_\ast)
		\;	=	\;
	\alpha_{\epsilon}(T)	\:	
\end{split}
\end{equation*}
as claimed.


To prove $\alpha_{\epsilon}(T) \geq {\mathcal R}(\epsilon)$, let $\mu_\ast \in \Lambda(T)$ be
such that ${\mathcal R}(\epsilon) = {\mathcal R}(\epsilon ; \mu_\ast)$. 
By the compactness of ${\mathcal S}$, there must be $\underline{\Delta} \in {\mathcal S}$ such that 
${\mathcal R}(\epsilon; \mu_\ast) = \text{Re}\{ \mu(\epsilon ; \underline{\Delta} , \mu_\ast) \}$.
But then, letting $\underline{\Delta} = (\underline{\Delta T}_{\, 1} , \dots , \underline{\Delta T}_{\, \kappa})$ and 
$\underline{\Delta T}(\lambda) = 
t_1(\lambda) w_1 \underline{\Delta T}_{\, 1} + \dots + t_{\kappa}(\lambda) w_\kappa \underline{\Delta T}_\kappa$,
we have $\epsilon (\underline{\Delta T}_{\, 1} , \dots , \underline{\Delta T}_{\, \kappa}) \in {\mathcal S}_\epsilon$, 
so, by the definition of $\Lambda_{\epsilon}(T)$, it follows that 
\begin{equation*}
\begin{split}
	\Lambda(T + \epsilon \underline{\Delta T})
		\subseteq
	\Lambda_{\epsilon}(T)
		\;\;	&	\Longrightarrow		\;\;
	\mu(\epsilon ; \underline{\Delta}, \mu_\ast)	\in \Lambda_{\epsilon}(T)		\\
		\;\;	&	\Longrightarrow		\;\;
	{\mathcal R}(\epsilon) = \text{Re}\{ \mu(\epsilon ; \underline{\Delta} , \mu_\ast) \}
		\;	\leq	\;		\alpha_{\epsilon}(T)	\:	,
\end{split}		
\end{equation*}
where the first implication is due to 
$\mu(\epsilon ; \underline{\Delta} , \mu_\ast) \in \Lambda(T + \epsilon \underline{\Delta T})$
(i.e.,  $\mu(\epsilon ; \underline{\Delta} , \mu_\ast)$ is an eigenvalue of 
$(T + \epsilon \underline{\Delta T})(\lambda)$), completing the proof of 
the first equality in (\ref{eq:main_char_psa_nle}).


The second equality in (\ref{eq:main_char_psa_nle}) follows immediately from (\ref{eq:R_approx}).
\end{proof}

\begin{example}\label{ex:damp1}
\rm{For a regular quadratic matrix polynomial $P(\lambda) = \lambda^2 M + \lambda C + K$
with the prescribed perturbation weights $w_m := w_1$, $w_c := w_2$, $w_k := w_3$
in the definition of $\Lambda_{\epsilon}(P)$, Theorem \ref{thm:fo_estimate} yields
(under the simplicity assumption on the eigenvalues of $(P + \Delta P)(\lambda)$
for all perturbations $\Delta P \in {\mathcal P}_{\epsilon}$, and under the 
assumption that the rightmost point in $\Lambda_{\epsilon}(P)$ is attained)
\[
	\alpha_{\epsilon}(P)
		\;	=	\;
	\max_{\mu_0 \in \Lambda(P)}
						\left\{
							\mathrm{Re}(\mu_0)		+
	\epsilon	
	\left\{	  \frac{ \sqrt{w_m^2 |\mu_0|^4 + w_c^2 |\mu_0|^2 + w_k^2}}{\: | y_{\mu_0}^\ast P'(\mu_0) x_{\mu_0} |} 	\right\}	
						\right\}
								+	{\mathcal O}(\epsilon^2)	\:	.
\] 

Let us specifically consider the parameter-dependent quadratic matrix polynomial example 
\begin{equation}\label{eq:damping_ex_20by20}
\begin{split}
	&	P(\lambda ; \nu)	=	\lambda^2 M + \lambda C(\nu) + K	\\[.3em]
	&	\text{with}	\;\;	
	M
		=
	\mathrm{diag}(     
	1,\dots, 20)	\:	,	\;\;
	C(\nu)
		=
	C_{\mathrm{int}}	\,	+	\,	\nu e_2 e_2^T	\:	,	\;\;
	K
		=
	\mathrm{tridiag}(-25,50,-25) \, ,
\end{split}
\end{equation}
$C_{\mathrm{int}} \, = \, 2\xi  M^{1/2} \sqrt{M^{-1/2} K M^{-1/2}} M^{1/2}$, 
$\xi = 0.005$ from \cite[Example 5.2]{MehM24}, and
with the weights $w_m = w_c = w_k = 1$.
The example concerns a mass-spring-damper system consisting of twenty masses
tied together through springs and with a damper on the second mass. The parameter $\nu$ 
is a nonnegative real number corresponding to the viscosity of the damper.
Let 
\[
	{\mathcal R}_{\epsilon}(P(\cdot ; \nu)) := \max_{\mu_0 \in \Lambda(P(\cdot ; \nu))}
						\left\{
							\mathrm{Re}(\mu_0)		+
		\epsilon	\left\{	  \frac{ \sqrt{|\mu_0|^4 + |\mu_0|^2 + 1}}{\: | y_{\mu_0}^\ast P'(\mu_0 ; \nu) x_{\mu_0} |} 	\right\}	
						\right\}
\]
be the estimate for the $\epsilon$-pseudospectral abscissa $\alpha_{\epsilon}(P(\cdot ; \nu))$
in (\ref{eq:main_char_psa_nle}) in Theorem \ref{thm:fo_estimate}, which must satisfy
$\alpha_{\epsilon}(P(\cdot ; \nu)) = {\mathcal R}_{\epsilon}(P(\cdot ; \nu)) + {\mathcal O}(\epsilon^2)$
for every $\nu$, under simplicity and attainment of rightmost point in $\Lambda_{\epsilon}(P(\cdot ; \nu))$
assumptions.


The plots of $\alpha_{\epsilon}(P(\cdot ; \nu))$ and ${\mathcal R}_{\epsilon}(P(\cdot ; \nu))$
for $\epsilon = 0.1$ as a function of $\nu \in [0, 100]$ are shown on the left in Figure \ref{fig:error}. 
It appears that ${\mathcal R}_{\epsilon}(P(\cdot ; \nu))$ approximates $\alpha_{\epsilon}(P(\cdot ; \nu))$
extremely well. Also, looking at the plots on the right in Figure \ref{fig:error},
that is, the plots of the errors 
$| \alpha_{\epsilon}(P(\cdot ; \nu)) - {\mathcal R}_{\epsilon}(P(\cdot ; \nu)) |$
as a function of $\nu \in [0,100]$ for $\epsilon = 0.025, 0.05, 0.1, 0.2$, the error
$| \alpha_{\epsilon}(P(\cdot ; \nu)) - {\mathcal R}_{\epsilon}(P(\cdot ; \nu)) |$ appears to be 
proportional to $\epsilon^2$ as asserted by Theorem \ref{thm:fo_estimate}.
\begin{figure}
	\begin{tabular}{cc}
		\hskip -3.5ex
			\includegraphics[width = .525\textwidth]{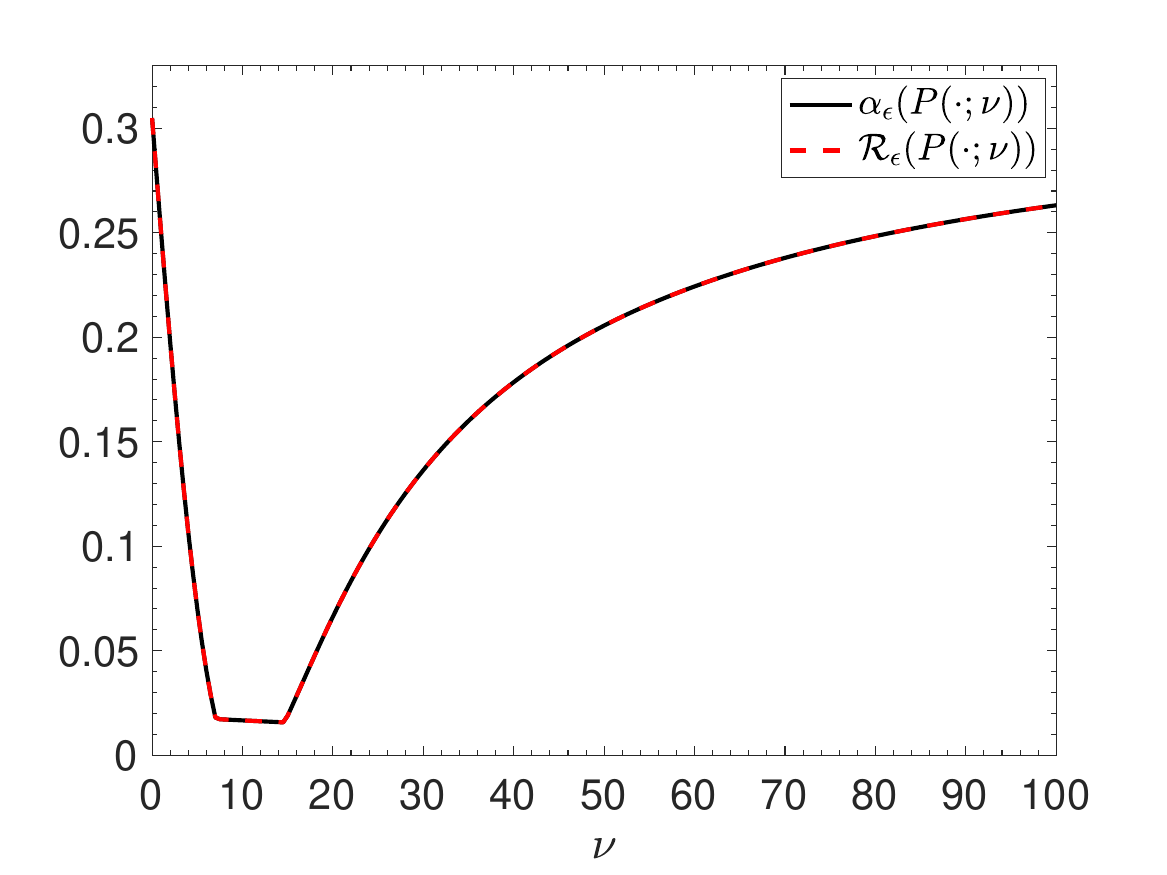} & 
			\hskip -1.5ex
			\includegraphics[width = .52\textwidth]{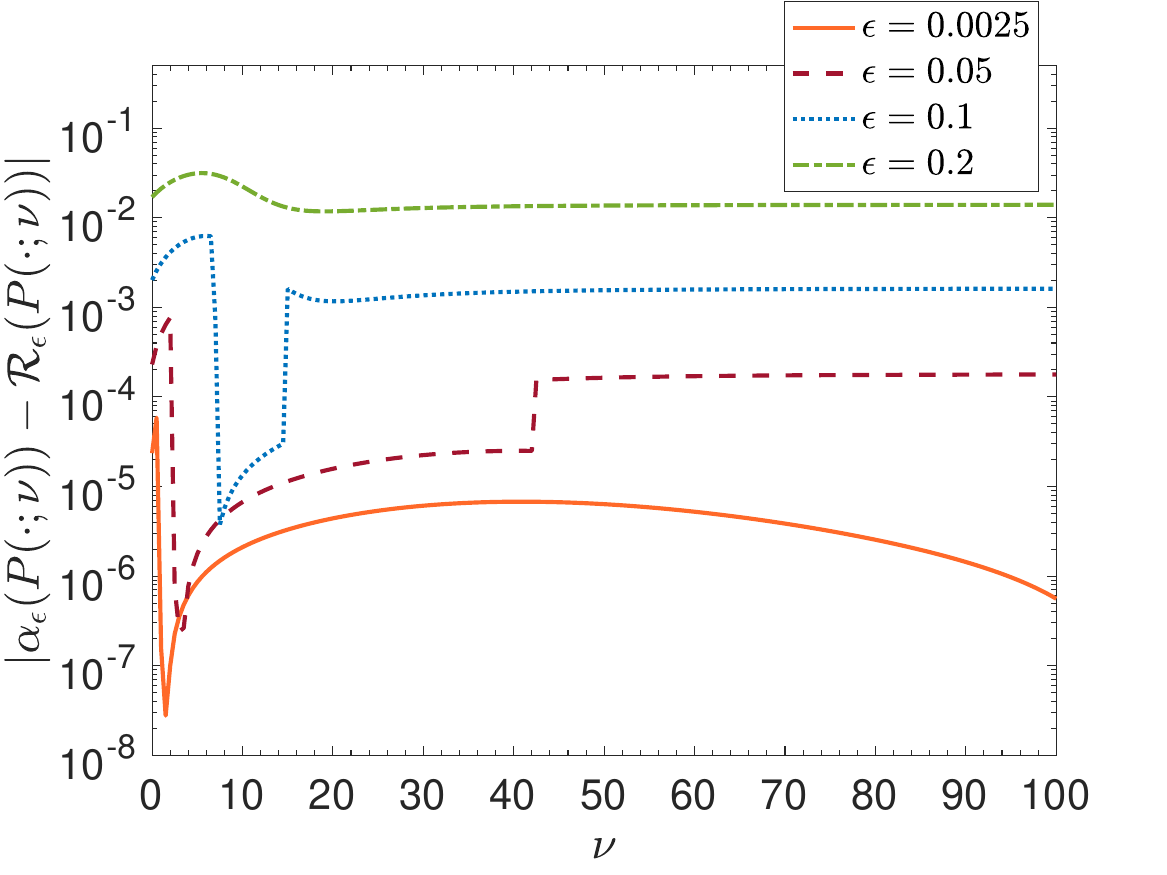} 
	\end{tabular}
		\caption{  Approximation of $\alpha_{\epsilon}(P(\cdot ; \nu))$ with 
				${\mathcal R}_{\epsilon}(P(\cdot ; \nu))$	for the damping
				optimization problem in Example \ref{ex:damp1}. For the plot
				on the left, $\epsilon = 0.1$. }
		\label{fig:error}
\end{figure}
}
\end{example}

\section{Fixed-point iterations for the pseudospectral abscissa 
				of a nonlinear eigenvalue problem}\label{sec:fp_mat_val}
The first-order approximation formula in (\ref{eq:main_char_psa_nle}) in Theorem \ref{thm:fo_estimate} 
for $\alpha_{\epsilon}(T)$ gives a good estimate if $\epsilon$ is small. For larger
values of $\epsilon$, this estimate is crude. Here, we describe two fixed-point iterations
for an accurate estimation of $\alpha_{\epsilon}(T)$ even when $\epsilon$ is large
based on the perturbation theory presented in the previous section. The idea
behind the fixed-iterations makes use of also the following backward error result, whose proof
is given in Appendix \ref{sec:proof_bw_nle_error}. Similar backward errors are considered
in the literature extensively; see, e.g., \cite[Section 4.2]{TisM01}, \cite[Section 2.2]{TisH01},
\cite[Theorem 1]{MicGWN06}.
\begin{theorem}\label{thm:berror_NEP}
Let $z \in {\mathbb C}$. We have
\vskip 1ex
\begin{equation*}
	\min \bigg\{ \epsilon \;\:  \bigg| \;\: \exists \Delta T \in {\mathcal P}_{\epsilon}	
				\;\, \mathrm{s.t.} \;\,	\mathrm{det}\big\{ (T + \Delta T)(z) \big\} = 0	\bigg\}	
		\;	=	\;
		\frac{\sigma_{\min}(T(z))}{\sqrt{ w_1^2 | t_1(z) |^2 + \dots + w_\kappa^2 | t_\kappa(z) |^2}}
				=:		\varphi(z)	
	\;	,
\end{equation*}
\vskip 1ex
\noindent
where ${\mathcal P}_{\epsilon}$ is the set defined in (\ref{eq:defn_NEP_pspec}).
Moreover, letting $u, v$ be consistent unit left, unit right singular vectors corresponding to 
$\sigma_{\min}(T(z))$, we have 
\vskip 1.5ex
\begin{tabular}{l}
		\hskip -4ex		\phantom{i}\textbf{(i)}
		$\mathrm{det}\left\{ (T + \underline{\Delta T})(z) \right\} = 0$, and	\\[.45em]
		\hskip -4ex		\textbf{(ii)}
		the inclusion $\underline{\Delta T} \in {\mathcal P}_{\varphi(z)}$
\end{tabular} 
\vskip 1.5ex
\noindent
for  $\underline{\Delta T}(\lambda) 
			= t_1(\lambda) w_1 \underline{\Delta T}_{\, 1} + \dots 
						+ t_\kappa(\lambda) w_\kappa \underline{\Delta T}_{\, \kappa}$
with
\begin{equation}\label{eq:defn_Tj}
	\underline{\Delta T}_{\, j}
		\;	:=	\;
	\frac{-\varphi(z) w_j \overline{t_j(z)} u v^\ast}{\sqrt{ w_1^2 | t_1(z) |^2 + \dots + w_\kappa^2 | t_\kappa(z) |^2}}	\:	,
	\quad	j = 1, \dots , \kappa.
\end{equation}
Indeed, $u^\ast \{ (T + \underline{\Delta T})(z) \} = 0$ and $\{ (T + \underline{\Delta T})(z) \} v = 0$, i.e.,
$u, v$ are left, right eigenvectors of $(T + \underline{\Delta T})(\lambda)$ corresponding to its eigenvalue $z$.
\end{theorem}

\noindent
The quantity $\varphi(z)$ in Theorem \ref{thm:berror_NEP} is the backward error of $z \in {\mathbb C}$ to be an 
eigenvalue of $T(\lambda)$. In other words, it is the norm of the smallest perturbation of $T(\lambda)$ so 
that $z$ is an eigenvalue of the perturbed matrix-valued function. 
As an immediate corollary of Theorem \ref{thm:berror_NEP},
the $\epsilon$-pseudospectrum of $T$ can be characterized as follows.
\begin{theorem}\label{thm:pspec_char}
The following assertion holds:
\begin{equation}\label{eq:pspec_char}
	\Lambda_{\epsilon}(T)
		\;	=	\;
	\left\{ z \in {\mathbb C}	\;\;\;		\bigg|	\;\;\;
		\frac{\sigma_{\min}(T(z))}
			{\sqrt{ w_1^2 | t_1(z) |^2 + \dots + w_\kappa^2 | t_\kappa(z) |^2}}  \; \leq \; \epsilon	
		\right\}	\:	.
\end{equation}
\end{theorem}


Now we describe a 
fixed-point iteration to estimate $\alpha_{\epsilon}(T)$. We start by perturbing $T$ in the direction
$\Delta T^{(0)}(\lambda) = t_1(\lambda) w_1 \Delta T^{(0)}_1 + 
			\dots + t_{\kappa}(\lambda) w_{\kappa} \Delta T^{(0)}_\kappa$
such that
\[
	\Delta T^{(0)}_{j}
			\;	:=	\;
	\frac{w_j \, \overline{t_j(\lambda_R)} y_0 x^\ast_0}
				{\sqrt{\sum_{\ell=1}^\kappa w_\ell^2 | t_\ell(\lambda_R) |^2}}	\;	,	
	\;\;\;	j	=	1,	\dots , \kappa \: ,
\]
where $\lambda_R$ is an eigenvalue of $T$, and $y_0, x_0$ are unit left, right eigenvectors of 
$T(\lambda)$ corresponding  to $\lambda_R$ normalized so that
\[
	\|x_0 \|_2 = \| y_0 \|_2 = 1
	\quad	\text{and}		\quad
	y^\ast_0 T'(\lambda_R) x_0 \text{ is real and negative}	. 
\]
This initial perturbation is similar to the one used in the fixed-point iteration in \cite{MicG12} for a slightly 
different definition of the $\epsilon$-pseudospectrum.
Letting $\Delta^{(0)} = (\Delta T^{(0)}_1 , \dots , \Delta T^{(0)}_\kappa)$, as asserted by Theorem \ref{thm:max_LDelta}, 
this choice is justified by the property
\[
		\Delta^{(0)}	\, \in \,	
			\argmax_{\Delta \in {\mathcal S}}	\:	\text{Re} \left\{ \mu'(0; \Delta, \lambda_R, T) \right\}
			\:	.
\]


Let $z_1$ be the rightmost eigenvalue of $(T + \epsilon \Delta T^{(0)})(\lambda)$. Clearly,
$z_1 \in \Lambda_{\epsilon}(T)$. In addition to $\epsilon \Delta T_0 \in {\mathcal P}_{\epsilon}$, 
there may be several perturbations $\Delta T \in {\mathcal P}_{\epsilon}$ such that $z_1$ 
is an eigenvalue of 
$(T + \Delta T)(\lambda)$. In particular, by Theorem \ref{thm:berror_NEP}, a minimal perturbation 
$\underline{\Delta T}^{(0)}(\lambda) = t_1(\lambda) w_1 \underline{\Delta T}^{(0)}_{\, 1} 
						+ \dots + t_{\kappa}(\lambda) w_{\kappa} \underline{\Delta T}^{(0)}_{\, \kappa}$ 
such that $z_1$ is an eigenvalue of $(T + \underline{\Delta T}^{(0)})(\lambda)$ is given by
the choice
\[
	\underline{\Delta T}^{(0)}_{\, j}
			\;	:=	\;
	\frac{-\varphi(z_1) w_j \overline{t_j(z_1)} u_1 v_1^\ast}
					{\sqrt{ \sum_{\ell = 1}^\kappa w_\ell^2 | t_\ell(z_1) |^2}}	
	\;	,	
	\;\;\;	j	=	1,	\dots , \kappa \:		,
\] 
where 
$\varphi(z_1) = \sigma_{\min}(T(z_1))/\sqrt{ \sum_{\ell = 1}^\kappa w_\ell^2 | t_\ell(z_1) |^2} \,$ and 
$u_1, v_1$ are consistent unit left, right singular vectors corresponding to 
$\sigma_{\min}(T(z_1))$. 
Recall also that, by Theorem \ref{thm:berror_NEP}, $u_1, v_1$ are left, right
eigenvectors of $(T + \underline{\Delta T}^{(0)})(\lambda)$ corresponding to its eigenvalue $z_1$. 
Consider instead the left, right eigenvectors $\widetilde{u}_1, \widetilde{v}_1$ of $(T + \underline{\Delta T}^{(0)})(\lambda)$ 
corresponding to the eigenvalue $z_1$ normalized such that
\begin{equation}\label{eq:nle_norm_cond}
	\| \widetilde{u}_1 \|_2 = \| \widetilde{v}_1 \|_2 = 1
	\quad\;\;	\text{and}		\quad\;\;
	\widetilde{u}_1^\ast \, \{ (T + \underline{\Delta T}^{(0)})'(z_1) \}  \widetilde{v}_1 
										\text{ is real and negative}	.
\end{equation}
Note that, to fulfill the normalization conditions in (\ref{eq:nle_norm_cond}), 
we can choose $\widetilde{u}_1 = -e^{{\rm i} \theta} u_1$, 
$\widetilde{v}_1 = v_1$, where $\theta$ is such that   
\begin{equation*}
\begin{split}
	\rho e^{{\rm i} \theta}
		\;	&	=	\;
		u_1^\ast \left\{ (T + \underline{\Delta T}^{(0)})'(z_1) \right\} v_1
			\;	=	\;
		u_1^\ast \left\{ T'(z_1) + (\underline{\Delta T}^{(0)})'(z_1) \right\} v_1		\\
		\;	&	=	\;
		u_1^\ast T'(z_1) v_1
				+
		\left\{ \frac{-\sigma_{\min}(T(z_1))}{\sum_{\ell=1}^\kappa w_\ell^2 | t_\ell(z_1) |^2} \right\}
		\sum_{\ell=1}^\kappa
			w_\ell^2 t_\ell'(z_1) \overline{t_\ell(z_1)}	\:	.
\end{split}
\end{equation*}
Now $\Delta^{(1)} = (\Delta T^{(1)}_1 , \dots , \Delta T^{(1)}_{\kappa})$ with
\[
	\Delta T^{(1)}_{j}
			\;	:=	\;
	\frac{w_j \, \overline{t_j(z_1)} \, \widetilde{u}_1 \widetilde{v}_1^\ast}
					{\sqrt{\sum_{\ell = 1}^\kappa w_\ell^2 | t_\ell(z_1) |^2}}	\;	,	
	\;\;\;	j	=	1,	\dots , \kappa \: 
\] 
corresponding to the perturbation 
$\Delta T^{(1)}(\lambda) = t_1(\lambda) w_1 \Delta T^{(1)}_1 + \dots + t_{\kappa}(\lambda) w_{\kappa} \Delta T^{(1)}_\kappa$
satisfies
\[
		\Delta^{(1)}	\, \in \,	
		\argmax_{\Delta \in {\mathcal S}}	\:	
				\text{Re} \left\{ \mu'(0; \Delta, z_1, T + \underline{\Delta T}^{(0)}) \right\}	\:	.
\]
We define $z_2$ as the rightmost eigenvalue of 
$(T + \epsilon \Delta T^{(1)})(\lambda)$, and
repeat the procedure. The resulting fixed-point iteration is outlined in 
Algorithm \ref{alg:fixed_point} below.

\begin{algorithm}
\begin{algorithmic}[1]
	\REQUIRE{A matrix-valued function $T$ as in (\ref{eq:mat_fun}), a real number $\epsilon > 0$,
				tolerance for termination $\mathsf{tol} > 0$.}
	\ENSURE{Estimates $f$ for $\alpha_{\epsilon}(T)$ and $z$
						for globally rightmost point in $\Lambda_{\epsilon}(T)$.}	
	\vskip 1.2ex
	\STATE{$z_0 \; \gets$ an eigenvalue of $T$.}\label{line_alg1:zm1}
			
	\vskip 1.2ex		
			
	\STATE{$x, y \; \gets$ unit right, left eigenvectors corr. to
				rightmost eigenvalue of $T$.}	\label{line_alg1:z0}
				
	\vskip 1.5ex
	
	\STATE{$y \; \gets \; - \{ (y^\ast T'(z_0) x) / | y^\ast T'(z_0) x| \} y $.}
	
	\vskip 1.7ex
	
	\STATE{$\Delta T^{(0)}_{j}
			\gets
			\left\{	 w_j \, \overline{t_j(z_0)} y x^\ast  \right\}	/
					\sqrt{\sum_{\ell=1}^\kappa w_\ell^2 | t_\ell(z_0) |^2 }  \,$ 
				for $j = 1, \dots , \kappa$.} 
				
	\vskip 1.3ex

\STATE{$\Delta T^{(0)}(\lambda) \; \gets \;  
			t_1(\lambda) w_1 \Delta T^{(0)}_1 + \dots + t_{\kappa}(\lambda) w_{\kappa} \Delta T^{(0)}_\kappa$.}
			\label{line_alg1:T0}
			
	\vskip 1.5ex
	
	\FOR{$k=1,2,\dots$}
	
	\vskip 1.2ex
	
	\STATE{$z_k \; \gets \;$ rightmost eigenvalue of $(T + \epsilon \Delta T^{(k-1)})(\lambda)$.}
	
	\vskip 1.2ex
	
	\STATE{\textbf{If} $| z_{k} - z_{k-1} | < \mathsf{tol} \;$ \textbf{return} 
						$z \gets z_k$, $f \gets \text{Re}(z_k)$.}\label{line:fp1_nle_ter}
						
	\vskip 1.2ex
	
	\STATE{$v_k, u_k \; \gets \;$ unit consistent right, left
							singular vectors corr. to $\sigma_{\min}(T(z_k))$.} 
							
	\vskip 1.9ex
	
	\STATE{$\delta_k \; \gets \; 	\left\{-\sigma_{\min}(T(z_k))
								/ \sum_{\ell=1}^\kappa w_\ell^2 | t_\ell(z_k) |^2 \right\}
		\sum_{j=1}^\kappa
			w_j^2 t_j'(z_k) \overline{t_j(z_k)}$}\label{line:fp1_nle_delta}				
							
	\vskip 1.8ex
	
	\STATE{$u_k \; \gets \; - \left\{ [u_k^\ast T'(z_k) v_k + \delta_k ]  
							/ | u_k^\ast T'(z_k) v_k + \delta_k | \right\} u_k $.}
	
	\vskip 1.3ex
	
	\STATE{$\Delta T^{(k)}_{j}
			\gets
			\left\{ w_j \, \overline{t_j(z_k)} u_k v_k^\ast \right\} /
				\sqrt{\sum_{\ell=1}^\kappa w_\ell^2 | t_\ell(z_k) |^2}  \,$ 
					for $j = 1, \dots , \kappa$.}  
	
	\vskip 1ex
	
	\STATE{$\Delta T^{(k)}(\lambda) \; \gets \;  
		t_1(\lambda) w_1 \Delta T^{(k)}_1 + \dots + t_{\kappa}(\lambda) w_{\kappa} \Delta T^{(k)}_\kappa$.}
	
	\vskip 1.5ex
	
	\ENDFOR
\end{algorithmic}
\caption{Fixed-point iteration for the pseudospectral abscissa of a matrix-valued function}
\label{alg:fixed_point}
\end{algorithm}

\subsection{Fixed-points of Algorithm \ref{alg:fixed_point}}\label{sec:fp_analysis}
We next analyze the fixed-points of Algorithm \ref{alg:fixed_point}.
To be precise, the fixed-point function associated with Algorithm \ref{alg:fixed_point}
is the continuous map $\zeta : {\mathcal D} \rightarrow {\mathcal D}$ such that $\widetilde{z} = \zeta(z)$
for any $z \in {\mathcal D}$ is defined as follows:
\begin{enumerate}
	\item Let $v, u$ be unit consistent right, left
							singular vectors corresponding to $\sigma_{\min}(T(z))$.
							
	\item Set $\widetilde{u} := - \left\{ [u^\ast T'(z) v + \delta ]  
							/ | u^\ast T'(z) v + \delta | \right\} u $, with
		\[
			\delta	:=  	-\frac{\sigma_{\min}(T(z))}{\sum_{\ell =1}^\kappa w_\ell^2 | t_\ell(z) |^2} 
				\sum_{\ell=1}^\kappa
						w_\ell^2 t_\ell'(z) \overline{t_\ell(z)}	.
		\]
	\item Set  
			$\Delta T(\lambda)  :=
				t_1(\lambda) w_1 \Delta T_1 + \dots + t_{\kappa}(\lambda) w_{\kappa} \Delta T_\kappa$,
		where
		\[
			\Delta T_{j}
				:=
			\frac{w_j \, \overline{t_j(z)} \widetilde{u} v^\ast} 
				{\sqrt{\sum_{\ell=1}^\kappa w_\ell^2 | t_\ell(z) |^2}}	 \: ,	\quad 
					j = 1, \dots , \kappa  \: .
		\]
	\item $\widetilde{z}$ is the rightmost eigenvalue of $(T + \epsilon \Delta T)(\lambda)$
	(if the rightmost eigenvalue is not unique, take one of the rightmost eigenvalues, e.g., the one with the
	largest imaginary part).
\end{enumerate}
The iterates of Algorithm \ref{alg:fixed_point} can be expressed as $z^{(k+1)} = \zeta(z^{(k)})$ for $k \geq 0$,
so, by the continuity of $\zeta$, the sequence $\{ z^{(k)} \}$ generated by Algorithm \ref{alg:fixed_point} 
without termination (i.e., without line \ref{line:fp1_nle_ter}) can only converge to the fixed-points of $\zeta$.


The subsequent arguments make use of the matrix-valued function
\begin{equation}\label{eq:defn_M}
	M(\lambda)
		\;	:=	\;
		\frac{T(\lambda)}{\sqrt{ w_1^2 | t_1(\lambda) |^2 + \dots + w_\kappa^2 | t_\kappa(\lambda) |^2}}	\:	,
\end{equation}
which is not analytic because of the term in the denominator.
However, the matrix-valued function ${\mathcal M}(s_1, s_2) := M(s_1 + {\rm i} s_2)$ over the domain 
$\{ (x, y) \; | \; x, y \in {\mathbb R}\, \text{ s.t. }\, x + {\rm i} y \in {\mathcal D} \} \subseteq {\mathbb R}^2$ is real-analytic, 
i.e., the real and imaginary parts of ${\mathcal M}(s_1, s_2)$
are real-analytic. We denote with ${\mathcal M}_{s_1}$, ${\mathcal M}_{s_2}$ the partial
derivatives of ${\mathcal M}$ with respect to $s_1$, $s_2$, respectively.
These (possibly complex) partial derivatives at a point $(x, y) \in {\mathbb R}^2$
such that $z = x + {\rm i} y \in {\mathcal D}$ are explicitly given by
 \small
\begin{equation}\label{eq:der_M}
\begin{split}
	&	
\hskip -.5ex
		{\mathcal M}_{s_1}(x,y)
			=
	\frac{1}{ \sqrt{ \sum_{\ell = 1}^\kappa w_\ell^2 | t_\ell (z) |^2}}
	\left\{  T'(z)  	-		
			\frac{ T(z) }{  \sum_{\ell = 1}^\kappa w_\ell^2 | t_\ell (z) |^2  }
			\sum_{\ell=1}^\kappa
						w_\ell^2 \text{Re}\left[ t'_\ell(z) \overline{t_\ell(z)} \right]		\right\} \: ,	\\[.4em]	
	&	
\hskip -.5ex
		{\mathcal M}_{s_2}(x,y)
			=
	\frac{1}{ \sqrt{\sum_{\ell = 1}^\kappa w_\ell^2 | t_\ell (z) |^2}}
	\left\{  {\mathrm i} T'(z)  	+		
			\frac{ T(z) }{  \sum_{\ell = 1}^\kappa w_\ell^2 | t_\ell (z) |^2  }
			\sum_{\ell=1}^\kappa
						w_\ell^2 \text{Im}\left[ t'_\ell(z) \overline{t_\ell(z)} \right]		\right\}	\:	.
\end{split}
\end{equation}
\normalsize
In our analysis, we make use of the map
\begin{equation}\label{eq:nond_opt_map}
	{\mathcal S} : {\mathcal D} \rightarrow {\mathbb C}	\,	,	\;\;\;
	{\mathcal S}(z)  \; = \;  \text{Re}\left\{ u^\ast {\mathcal M}_{s_1}(x,y) v \right\} 	\: -  \:	 
						{\rm i} \cdot \text{Re}\left\{ u^\ast {\mathcal M}_{s_2}(x,y) v \right\}	\:	,
\end{equation} 
where $x, y \in {\mathbb R}$ are such that $z = x + {\rm i} y \in {\mathcal D}$, and
$u$, $v$ denote a pair of consistent unit left, 
		unit right singular vectors of $T(z)$ corresponding to $\sigma_{\min}(T(z))$.
We define nondegenerate points of $\Lambda_{\epsilon}(T)$ 
in terms of ${\mathcal S}$ (and so in terms of ${\mathcal M}, M$) as follows.
\begin{definition}\label{defn:nondegenerate}
	A point $z \in \Lambda_{\epsilon}(T)$ is called nondegenerate if
	\begin{enumerate}
		\item the smallest singular value of $T(z)$ is simple, and
		\item ${\mathcal S}(z) \neq 0$.
	\end{enumerate}
\end{definition}

\noindent
Our main result in this subsection makes use of the following lemma regarding 
nondegenerate points on the boundary of $\Lambda_{\epsilon}(T)$.
A proof of the lemma is included in Appendix \ref{sec:opt_rightmost_pt}.
\begin{lemma}\label{Lem:char_rightp}
Let $z \in {\mathcal D}$ be a nondegenerate point. If the point $z$ is a locally rightmost 
point in $\Lambda_{\epsilon}(T)$, then ${\mathcal S}(z)$ is real and positive.
\end{lemma}

\noindent
A nondegenerate point $z \in {\mathcal D}$ on the boundary of $\Lambda_{\epsilon}(T)$
such that ${\mathcal S}(z)$ is real and positive corresponds to a point
on the right boundary of $\Lambda_{\epsilon}(T)$ (i.e., in the sense that
$z + h$ does not belong to $\Lambda_{\epsilon}(T)$ for all positive $h \in {\mathbb R}$
sufficiently small) with a vertical tangent line. We give a special name to such points.
\begin{definition}\label{defn:rbvt_point} 
We call $z \in {\mathbb C}$ an rbvt (right boundary with a vertical tangent) point if 
\begin{enumerate}
	\item $z$ is nondegenerate,
	\item $z$ is on the boundary of $\Lambda_{\epsilon}(T)$, and
	\item ${\mathcal S}(z)$ is real and positive. 
\end{enumerate} 
\end{definition}

\noindent
Lemma \ref{Lem:char_rightp} shows that a nondegenerate locally rightmost
point in $\Lambda_{\epsilon}(T)$ is indeed an rbvt point. Conversely, an rbvt point is 
likely to be, but not necessarily, a locally rightmost point in $\Lambda_{\epsilon}(T)$.
In particular, if, at an rbvt point $z$, the second derivative of $\sigma_{\min}(M(\lambda))$
with respect to the imaginary part of $\lambda$ is negative, then every ball 
${\mathcal B}_r(z) = \{ \widetilde{z} \in {\mathbb C} \; | \; |\widetilde{z} - z | < r \}$ with 
positive radius $r$ contains a point inside $\Lambda_{\epsilon}(T)$ with real part
greater than $\text{Re}(z)$, so $z$ is not locally rightmost in $\Lambda_{\epsilon}(T)$.


We now present the main result that characterizes the fixed points of the
map $\zeta$ associated with Algorithm \ref{alg:fixed_point}. This result
implies that Algorithm \ref{alg:fixed_point}, if it converges to a nondegenerate point of $\Lambda_{\epsilon}(T)$,
the converged point must be an rbvt point, likely to be a locally rightmost point in $\Lambda_{\epsilon}(T)$.

\begin{theorem}\label{thm:fp_main_result}
Let $z \in \Lambda_{\epsilon}(T)$ be nondegenerate, and $\zeta$ be the fixed-point
map associated with Algorithm \ref{alg:fixed_point} defined as above.
\begin{enumerate}
	\item If $z$ is not an rbvt point,
	then $z$ is not a fixed-point of $\zeta$.
	\item If $z$ is the unique globally rightmost point in $\Lambda_{\epsilon}(T)$,
	then $z$ is a fixed-point of $\zeta$.
\end{enumerate}
\end{theorem}
\begin{proof}
Let us first prove 2. Since $z$ is a globally rightmost point and nondegenerate, we must have
${\mathcal S}(z)$ real and positive by Lemma \ref{Lem:char_rightp}. Now,
recalling the definition of ${\mathcal S}(z)$ in (\ref{eq:nond_opt_map}) with $u$, $v$
denoting a pair of consistent unit left, unit right singular vectors corresponding to $\sigma_{\min}(T(z))$, 
and observing $u^\ast T(z) v = \sigma_{\min}(T(z)) \| u \|_2^2 = \sigma_{\min}(T(z))$ is real,
it follows from the expressions in (\ref{eq:der_M}) for the partial derivatives of ${\mathcal M}$ that
\begin{equation*}
	{\mathcal S}(z)
		\;		=	\;
	\frac{1}{ \sqrt{\sum_{\ell = 1}^\kappa w_\ell^2 | t_\ell (z) |^2}}
	\left\{ u^\ast T'(z) v 	-		
			\frac{u^\ast T(z) v}{  \sum_{\ell = 1}^\kappa w_\ell^2 | t_\ell (z) |^2  }
			\sum_{\ell=1}^\kappa
						w_\ell^2 t_\ell'(z) \overline{t_\ell(z)}		\right\}	\:	.
\end{equation*}
As a result, since ${\mathcal S}(z)$ is real and positive,
\begin{equation}\label{eq:same_as_Mp}
	 u^\ast T'(z) v 	-		
			\frac{u^\ast T(z) v}{  \sum_{\ell = 1}^\kappa w_\ell^2 | t_\ell (z) |^2  }
			\sum_{\ell=1}^\kappa
						w_\ell^2 t_\ell'(z) \overline{t_\ell(z)}		
\end{equation}
is also real and positive, implying $\widetilde{u} = -u$  in the definition of $\zeta$. It follows that
\[
	\Delta T_{j}
				:=
			- \frac{w_j \, \overline{t_j(z)} u v^\ast} 
				{\sqrt{\sum_{\ell=1}^\kappa w_\ell^2 | t_\ell(z) |^2}}	 \: ,	\quad 
					j = 1, \dots , \kappa  \: .
\]
By definition, $\widetilde{z} = \zeta(z)$ is a rightmost eigenvalue of $(T + \epsilon \Delta T)(\lambda)$,
where $\Delta T(\lambda) = \sum_{\ell = 1}^\kappa w_\ell t_\ell(\lambda) \Delta T_\ell$. Additionally,
\begin{equation}\label{eq:fp1_der_int}	
\begin{split}
	\left\{ (T + \epsilon \Delta T)(z) \right\} v	\;	&	=	\;
				T(z) v	+	\epsilon \Delta T(z) v		\\[.4em]
				&	=	\;
				\sigma_{\min} (T(z)) u		+	\epsilon \sum_{j = 1}^\kappa  w_j t_j(z) \Delta T_j v	\\[.2em]
				&	=	\;
				\epsilon \left\{ \sqrt{\sum_{\ell=1}^\kappa w_\ell^2 | t_\ell(z) |^2} \right\} u
							-
		\epsilon \sum_{j=1}^\kappa \frac{w_j^2 | t_j(z) |^2}{\sqrt{\sum_{\ell=1}^\kappa w_\ell^2 | t_\ell(z) |^2}}u	\\[.6em]
				&	=	\;
				\epsilon \left\{ \sqrt{\sum_{\ell=1}^\kappa w_\ell^2 | t_\ell(z) |^2} \right\} u
							-
				\epsilon \left\{ \sqrt{\sum_{\ell=1}^\kappa w_\ell^2 | t_\ell(z) |^2} \right\} u		\;	=	\;	0	\:	,	
\end{split}
\end{equation}
where for the third equality, we note 
$\sigma_{\min} (T(z)) = \epsilon \bigg\{ \sqrt{\sum_{\ell=1}^\kappa w_\ell^2 | t_\ell(z) |^2} \bigg\}$,
as $z$ is the globally rightmost point in $\Lambda_{\epsilon}(T)$, using also the characterization
of $\Lambda_{\epsilon}(T)$ in (\ref{eq:pspec_char}).
This shows that $z$ is also an eigenvalue of $(T + \epsilon \Delta T)(\lambda)$. Indeed, the point $z$ must be
the unique rightmost eigenvalue of $(T + \epsilon \Delta T)(\lambda)$, as otherwise 
(i.e., if, in addition to $z$, the matrix-valued function
$(T + \epsilon \Delta T)(\lambda)$ has other rightmost eigenvalues) $z$ cannot be the
unique rightmost point in $\Lambda_{\epsilon}(T)$. Consequently, $\widetilde{z} = z$, proving that $z$
is a fixed-point of $\zeta$.

As for 1. let us first assume $z$ is on the boundary of $\Lambda_{\epsilon}(T)$. Since $z$ is not an
rbvt point, then 
the scalar ${\mathcal S}(z)$ is either not real, or real but negative. Following the steps in the previous paragraph,
${\mathcal S}(z)$ and the matrix in (\ref{eq:same_as_Mp}) have the same complex sign, say $e^{{\rm i} \theta}$
such that $\theta \neq 0$. Consequently, $\widetilde{u} = - e^{{\rm i} \theta} u$. Now let 
$\{v, v_2, \dots , v_n \}$ and $\{u, u_2, \dots , u_n \}$ be the orthonormal sets consisting of all right 
singular vectors and left singular vectors, respectively, of $T(z)$ satisfying $T(z) v_j = \sigma_j(z) u_j$
for $j = 2, \dots , n$, where $\sigma_j(z)$ is a singular value of $T(z)$ other than $\sigma_{\min}(T(z))$.
Take any vector $\widehat{v} \in {\mathbb C}^n$ and expand it as 
$\widehat{v} = c v + c_2 v_2 + \dots + c_n v_n$ for some scalars $c, c_2, \dots , c_n \in {\mathbb C}$.
Now suppose
\begin{equation}\label{eq:fp1_der_int2}
\begin{split}
	0	\;	&	=	\;
	\left\{ (T + \epsilon \Delta T)(z) \right\} \widehat{v}	\;		=	\;
				T(z) \widehat{v}	+	\epsilon \Delta T(z) \widehat{v}		\\[.2em]
				&	\quad\quad\quad	=	\;
		\epsilon \left\{ \sqrt{\sum_{\ell=1}^\kappa w_\ell^2 | t_\ell(z) |^2} \right\} c u + \sum_{j=2}^n c_j \sigma_j(z) u_j 
			-
		\epsilon \left\{ \sqrt{\sum_{\ell=1}^\kappa w_\ell^2 | t_\ell(z) |^2} \right\} c e^{{\rm i} \theta} u	\:	,
\end{split}
\end{equation}
where the steps when passing to the second line are similar to those in (\ref{eq:fp1_der_int}).
By the linear independence of $\{u, u_2, \dots , u_n \}$, we have $c_2 = \dots = c_n = 0$,
and $c(1 - e^{{\rm i} \theta}) = 0$, which implies $c = 0$ as $\theta \neq 0$. Consequently, $\widehat{v} = 0$.
This shows that $(T + \epsilon \Delta T)(z)$ is invertible, so
$z$ is not an eigenvalue of $(T + \epsilon \Delta T)(\lambda)$, whereas $\widetilde{z}$
is an eigenvalue of $(T + \epsilon \Delta T)(\lambda)$ by definition. Consequently, $\widetilde{z} \neq z$,
and $z$ is not a fixed-point of $\zeta$. 

Finally, if $z$ is in the interior of $\Lambda_{\epsilon}(T)$, then by the nondegeneracy of $z$, 
we have ${\mathcal S}(z) \neq 0$.
If $\sigma_{\min}(M(z)) = \epsilon$, then $z$ is a local maximizer of $\sigma_{\min}(M(z))$,
but this contradicts with ${\mathcal S}(z) \neq 0$ (i.e., the first-order necessary conditions
for $z$ to be a local maximizer of $\sigma_{\min}(M(z))$ imply ${\mathcal S}(z) = 0$). 
Thus, $\sigma_{\min}(M(z)) < \epsilon$, equivalently 
\[
	\sigma_{\min}(T(z)) = \rho \sqrt{\sum_{\ell=1}^\kappa w_\ell^2 | t_\ell(z) |^2}
	< \epsilon \sqrt{\sum_{\ell=1}^\kappa w_\ell^2 | t_\ell(z) |^2}.
\]
for some nonnegative real number $\rho < \epsilon$. Proceeding as in the previous paragraph, letting
$\{v, v_2, \dots , v_n \}$ and $\{u, u_2, \dots , u_n \}$ be the orthonormal sets consisting of all right 
singular vectors and left singular vectors, take any $\widehat{v} \in {\mathbb C}$, and
expand it as $\widehat{v} = c v + c_2 v_2 + \dots + c_n v_n$ for some scalars $c, c_2, \dots , c_n \in {\mathbb C}$.
Calculations analogous to (\ref{eq:fp1_der_int2}) yield
\begin{equation}\label{eq:fp1_der_int3}
\begin{split}
	0	\;	&	=	\;
	\left\{ (T + \epsilon \Delta T)(z) \right\} \widehat{v}	\;		=	\;
				T(z) \widehat{v}	+	\epsilon \Delta T(z) \widehat{v}		\\[.2em]
				&	\quad\quad\quad	=	\;
		\rho \left\{ \sqrt{\sum_{\ell=1}^\kappa w_\ell^2 | t_\ell(z) |^2} \right\} c u + \sum_{j=2}^n c_j \sigma_j(z) u_j 
			-
		\epsilon \left\{ \sqrt{\sum_{\ell=1}^\kappa w_\ell^2 | t_\ell(z) |^2} \right\} c e^{{\rm i} \theta} u	\:	,
\end{split}
\end{equation}
where now $\theta$ can be zero or nonzero. Linear independence of $\{u, u_2, \dots , u_n \}$ leads again to
$c_2 = \dots = c_n = 0$, as well as
and $c(\rho - \epsilon e^{{\rm i} \theta}) = 0$ implying $c = 0$ as $\rho < \epsilon$. It follows that
$\widehat{v} = 0$, and $(T + \epsilon \Delta T)(z)$ is invertible. Once again $z$ is not eigenvalue
and $\widetilde{z}$ is an eigenvalue of  $(T + \epsilon \Delta T)(\lambda)$, so $\widetilde{z} \neq z$.
\end{proof}

\subsection{An alternative fixed-point iteration}
Here we present an alternative fixed-point iteration to compute $\alpha_{\epsilon}(T)$
based on constant matrix perturbations of $T(\lambda)$. For a moment, let us assume
$T(\lambda)$ in (\ref{eq:mat_fun}) is such that $t_{\kappa}(\lambda) = 1$, and the weights
are such that $w_1  = \dots = w_{\kappa - 1} = 0$ and $w_\kappa = 1$. 
In this special case, the definition of the $\epsilon$-pseudospectrum given by (\ref{eq:defn_NEP_pspec})
and its characterization in (\ref{eq:pspec_char}) take the form
\begin{equation}\label{eq:psa_cons_per}
	\Lambda_{\epsilon}(T)
		\;	=	\;
	\bigcup_{\Delta \in {\mathbb C}^{n\times n}, \| \Delta \|_2 \leq \epsilon}
	\Lambda(T + \Delta)
		\;	=	\;
	\left\{
		z \in {\mathbb C}	\;	|	\;
			\sigma_{\min}(T(z)) \leq \epsilon
	\right\},
\end{equation}
where $\: (T + \Delta)(\lambda)	=	\left\{ \sum_{\ell = 1}^{\kappa - 1} t_{\ell}(\lambda) T_\ell \right\}
										+ \left\{ T_{\kappa} + \Delta \right\}$.
Hence,  in this case, $\Lambda_{\epsilon}(T)$ is defined in terms of only constant
perturbations of $T_{\kappa}$, whereas the perturbations of other coefficients
are not allowed. In this setting, Algorithm \ref{alg:fixed_point} takes
the special form in Algorithm \ref{alg:fixed_point1m}. We note that $\delta_k$ in 
line \ref{line:fp1_nle_delta} of Algorithm \ref{alg:fixed_point} becomes zero in 
the special setting, as $w_1 = \dots = w_{\kappa - 1} = 0$ and $t'_{\kappa}(\lambda) = 0$. 
The specialized algorithm is based on constant perturbations of $T_{\kappa}$;
in particular,
\begin{equation*}
	(T + \epsilon \Delta^{(k-1)})(\lambda) = 
				\left\{ \sum_{\ell = 1}^{\kappa - 1} t_{\ell}(\lambda) T_\ell \right\}
										+ \left\{ T_{\kappa} + \epsilon \Delta^{(k-1)} \right\}
\end{equation*}
in line \ref{line_alg2:cons_per} of Algorithm \ref{alg:fixed_point1m}.

\begin{algorithm}
\begin{algorithmic}[1]
	\REQUIRE{A matrix-valued function $T$ as in (\ref{eq:mat_fun}) with $t_{\kappa}(\lambda) = 1$
	and $w_1 = \dots = w_{\kappa - 1} = 0$, $w_{\kappa} = 1$, a real number $\epsilon > 0$,
				tolerance for termination $\mathsf{tol} > 0$.}
	\ENSURE{Estimates $f$ for $\alpha_{\epsilon}(T)$ and $z$
						for globally rightmost point in $\Lambda_{\epsilon}(T)$.}	
	\vskip 1.2ex
	\STATE{$z_0 \; \gets$ an eigenvalue of $T$.}\label{line_alg1:zm1}
			
	\vskip 1.2ex		
			
	\STATE{$x, y \; \gets$ unit right, left eigenvectors corr. to
				rightmost eigenvalue of $T$.}	\label{line_alg1:z0}
				
	\vskip 1.5ex
	
	\STATE{$y \; \gets \; - \{ (y^\ast T'(z_0) x) / | y^\ast T'(z_0) x| \} y $.}
	
	\vskip 1.7ex
	
	\STATE{$\Delta^{(0)}
			\gets
			 y x^\ast$.} 
			\label{line_alg2:Del0}
			
	\vskip 1.5ex
	
	\FOR{$k=1,2,\dots$}
	
	\vskip 1.2ex
	
	\STATE{$z_k \; \gets \;$ rightmost eigenvalue of $(T + \epsilon \Delta^{(k-1)})(\lambda)$.}\label{line_alg2:cons_per}
	
	\vskip 1.5ex
	
	\STATE{\textbf{If} $| z_{k} - z_{k-1} | < \mathsf{tol} \;$ \textbf{return} 
						$z \gets z_k$, $f \gets \text{Re}(z_k)$.}\label{line:fp2_nle_ter}
						
	\vskip 1.2ex
	
	\STATE{$v_k, u_k \; \gets \;$ unit consistent right, left
							singular vectors corr. to $\sigma_{\min}(T(z_k))$.}

	\vskip 1.8ex
	
	\STATE{$u_k \; \gets \; - \left\{ [u_k^\ast T'(z_k) v_k  ]  
							/ | u_k^\ast T'(z_k) v_k  | \right\} u_k $.}
	
	\vskip 2.3ex
	
	\STATE{$\Delta^{(k)}
			\gets
			u_k v_k^\ast$}

	\vskip 1.5ex
	
	\ENDFOR
\end{algorithmic}
\caption{Fixed-point iteration for the
pseudospectral abscissa of a matrix-valued function subject to constant perturbations}
\label{alg:fixed_point1m}
\end{algorithm}

Going back to our general setting in (\ref{eq:defn_NEP_pspec}), we can view $\Lambda_{\epsilon}(T)$ as the
$\epsilon$-pseudospectrum of
\begin{equation}\label{eq:defn_M2}
\begin{split}
	&
	M(\lambda)		=	\widetilde{t}_1(\lambda)	\widetilde{T}_1
						+	\dots		+	\widetilde{t}_\kappa(\lambda) \widetilde{T}_{\kappa}
						+	\widetilde{t}_{\kappa + 1}(\lambda) \widetilde{T}_{\kappa + 1}	
				\\[.7em]
	&
	\text{with}	\quad
	\widetilde{T}_j	\;	=	\;	T_j	\:	,	\quad
	\widetilde{t}_j(\lambda)	=
			\frac{t_j(\lambda)}
				{\sqrt{ w_1^2 | t_1(\lambda) |^2 + \dots + w_\kappa^2 | t_\kappa(\lambda) |^2}} \: ,
				\;\;\;\;		j = 1, \dots , \kappa	\:	, 	\\[.2em]
	&
	\phantom{\text{with}}	\quad
	\widetilde{T}_{\kappa + 1}	= 0	\:	,	\quad
	\widetilde{t}_{\kappa + 1}(\lambda) =1	\:	,		\\[.2em]
\end{split}
\end{equation}
and the weights $\widetilde{w}_{1} = \dots = \widetilde{w}_{\kappa} = 0$ 
and $\widetilde{w}_{\kappa+1} = 1$. Indeed, it follows from (\ref{eq:psa_cons_per}) that
\begin{equation*}
\begin{split}
	\Lambda_{\epsilon}(M)	
		\;	&	=	\;
	\left\{
		z \in {\mathbb C}	\;	|	\;
			\sigma_{\min}(M(z)) \leq \epsilon
	\right\}		\\[.3em]
		\;	&	=	\;
	\left\{ z \in {\mathbb C}	\;\;\;		\bigg|	\;\;\;
		\frac{\sigma_{\min}(T(z))}
			{\sqrt{ w_1^2 | t_1(z) |^2 + \dots + w_\kappa^2 | t_\kappa(z) |^2}}  \; \leq \; \epsilon	
		\right\}
			\;	=	\;
	\Lambda_{\epsilon}(T)	\:	,
\end{split}
\end{equation*}
where the second equality is due to
\[
	M(\lambda)
		\;	=	\;
	\frac{T(\lambda)}
				{\sqrt{ w_1^2 | t_1(\lambda) |^2 + \dots + w_\kappa^2 | t_\kappa(\lambda) |^2}}  \: ,
\]
as is apparent from (\ref{eq:defn_M2}). Hence, $\Lambda_{\epsilon}(T)$
is the same as the $\epsilon$-pseudospectrum of $M$ but when it is subject
to constant perturbations only. Now $\alpha_{\epsilon}(M) = \alpha_{\epsilon}(T)$,
and it seems we can apply Algorithm \ref{alg:fixed_point1m} to $M(z)$ to compute 
$\alpha_{\epsilon}(T)$. However, there is a technical difficulty, namely $M(\lambda)$ is not
analytic, and the ideas behind Algorithm \ref{alg:fixed_point1m} relies on the analyticity of $T(\lambda)$.
It is for instance not possible to replace $T(\lambda)$ in lines \ref{line_alg2m:normal1} 
and \ref{line_alg2m:normal2} of Algorithm \ref{alg:fixed_point1m} with $M(\lambda)$, as
$M(\lambda)$ is not differentiable (i.e., not holomorphic).


We instead proceed on the real-analytic counterpart ${\mathcal M}(s_1 , s_2) := M(s_1 + {\rm i} s_2)$,
when analyticity plays a role.
Let us in particular describe how we form $\Delta^{(k)}$ at iteration $k$
given the perturbation $\Delta^{(k-1)}$ from the previous iteration.
The point $z_k$ is now the rightmost eigenvalue of $(M + \epsilon \Delta^{(k-1)})(\lambda)$,
and $v_k$, $u_k$ are unit consistent right, left singular vectors corresponding to $\sigma_{\min}(M(z_k)) =: \gamma$. 
The point $z_k$ is also an eigenvalue of $(M + \gamma \underline{\Delta}^{(k-1)})(\lambda)$
for $\underline{\Delta}^{(k-1)} = u_k v_k^\ast$, and the vectors $v_k$, $u_k$ are corresponding right, left eigenvectors.
We assume $z_k$ is a simple eigenvalue of $(M + \gamma \underline{\Delta}^{(k-1)})(\lambda)$.
The perturbation $\Delta_k$ is the matrix $\Delta \in {\mathbb C}^{n\times n}$, $\| \Delta \|_2 \leq 1$ 
maximizing the rate of change in the real part of the eigenvalue $z_k$ of
$(M + \gamma \underline{\Delta}^{(k-1)} + \eta \Delta)(\lambda)$ at $\eta = 0$.

Hence, consider any 
$\Delta \in {\mathbb C}^{n\times n}$ such that $\| \Delta \|_2 \leq 1$. By the analytic implicit function
theorem, there exist an open interval $U$ containing 0 and unique real-analytic functions 
$\mu_1(\eta ; \Delta),  \mu_2(\eta ; \Delta)$ satisfying 
$\mu_1(0; \Delta) = \text{Re}(z_k), \; \mu_2(0 ; \Delta) = \text{Im}(z_k)$ and
\[
	\text{det}
	\{  ( {\mathcal M} + \gamma \underline{\Delta}^{(k-1)} + \eta \Delta)
		( \mu_1(\eta ; \Delta), \mu_2(\eta ; \Delta) ) \}	=	0	\quad	\forall \eta \in U.
\]
There is also an analytic matrix-valued function $v(\eta; \Delta)$
such that $v(0; \Delta) = v_k$ and
\[
	\left\{ \left( {\mathcal M} + \gamma \underline{\Delta}^{(k-1)} + \eta \Delta \right)
		\left( \mu_1(\eta ; \Delta), \mu_2(\eta ; \Delta) \right) \right\} \, v(\eta ; \Delta) 	\;	=	\;	0
		\:	
\]
for all $\eta \in U$ \cite[pages 32-33]{Rel69}. The equation above is analogous to (\ref{eq:eig_eqn}).
Differentiating the last equation at $\eta = 0$, then multiplying with $u_k^\ast$ from left, we obtain
\begin{equation}\label{eq:alg_der_nlp_inter}
	\mu_1'  \left\{ u^\ast_k {\mathcal M}_{s_1}(x_k, y_k) v_k \right\}		+	
			\mu_2'  \left\{ u^\ast_k {\mathcal M}_{s_2}(x_k , y_k) v_k \right\}
				+	u^\ast_k \Delta v_k 		\;	=	\;	0,
\end{equation}
where $x_k$, $y_k$ are real, imaginary parts of $z_k$, and
explicit expressions for ${\mathcal M}_{s_1}$, ${\mathcal M}_{s_2}$
are as in (\ref{eq:der_M}). Let
\begin{equation}\label{eq:defn_Md}
		M^{D}(\lambda)
			\;	:=	\;
		\frac{1}{ \sqrt{\sum_{\ell = 1}^\kappa w_\ell^2 | t_\ell (\lambda) |^2}}
		\left\{T'(\lambda)  	-		
			\frac{ T(\lambda) }{  \sum_{\ell = 1}^\kappa w_\ell^2 | t_\ell (\lambda) |^2  }
			\sum_{\ell=1}^\kappa
						w_\ell^2 t_\ell'(\lambda) \overline{t_\ell(\lambda)}		\right\}	\:	.
\end{equation}
It turns out that setting $\widetilde{u}_k = -u_k e^{{\rm i} \theta}$, $\widetilde{v}_k = v_k$
for $\theta$ such that $u_k^\ast M^{D}(z_k) v_k = \rho e^{{\rm i} \theta}$, we have
\[
	\text{Re} \{ \widetilde{u}^\ast_k {\mathcal M}_{s_1}(x_k, y_k) \widetilde{v}_k \}	
			=
		\widetilde{u}_k^\ast M^{D}(z_k) \widetilde{v}_k		\quad	\text{and}		\quad
	\text{Re} \{ \widetilde{u}^\ast_k {\mathcal M}_{s_2}(x_k, y_k) \widetilde{v}_k \}	
			=		0	\:	.
\]
Thus, multiplying (\ref{eq:alg_der_nlp_inter}) with $-e^{-{\rm i} \theta}$ and taking the real parts
yield
\begin{equation*}
\begin{split}
	0	\;	&	=	\;
		\mu_1' \text{Re} \left\{ \widetilde{u}^\ast_k {\mathcal M}_{s_1}(x_k, y_k) \widetilde{v}_k \right\}		+	
			\mu_2' \text{Re} \left\{ \widetilde{u}^\ast_k {\mathcal M}_{s_2}(x_k , y_k) \widetilde{v}_k \right\}
				+	\text{Re} \left\{ \widetilde{u}^\ast_k \Delta \widetilde{v}_k \right\}		\\[.5em]
		&	=	\;
				\mu'_1 \left\{ \widetilde{u}_k^\ast M^{D}(z_k) \widetilde{v}_k \right\}	+	
						\text{Re} \left\{ \widetilde{u}^\ast_k \Delta \widetilde{v}_k \right\}	\:	.	
\end{split}
\end{equation*}
Finally, as $\widetilde{u}_k^\ast M^{D}(z_k) \widetilde{v}_k = -\rho$ is real and negative,
we deduce
\[
	\mu'_1
		\;	=	\;
	- \frac{\text{Re} \left\{ \widetilde{u}^\ast_k \Delta \widetilde{v}_k \right\}}
				{\widetilde{u}_k^\ast M^{D}(z_k) \widetilde{v}_k}
		\;	=	\;
	\frac{\text{Re} \left\{ \widetilde{u}^\ast_k \Delta \widetilde{v}_k \right\}}
				{| \widetilde{u}_k^\ast M^{D}(z_k) \widetilde{v}_k |}	\:	.
\]
It follows that the matrix $\Delta$ with $\| \Delta \|_2 \leq 1$
maximizing $\mu'_1$ is given by $\Delta_k = \widetilde{u}_k \widetilde{v}_k^\ast$.
Initially, we start from an eigenvalue $z_0$ of $M(\lambda)$ (equivalently $T(\lambda)$), 
so $v_0 = x$, $u_0 =  y$ are right, left
eigenvectors of $M(\lambda)$ (equivalently $T(\lambda)$) corresponding to $z_0$.
The derivation above applies but with the simplification
$u_0^\ast M^D(z_0) v_0 = y^\ast T'(z_0) x$.
These arguments lead us to Algorithm \ref{alg:fixed_point2m} given below.

We emphasize that Algorithm \ref{alg:fixed_point2m} is inspired from the ideas behind
Algorithm \ref{alg:fixed_point1m} but designed for $M(z)$, and stated mostly in terms of $T(z)$
for which the $\epsilon$-pseudospectral abscissa is aimed at. To this end,
in line \ref{line_alg2m:cons_per} of Algorithm \ref{alg:fixed_point2m}, the point
$z_k$ is indeed a rightmost eigenvalue of $(M + \epsilon \Delta^{(k-1)})(\lambda)$,
which is equal to  
$\{ T(\lambda) / \sqrt{ \sum_{\ell = 1}^\kappa w_\ell^2 | t_\ell(\lambda) |^2} \}
			+ \epsilon \Delta^{(k-1)}$.
Additionally, in line \ref{line:fp2_nle_sval}, the vectors $v_k, u_k$ are consistent unit right,
unit left singular vectors of $M(z_k)$ corresponding to its smallest singular value,
as it turns out that the singular vectors of $T(z_k)$ and $M(z_k)$ are the same. 

\begin{algorithm}
\begin{algorithmic}[1]
	\REQUIRE{A matrix-valued function $T$ as in (\ref{eq:mat_fun}), a real number $\epsilon > 0$,
				tolerance for termination $\mathsf{tol} > 0$.}
	\ENSURE{Estimates $f$ for $\alpha_{\epsilon}(T)$ and $z$
						for globally rightmost point in $\Lambda_{\epsilon}(T)$.}	
	\vskip 1.2ex
	\STATE{$z_0 \; \gets$ an eigenvalue of $T$.}\label{line_alg2m:zm1}
			
	\vskip 1.2ex		
			
	\STATE{$x, y \; \gets$ unit right, left eigenvectors corr. to
				rightmost eigenvalue of $T$.}	\label{line_alg2m:z0}
				
	\vskip 1.5ex
	
	\STATE{$y \; \gets \; - \{ (y^\ast T'(z_0) x) / | y^\ast T'(z_0) x| \} y $.}\label{line_alg2m:normal1}
	
	\vskip 1.7ex
	
	\STATE{$\Delta^{(0)}
			\gets
			 y x^\ast$.} 
			\label{line_alg2m:Del0}
			
	\vskip 1.5ex
	
	\FOR{$k=1,2,\dots$}
	
	\vskip 1.2ex
	
	\STATE{$z_k \; \gets \;$ rightmost eigenvalue of 
	$\left\{ T(\lambda) / \sqrt{ \sum_{\ell = 1}^\kappa w_\ell^2 | t_\ell(\lambda) |^2} \right\}
		+  \epsilon \Delta^{(k-1)}$.}\label{line_alg2m:cons_per}
	
	\vskip 1.5ex
	
	\STATE{\textbf{If} $| z_{k} - z_{k-1} | < \mathsf{tol} \;$ \textbf{return} 
						$z \gets z_k$, $f \gets \text{Re}(z_k)$.}\label{line:fp2_nle_ter}
						
	\vskip 1.6ex
	
	\STATE{$v_k, u_k \; \gets \;$ unit consistent right, left
					singular vectors corr. to $\sigma_{\min}(T(z_k))$.}\label{line:fp2_nle_sval}

	\vskip 1.8ex
	
	\STATE{$u_k \; \gets \; - \left\{ [u_k^\ast M^D(z_k) v_k  ]  
							/ | u_k^\ast M^D(z_k) v_k  | \right\} u_k $.}\label{line_alg2m:normal2}
	
	\vskip 2.3ex
	
	\STATE{$\Delta^{(k)}
			\gets
			u_k v_k^\ast$}

	\vskip 1.5ex
	
	\ENDFOR
\end{algorithmic}
\caption{Fixed-point iteration for the
pseudospectral abscissa of a matrix-valued function based on constant perturbations}
\label{alg:fixed_point2m}
\end{algorithm} 	

As Algorithm \ref{alg:fixed_point1m} is a special case of Algorithm \ref{alg:fixed_point},
the analysis concerning the fixed-points of Algorithm \ref{alg:fixed_point} in Section
\ref{sec:fp_analysis} applies to Algorithm \ref{alg:fixed_point1m} as well. In particular
Theorem \ref{thm:fp_main_result} still holds, but now in the definition of a nondegenerate
point the  condition ${\mathcal S}(z) \neq 0$ (see Definition \ref{defn:nondegenerate}, part 2.)  
simplifies as $u^\ast T'(z) v \neq 0$, in the definition of an rbvt point the condition
${\mathcal S}(z)$ is real and positive (see Definition \ref{defn:rbvt_point}, part 3.)  
simplifies as $u^\ast T'(z) v$ is real and positive (as the denominator in (\ref{eq:defn_M})
is constant one in this special case), and the fixed-point function $\zeta$, letting
$\widetilde{z} = \zeta(z)$,  simplifies as follows:
\begin{enumerate}
	\item Let $v, u$ be unit consistent right, left
							singular vectors corresponding to $\sigma_{\min}(T(z))$.
							
	\item Set $\widetilde{u} := - \left\{ u^\ast T'(z) v / | u^\ast T'(z) v  | \right\} u $.
	\item Set  
			$\Delta   :=	\widetilde{u} v^\ast$.
	\item $\widetilde{z}$ is the rightmost eigenvalue of $T(\lambda) + \epsilon \Delta$.
\end{enumerate}
In item 4., if there are multiple rightmost eigenvalues, one of them can be taken
as explained before. 

As for Algorithm \ref{alg:fixed_point2m},
the associated fixed-point function that we now denote by $\zeta_3$ (to distinguish
the fixed-point functions associated with Algorithm \ref{alg:fixed_point},
Algorithm \ref{alg:fixed_point1m} explicitly), letting $\widetilde{z} = \zeta_3(z)$,
is defined as follows:
\begin{enumerate}
	\item Let $v, u$ be unit consistent right, left
							singular vectors corresponding to $\sigma_{\min}(T(z))$.
							
	\item Set $\widetilde{u} := - \left\{ u^\ast M^D(z) v / | u^\ast M^D(z) v  | \right\} u $.
	\item Set  
			$\Delta   :=	\widetilde{u} v^\ast$.
	\item $\widetilde{z}$ is the rightmost eigenvalue of 
	$\{ T(\lambda) / \sqrt{ \sum_{\ell = 1}^\kappa w_\ell^2 | t_\ell(\lambda) |^2} \} + \epsilon \Delta$.
\end{enumerate}
We state the fixed-point results for $\zeta_3$ formally below. Its proof is similar
to that for Theorem \ref{thm:fp_main_result} concerning the fixed-point function $\zeta$
associated with Algorithm \ref{alg:fixed_point}. The only difference in the proof is that
equations (\ref{eq:fp1_der_int}),
(\ref{eq:fp1_der_int2}), (\ref{eq:fp1_der_int3}) now involve $\{ (M + \epsilon \Delta)(z) \} v  = 0$
or $\{ (M + \epsilon \Delta)(z) \} \widehat{v}  = 0$
with $\Delta = \widetilde{u} v^\ast \:$ (rather than $\{ (T + \epsilon \Delta T)(z) \} v = 0$
or $\{ (T + \epsilon \Delta T)(z) \} \widehat{v} = 0$ with 
$\Delta T(\lambda) = 
			\sum_{j = 1}^\kappa w_j t_j(\lambda) \Delta T_j \,$,
	$\,\Delta T_j = \big\{ w_j \, \overline{t_j(z)} \widetilde{u} v^\ast \big/ 
						\sqrt{\sum_{\ell=1}^\kappa w_\ell^2 | t_\ell(z) |^2} \big\}$
						for $j = 1, \dots, \kappa$). 
\begin{theorem}\label{thm:fp_main_result2}
Let $z \in \Lambda_{\epsilon}(T)$ be nondegenerate (in the sense of 
Definition \ref{defn:nondegenerate}, with ${\mathcal S}(z)$ in part 2 as in (\ref{eq:nond_opt_map})), and
$\zeta_3$ be the fixed-point  map associated with Algorithm \ref{alg:fixed_point2m} defined as above.
\begin{enumerate}
	\item If $z$ is not an rbvt point (in the sense of Definition \ref{defn:rbvt_point}) 
	in $\Lambda_{\epsilon}(T)$, then $z$ is not a fixed-point of $\zeta_3$.
	\item If $z$ is the unique globally rightmost point in $\Lambda_{\epsilon}(T)$,
	then $z$ is a fixed-point of $\zeta_3$.
\end{enumerate}
\end{theorem}
\noindent
According to Theorem \ref{thm:fp_main_result2}, just like Algorithm \ref{alg:fixed_point},
assuming Algorithm \ref{alg:fixed_point2m} converges to a nondegenerate point,
this point must be an rbvt point.

\begin{example}\label{ex:damp1_2}
\rm{Let us again consider the damping problem in Example \ref{ex:damp1},
in particular the matrix polynomial $P(\lambda) = \lambda^2 M + \lambda C_{\mathrm{int}} + K$
in that example, i.e., without external damping, that is with the damping parameter $\nu = 0$.
We apply Algorithms \ref{alg:fixed_point} and \ref{alg:fixed_point2m}
to compute the rightmost point in $\Lambda_{\epsilon}(P)$ for $\epsilon = 0.1$ and $\epsilon = 0.2$
with the weights $w_1 = w_2 = w_3 = 1$ and the tolerance ${\sf tol} = 10^{-10}$.

Both Algorithm \ref{alg:fixed_point} and Algorithm \ref{alg:fixed_point2m} are initialized
with $z_0$ equal to the eigenvalue with the largest imaginary part.
The two algorithms as well as the criss-cross algorithm \cite[Section 2.1]{MehM24} return 
the same globally rightmost points up to prescribed tolerances, specifically 
$0.3049280 + 7.7520368{\rm i}$ and $0.6614719 + 7.8301883{\rm i}$ for 
$\epsilon = 0.1$ and $\epsilon = 0.2$, respectively. These computed globally
rightmost points together with the boundary of $\Lambda_{\epsilon}(P)$ for 
$\epsilon = 0.1, 0.2$ are illustrated in Figure \ref{fig:poly_psa}.
Additionally, Table \ref{tab:nle_iterates} list the number of iterations needed
by Algorithm \ref{alg:fixed_point} and Algorithm \ref{alg:fixed_point2m} to reach the 
prescribed tolerance ${\sf tol} = 10^{-10}$, as well as the first few iterates of the 
algorithms until the first five decimal digits of the iterates become correct. In these
two applications of the algorithms with $\epsilon = 0.1$, $\epsilon = 0.2$,
Algorithm \ref{alg:fixed_point} requires fewer iterations for the prescribed accuracy
and seems to be converging faster. This is a general pattern we observe for
decent values of $\epsilon$. However, for larger values of $\epsilon$ close to $\sigma_{\min}(M)$
(note that $\Lambda_{\epsilon}(P)$ is unbounded for $\epsilon > \sigma_{\min}(M)$),
it seems that Algorithm \ref{alg:fixed_point2m} converges more reliably;
we refer to Section \ref{sec:num_examp} for numerical examples with larger $\epsilon$ values
for which Algorithm \ref{alg:fixed_point2m} converges accurately, while Algorithm \ref{alg:fixed_point} 
does not converge.}
\end{example}


\begin{figure}
\begin{floatrow}
\hskip -.4ex
\ffigbox{%
	\includegraphics[width = .28\textwidth]{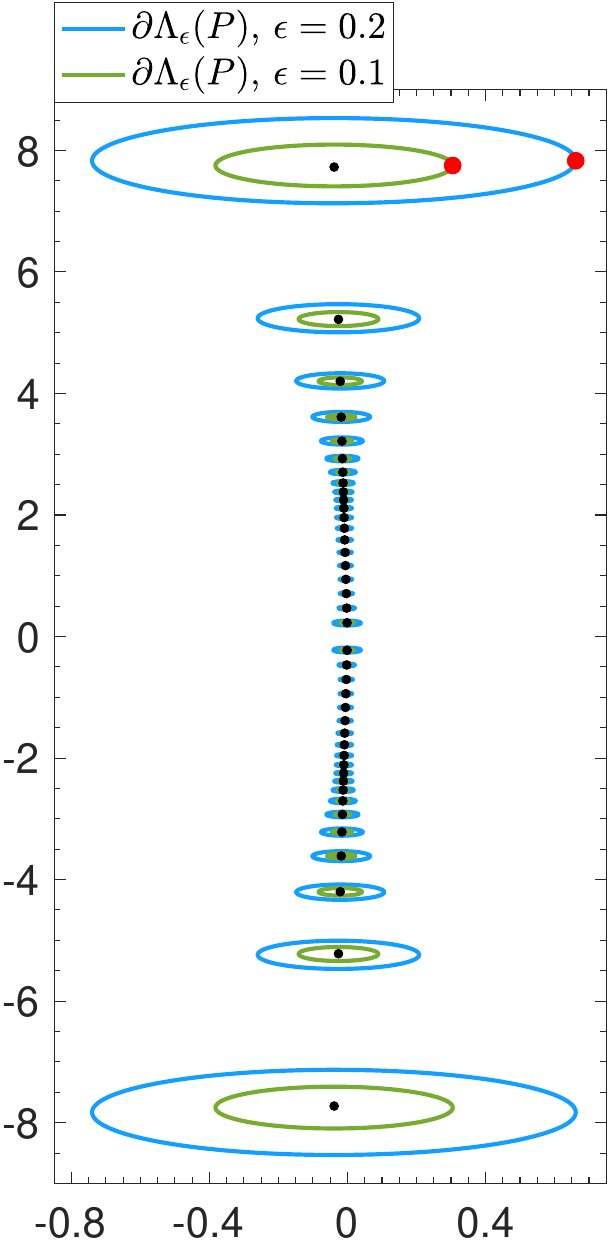} 
}{%
  \caption{  The boundary of $\Lambda_{\epsilon}(P)$ for the matrix polynomial
  $P$ in Example \ref{ex:damp1_2}. Black and red dots represent eigenvalues
  and computed rightmost points in $\Lambda_{\epsilon}(P)$. }
  \label{fig:poly_psa}
}
\hskip -.5ex
\capbtabbox{%
\small
\hskip -.7ex
\begin{tabular}{c}
    \begin{tabular}{|c||c|c|}
\hline
	$\epsilon$		&	Alg.~\ref{alg:fixed_point}	&	Alg.~\ref{alg:fixed_point2m}	\\
\hline
\hline
	0.1			&	7		&			10		\\
	0.2			&	9		&			13		\\
\hline
    \end{tabular}		\\[2.2em]
   \hskip -.7ex
    \begin{tabular}{|c||c|c|}
\hline
	$k$		&	Alg.~\ref{alg:fixed_point}		&	Alg.~\ref{alg:fixed_point2m}	\\
\hline
\hline
	0		& \hskip -.5ex	$-0.03863 + \underline{7.7}2651{\mathrm i}$	&	\hskip -.7ex $-0.03863 + \underline{7.7}2651{\mathrm i}$	\\
	1		& \hskip -.5ex	$\phantom{-}\underline{0.30}088 + \underline{7.7}0127{\mathrm i}$	&	\hskip -.7ex $\phantom{-}\underline{0.30}310 + \underline{7.7}3108{\mathrm i}$	\\
	2		& \hskip -.5ex	$\phantom{-}\underline{0.3049}2 + \underline{7.75204}{\mathrm i}$	&	\hskip -.7ex $\phantom{-}\underline{0.30}308 + \underline{7.75}194{\mathrm i}$	\\
	3		& \hskip -.5ex	$\phantom{-}\underline{0.30493} + \underline{7.75204}{\mathrm i}$	&	\hskip -.7ex $\phantom{-}\underline{0.30492} + \underline{7.75}193{\mathrm i}$	\\
	4		& \hskip -.5ex	$\phantom{-}\underline{0.30493} + \underline{7.75204}{\mathrm i}$	&	\hskip -.7ex $\phantom{-}\underline{0.30492} + \underline{7.75204}{\mathrm i}$	\\
\hline
    \end{tabular}		\\[4em]
    \hskip -.7ex
    \begin{tabular}{|c||c|c|}
\hline
	$k$		&	Alg.~\ref{alg:fixed_point}		&	Alg.~\ref{alg:fixed_point2m}	\\
\hline
\hline
	0		& \hskip -.5ex	$-0.03863 + \underline{7.7}2651{\mathrm i}$	&	\hskip -.7ex $-0.03863 + \underline{7.7}2651{\mathrm i}$	\\
	1		& \hskip -.5ex	$\phantom{-}\underline{0.6}2923  + \underline{7}.62747{\mathrm i}$	&	\hskip -.7ex $\phantom{-}\underline{0.6}4584 + \underline{7}.74547{\mathrm i}$	\\
	2		& \hskip -.5ex	$\phantom{-}\underline{0.661}28 + \underline{7.830}33{\mathrm i}$	&	\hskip -.7ex $\phantom{-}\underline{0.6}4621 + \underline{7.8}2829{\mathrm i}$	\\
	3		& \hskip -.5ex	$\phantom{-}\underline{0.66147} + \underline{7.8301}2{\mathrm i}$	&	\hskip -.7ex $\phantom{-}\underline{0.661}21 + \underline{7.8}2844{\mathrm i}$	\\
	4		& \hskip -.5ex	$\phantom{-}\underline{0.66147} + \underline{7.83019}{\mathrm i}$	&	\hskip -.7ex $\phantom{-}\underline{0.661}15 + \underline{7.8301}6{\mathrm i}$	\\	
	5		& \hskip -.5ex	$\phantom{-}\underline{0.66147} + \underline{7.83019}{\mathrm i}$	&	\hskip -.7ex $\phantom{-}\underline{0.66147} + \underline{7.8301}5{\mathrm i}$	\\
	6		& \hskip -.5ex	$\phantom{-}\underline{0.66147} + \underline{7.83019}{\mathrm i}$	&	\hskip -.7ex $\phantom{-}\underline{0.66147} + \underline{7.83019}{\mathrm i}$	\\
\hline
   \end{tabular}		
\end{tabular} 
}{%
  \caption{ Concerns applications of Alg.~\ref{alg:fixed_point}, Alg.~\ref{alg:fixed_point2m}
  		 to the matrix polynomial $P$ in Example \ref{ex:damp1_2}.
  		(Top) Number of iterations until termination.
  		(Middle) Iterates $z_k$, $\epsilon = 0.1$. 
		 (Bottom) Iterates $z_k$, $\epsilon = 0.2$.}%
	\label{tab:nle_iterates}
}
\end{floatrow}
\end{figure}

\section{Approximation of the pseudospectral abscissa of a matrix}\label{sec:psaA}
The $\epsilon$-pseudospectrum $\Lambda_{\epsilon}(A)$
of a matrix $A \in {\mathbb C}^{n\times n}$ is defined as in (\ref{eq:defn_ps}),
and the $\epsilon$-pseudospectral abscissa $\alpha_{\epsilon}(A)$ in (\ref{eq:defn_pspa})
is the real part of the rightmost point in $\Lambda_{\epsilon}(A)$.


Recalling the definition of $\Lambda_{\epsilon}(T)$, i.e., (\ref{eq:defn_NEP_pspec})
for a matrix-valued function $T$ as in (\ref{eq:mat_fun}), 
this matrix setting is only a special case of our treatment in the previous 
two sections; in particular, $\Lambda_{\epsilon}(A) = \Lambda_{\epsilon}(T)$ 
and $\alpha_{\epsilon}(A) = \alpha_{\epsilon}(T)$ for $T(\lambda) = \lambda I - A \;$ 
(that is $\kappa = 2$ and $t_1(\lambda) = \lambda$, $t_2(\lambda) = -1$, $T_1 = I$, $\, T_2 = A$ in (\ref{eq:mat_fun}))
$\;$with the weights $w_1 = 0$, $w_2 = 1$.

\subsection{First-order approximation}\label{sec:psa_fo_approx}
In this special matrix setting, we assume that every eigenvalue of $A$ is simple.
Let $\mu_0$ be an eigenvalue of $A$, and $U$ be 
as in Theorem \ref{thm:eig_well_posed}, i.e., $U$ is an open interval
containing 0 such that the $A + \eta \Delta$ has a unique eigenvalue
in $B_r(\mu_0)$, that is simple, for every $\eta \in U$ 
every $\Delta \in {\mathbb C}^{n\times n}$ such that $\| \Delta \|_2 \leq 1$
for some real number $r > 0$. Moreover, we assume $\epsilon$ is sufficiently small,
in particular $[0, \epsilon] \in U$.

For a given $\Delta \in {\mathbb C}^{n\times n}$ with $\| \Delta \|_2 \leq 1$, 
the function $\mu(\eta; \Delta, \mu_0)$ is defined as in Section \ref{sec:der_mu}, i.e.,
$\mu(\eta; \Delta, \mu_0)$ for $\eta \in U$ is the unique eigenvalue of
$A + \eta \Delta$ in $B_r(\mu_0)$. The eigenvalue function $\mu(\eta; \Delta, \mu_0)$
is analytic on $(0,\epsilon)$, continuous on $[0,\epsilon]$, and satisfies
$\mu(0; \Delta, \mu_0) = \mu_0$. 
In some occasions in this section, we also use the notation $\mu(\eta; \Delta, \mu_0, A)$ 
to make the dependence of this
analytic eigenvalue function on the matrix $A$ explicit. In this special setting, equations
(\ref{eq:der_mu}) and (\ref{eq:derL}) in Section \ref{sec:der_mu} concerning the derivatives 
of the eigenvalue function $\mu$ simplify as
\[
	\mu'(0 ; \Delta, \mu_0)
		\;	=	\;
	\frac{y^\ast \Delta x}{y^\ast x}
	\quad\;\;\;		\text{and}		\quad\;\;\;
	{\mathcal L}'_\Delta(0 ; \mu_0)
		\;	=	\;
	\mathrm{Re} \left\{ \mu'(0 ; \Delta, \mu_0) \right\}
		\;	=	\;
	\mathrm{Re} \left\{ \frac{y^\ast \Delta x}{y^\ast x} \right\}		\:	
\]
for a pair of unit right and unit left eigenvectors $x$ and $y$ of $A$ corresponding to $\mu_0$.
Moreover,
\[
	\max_{\Delta \in {\mathbb C}^{n\times n} \text{ s.t. } \| \Delta \|_2 \leq 1} \: {\mathcal L}'_\Delta(0 ; \mu_0)
		\;\;	=	\;\;
	\frac{1}{| y^\ast x |}
		\;\;	=	\;\;
	{\mathcal L}'_{\Delta_\ast}(0 ; \mu_0)		\;\;	, 	\quad	\text{where}	\;\;
	\Delta_\ast
		\;	=	\;
	 y x^\ast	\:	
\]
and, without loss of generality, we assume unit left, right eigenvectors $y$, $x$ are such that $y^\ast x$ is real
and positive. As a result, 
\begin{equation}\label{eq:RA_approx}
\begin{split}
	{\mathcal R}(\epsilon ; \mu_0)
		\;	&	=	\;
\max_{\Delta \in {\mathbb C}^{n\times n} \text{ s.t. } \| \Delta \|_2 \leq 1} \: {\mathcal L}_{\Delta}(\epsilon ; \mu_0)	\\[.3em]
		\;	&	=	\;
	\mathrm{Re}(\mu_0)	\: + \:	\epsilon \left( \frac{1}{| y^\ast x|} \right) \: + \: {\mathcal O}(\epsilon^2)	\:	.
\end{split}
\end{equation}
The main estimation result, that is Theorem \ref{thm:fo_estimate},
specialize into the following form. Here we remark that Assumption \ref{ass:psa_nlevp} 
holds trivially in the matrix setting, i.e., $\Lambda_{\epsilon}(A)$ is bounded, and 
$\Lambda(A)$, that is the set of eigenvalues of $A$, is finite, so it can be dropped.
\begin{theorem}\label{thm:psaA_est}
Suppose that $\epsilon > 0$ is sufficiently small so that $[0,\epsilon] \subset U \subset {\mathbb R}$ 
for an open interval $U$ satisfying the following condition: for every $\mu_0 \in \Lambda(A)$
there is a real number $r(\mu_0) > 0$
such that $A + \eta \Delta$
has only one eigenvalue in $B_{r(\mu_0)}(\mu_0)$, which is simple, for every $\eta \in U$ 
every $\Delta \in {\mathbb C}^{n\times n}$ such that $\| \Delta \|_2 \leq 1$.
Then, we have
\begin{equation*}
\begin{split}
	\alpha_{\epsilon}(A)  \;   & = \; \max \{ {\mathcal R}(\epsilon ; \mu_0) \; | \; \mu_0 \in \Lambda(A)	\}	\\
		& = \;
		\max_{\mu_0 \in \Lambda(A)}
						\left\{
							\mathrm{Re}(\mu_0)		+
			\epsilon	\left(	  \frac{1}{| y_{\mu_0}^\ast x_{\mu_0} |} 	\right)		
						\right\}
							\:	+	\:
						{\mathcal O}(\epsilon^2)		\:	,	
\end{split}
\end{equation*}
where $x_{\mu_0}$, $y_{\mu_0}$ denote a unit right eigenvector, a unit left eigenvector, respectively, 
of $A$ corresponding to its eigenvalue $\mu_0$.
\end{theorem}

\subsection{Second-order approximation}\label{sec:psa_so_approx}
We further derive a second-order approximation formula for $\alpha_{\epsilon}(A)$
in the matrix setting for a matrix $A \in {\mathbb C}^{n\times n}$
with an error ${\mathcal O}(\epsilon^3)$ below. The 
$\epsilon$-pseudospectrum of $A$ can alternatively be defined in terms of the
Frobenius norm as
\[
	\Lambda_{\epsilon}(A)	\;		:=	\;	
		\bigcup_{\Delta \in {\mathbb C}^{n\times n} \text{ s.t. } \| \Delta \|_F \leq \epsilon}	
		\:	\Lambda(A + \Delta) \: ,
\]
which turns out to be equivalent to the definition in terms of the 2-norm in (\ref{eq:defn_ps}). Here, we
employ the definition above in terms of the Frobenius norm, as it facilitates the derivation
due to the existence of the inner product $\langle F , G \rangle = \text{trace}(F^\ast G)$
for $F, G \in {\mathbb C}^{n\times n}$ such that $\| A \|_F = \sqrt{\langle A , A \rangle}$.


As in the previous subsection, let us consider the eigenvalue function
$\mu(\eta; \Delta) = \mu(0; \Delta, \mu_0)$ of 
$A + \eta \Delta$ for a prescribed  $\Delta \in {\mathbb C}^{n\times n}$
as a function of $\eta$. 
Now, we focus on $\Delta$ satisfying $\| \Delta \|_F \leq 1$, and assume that 
$\epsilon > 0$ is small enough so that the eigenvalue $\mu(\eta; \Delta)$ is analytic 
on an open interval $U$ such that $[0,\epsilon] \subset U$.
We would like to come up with a second-order approximation of the quantity
\begin{equation*}
	\widetilde{\mathcal R}(\epsilon; \mu_0)
			\;	:=	\;	
	\max_{\Delta \in \widetilde{\mathcal S}}	\:	\text{Re} \{ \mu(\epsilon; \Delta) \}
\]
with an approximation error ${\mathcal O}(\epsilon^3)$, where
$\widetilde{\mathcal S} := \{ \Delta \in {\mathbb C}^{n\times n} \; | \; \| \Delta \|_F \leq 1 \}	\:$.


We exploit the Taylor expansion
\begin{equation*}
\begin{split}
	\max_{\Delta \in \widetilde{\mathcal S}}	\:	\text{Re} \{ \mu(\epsilon; \Delta) \}
		\;	&	=	\;
	\max_{\Delta \in \widetilde{\mathcal S}}	\: \text{Re} \{ \mu_0 \}
				+
				\epsilon \text{Re} \{ \mu'(0; \Delta) \}
				+
				\frac{\epsilon^2}{2} \text{Re} \{ \mu''(0; \Delta) \}
				+	{\mathcal O}(\epsilon^3)		\\
		&	=	\;
	\text{Re} \{ \mu_0 \}	+
			\epsilon \left\{ \max_{\Delta \in \widetilde{\mathcal S}} \:  \text{Re} \{ \mu'(0; \Delta) \}
				+
				\frac{\epsilon}{2} \text{Re} \{ \mu''(0; \Delta) \}
				+	{\mathcal O}(\epsilon^2)	\right\}	\:	.
\end{split}
\end{equation*}
Hence, it suffices to approximate 
$\: \max_{\Delta \in \widetilde{\mathcal S}} \:  \text{Re} \{ \mu'(0; \Delta) \}
				+
				\frac{\epsilon}{2} \text{Re} \{ \mu''(0; \Delta) \} \:$
with an approximation error ${\mathcal O}(\epsilon^2)$.
There are real-analytic vector-valued functions $x(\eta; \Delta)$, $y(\eta; \Delta)$
satisfying $\| x(\eta; \Delta) \|_2 = \| y(\eta; \Delta) \|_2 = 1$ and
\[
	(A + \eta \Delta - \mu(\eta; \Delta)I) x(\eta; \Delta)	=	0	\;	,
	\quad
	y(\eta; \Delta)^\ast (A + \eta \Delta - \mu(\eta; \Delta)I)	=	0
\]
for all $\eta \in U$ \cite[pages 32-33]{Rel69}. We can assume, without loss of generality, that
$y(\eta; \Delta)^\ast x(\eta; \Delta)$ is real and positive for all $\eta \in U$; if not,
$y(\eta; \Delta)$ can be replaced by
\[
	\widetilde{y}(\eta; \Delta) 
	:= \big\{ y(\eta; \Delta)^\ast x(\eta; \Delta) / | y(\eta; \Delta)^\ast x(\eta; \Delta) | \big\} y(\eta; \Delta)
\]
-- an analytic function due to $y(\eta; \Delta)^\ast x(\eta; \Delta) \neq 0$ as $\mu(\eta; \Delta)$
is a simple eigenvalue -- $\;$so that $\widetilde{y}(\eta; \Delta)^\ast x(\eta; \Delta)$ is real and positive
for all $\eta \in U$. Set $x := x(0; \Delta)$, $y := y(0; \Delta)$, which are unit right, unit left
eigenvectors of $A$ corresponding to the eigenvalue $\mu_0$ such that $y^\ast x$ is real
and positive. Moreover, let us use the notations
 $x'_{\Delta} := x'(0; \Delta)$, $y'_{\Delta} := y'(0; \Delta)$. Observe
\begin{equation}\label{eq:intermed}
\begin{split}
	&
	\max_{\Delta \in \widetilde{\mathcal S}} \:  \text{Re} \{ \mu'(0; \Delta) \}
				\: + \:
				\frac{\epsilon}{2} \cdot \text{Re} \{ \mu''(0; \Delta) \}
				\;	=	\;		\\
	&
	\max_{\Delta \in \widetilde{\mathcal S}}	\:
	\text{Re}
	\left\{
		\frac{y^\ast \Delta x}{y^\ast x}
	\right\}
		\: +	\\
	&
	\phantom{aaaaaaaaaaa}
	\frac{\epsilon}{2}	\cdot
	\text{Re}
	\left\{
		\frac{(y'_\Delta)^\ast \Delta x}{y^\ast x}
				+
		\frac{y^\ast \Delta (x'_{\Delta})}{y^\ast x}
				-
		\left\{	(y'_\Delta)^\ast x + y^\ast (x'_\Delta \right)\} \frac{y^\ast \Delta x}{(y^\ast x)^2}
	\right\}	\;	=	\;	\\
	&
	\max_{\Delta \in \widetilde{\mathcal S}}	\:
	\frac{1}{y^\ast x}
	\text{Re}
	\left\{
	\left\langle
		y x^\ast
			+
		\frac{\epsilon}{2}
		\left\{
		(y'_\Delta) x^\ast	 + y (x'_\Delta)^\ast	+	(\beta_{\Delta})	
										y x^\ast \right\}	\:	,	\:	\Delta
	\right\rangle
	\right\}	\:	,
\end{split}
\end{equation}		
where
$
	\beta_{\Delta}	 :=	\frac{-1}{y^\ast x}\left\{	(y'_\Delta)^\ast x + y^\ast (x'_\Delta) \right\}
$,		
recalling also that the inner product $\langle \cdot , \cdot \rangle$ is defined by
$\langle F , G \rangle := \text{trace}(F^\ast G)$ for $F, G \in {\mathbb C}^{n\times n}$.
The matrix $\Delta \in \widetilde{\mathcal S}$ maximizing the expression in the last line above
is of unit Frobenius norm, and of the form
\[
	\Delta_\ast	\;	=	\;	
	\frac{yx^\ast + {\mathcal O}(\epsilon)\;\; }{\| yx^\ast + {\mathcal O}(\epsilon) \|_F}
		\;	=	\;
	yx^\ast	+	{\mathcal O}(\epsilon)	\:	.
\]	
For $h = O(\epsilon)$, we approximate the derivative of the right eigenvector in this direction by
\begin{equation*}
\begin{split}
	x'_{\Delta_\ast}	\;	=	\;	\frac{x(h; \Delta_\ast) - x}{h}	+	{\mathcal O}(h)
			\;	&	=	\;
	\frac{x(h; y x^\ast) + {\mathcal O}(h\epsilon) - x}{h}	+	{\mathcal O}(h)	\\
			&	=	\;
	\frac{x(h; y x^\ast) - x}{h}	+	{\mathcal O}(\epsilon)	\\
			&	=	\;
				\widetilde{x}_p	+	{\mathcal O}(\epsilon)	\:\;	,
				\quad	\text{with}	\;\;\;
			\widetilde{x}_p
				:=
			\frac{x(h; y x^\ast) - x}{h}	\:	.
\end{split}
\end{equation*}
Note that $x(h; \Delta_\ast) = x(h; yx^\ast) + {\mathcal O}(h\epsilon)$, since $x(h; \Delta_\ast)$
is the eigenvector of 
$A + h \Delta_\ast = A + h ( yx^\ast + {\mathcal O}(\epsilon)) = A + hyx^\ast + {\mathcal O}(h \epsilon)$,
whereas $x(h; yx^\ast)$ is the eigenvector of $A + h y x^\ast$.
Similarly,
\[
	y'_{\Delta_\ast}	\;	=	\;	\widetilde{y}_p	+	{\mathcal O}(\epsilon)	
			\:\;	,
				\quad	\text{with}	\;\;\;
			\widetilde{y}_p
				:=
			\frac{y(h; y x^\ast) - y}{h}	\:	.
\]


It follows from (\ref{eq:intermed}) that
\begin{equation*}
\begin{split}
	&
	\max_{\Delta \in \widetilde{\mathcal S}} \:  \text{Re} \{ \mu'(0; \Delta) \}
				+
				\frac{\epsilon}{2} \cdot \text{Re} \{ \mu''(0; \Delta) \}
				\;	=	\;		\\
	&
	\frac{1}{y^\ast x}
	\text{Re}
	\left\{
	\left\langle
		y x^\ast
			+
		\frac{\epsilon}{2}
		\left\{
			(\widetilde{y}_p) x^\ast	 + y (\widetilde{x}_p)^\ast	+	
								(\widetilde{\beta}_p)	y x^\ast \right\}	\:	,	\:	\Delta_\ast
	\right\rangle
	\right\}	\:	+	\:	{\mathcal O}(\epsilon^2)	\:	=	\;			\\
	&
	\max_{\Delta \in \widetilde{\mathcal S}}	\:
	\frac{1}{y^\ast x}
	\text{Re}
	\left\{
	\left\langle
		y x^\ast
			+
		\frac{\epsilon}{2}
		\left\{
			(\widetilde{y}_p) x^\ast	 + y (\widetilde{x}_p)^\ast	+	
								(\widetilde{\beta}_p)	y x^\ast \right\}	\:	,	\:	\Delta
	\right\rangle
	\right\}	\:	+	\:	{\mathcal O}(\epsilon^2)	\:	,
\end{split}
\end{equation*}
where 
$\widetilde{\beta}_p :=	\frac{-1}{y^\ast x}\left\{	(\widetilde{y}_p)^\ast x + y^\ast (\widetilde{x}_p) \right\}$.
Clearly, $\Delta$ of unit Frobenius norm maximizing the expression in the last line is
\begin{equation}\label{eq:near_opt_pert}
	\widetilde{\Delta}_\ast	
		\;	:=	\;
		\frac{
		y x^\ast
			+
		\displaystyle
		\frac{\epsilon}{2}
		\left\{
			(\widetilde{y}_p) x^\ast	 + y (\widetilde{x}_p)^\ast	+	
								\widetilde{\beta}_p	y x^\ast \right\}
			\:
		 \;\; }
		{\left\|
		y x^\ast
			+
		\displaystyle
		\frac{\epsilon}{2}
		\left\{
			(\widetilde{y}_p) x^\ast	 + y (\widetilde{x}_p)^\ast	+	
								\widetilde{\beta}_p	y x^\ast \right\}
			\:
			\right\|_F}	\:	.
\end{equation}


To summarize, we have
\begin{equation*}
\begin{split}
&	\quad\quad\quad\quad
	\max_{\Delta \in \widetilde{\mathcal S}}	\:	\text{Re} \{ \mu(\epsilon; \Delta) \}
		\;\;\;		=	\;	\\
&	\text{Re} \{ \mu_0 \}
				+
	\frac{\epsilon}{y^\ast x}
	\text{Re}
	\left\{
	\left\langle
		y x^\ast
			+
		\frac{\epsilon}{2}
		\left\{
			(\widetilde{y}_p) x^\ast	 + y (\widetilde{x}_p)^\ast	+	
				(\widetilde{\beta}_p)	y x^\ast \right\}	\:	,	\:	\widetilde{\Delta}_\ast
	\right\rangle
	\right\}		+		{\mathcal O}(\epsilon^3)	\:	=		\\[.2em]
&	\text{Re} \{ \mu_0 \}
				+
	\frac{\epsilon}{y^\ast x}
	\text{Re}
	\left\{
	\left\langle
		y x^\ast
			+
		\frac{\epsilon}{2}
		\left\{
			(y'_{\widetilde{\Delta}_\ast}) x^\ast	 + y (x'_{\widetilde{\Delta}_\ast})^\ast	+	
				( \beta_{\widetilde{\Delta}_\ast} )	y x^\ast  + {\mathcal O}(\epsilon) \right\}	\:	,	\:	\widetilde{\Delta}_\ast
	\right\rangle
	\right\}		+		{\mathcal O}(\epsilon^3)		\: 	=	\\[.4em]
&
	\text{Re} \{ \mu_0 \}
				+
	\frac{\epsilon}{y^\ast x}
	\text{Re}
	\left\{
	\left\langle
		y x^\ast
			+
		\frac{\epsilon}{2}
		\left\{
			(y'_{\widetilde{\Delta}_\ast}) x^\ast	 + y (x'_{\widetilde{\Delta}_\ast})^\ast	+	
				( \beta_{\widetilde{\Delta}_\ast} )	y x^\ast  \right\}	\:	,	\:	\widetilde{\Delta}_\ast
	\right\rangle
	\right\}		+		{\mathcal O}(\epsilon^3)	\:	=			\\[.1em]
&
	 \text{Re} \{ \mu_0 \}
				+
	\epsilon \text{Re} \{ \mu'(0; \widetilde{\Delta}_\ast) \}
				+
	\frac{\epsilon^2}{2} \text{Re} \{ \mu''(0; \widetilde{\Delta}_\ast) \}
				+	
			{\mathcal O}(\epsilon^3)		
					\: =	\:
	\text{Re} \{ \mu(\epsilon; \widetilde{\Delta}_\ast) \}	+	{\mathcal O}(\epsilon^3)	
\end{split}
\end{equation*}
for $\widetilde{\Delta}_\ast$ as in (\ref{eq:near_opt_pert}). These arguments lead to the
following second-order approximation result for $\alpha_{\epsilon}(A)$. Note that we use the
notation $\alpha(F)$ to denote the spectral abscissa of a square matrix $F$ in the result. 
\begin{theorem}\label{thm:mat_sec_ord_approx}
Let $U \subset {\mathbb R}$ be an open interval containing 0 as in
Theorem \ref{thm:psaA_est}  but for $\Delta \in {\mathbb C}^{n\times n}$ such that $\| \Delta \|_F \leq 1$.
Suppose also $\epsilon > 0$ is small enough so that $[0,\epsilon] \subset U$.
Then, we have
\begin{equation}\label{eq:sec_ord_thm1}
\begin{split}
	&  \alpha_{\epsilon}(A)  \;    = \; 
			\max \{ \widetilde{{\mathcal R}}(\epsilon ; \mu_0) \; | \; \mu_0 \in \Lambda(A)	\}	\\[.4em]
	&  \phantom{\alpha_{\epsilon}(A)}  \:    = 
			\max_{\mu_0 \in \Lambda(A)}
				\bigg\{
					\mathrm{Re} \{ \mu(\epsilon; \widetilde{\Delta}^{\mu_0}_{\ast} , \mu_0) \}
				\bigg\}  +	{\mathcal O}(\epsilon^3)	\\[.4em]
	&  \phantom{\alpha_{\epsilon}(A)}  \:    = 
		\max_{\mu_0 \in \Lambda(A)}
		\bigg\{
		\alpha\left(
			A + \epsilon  \widetilde{\Delta}^{\mu_0}_{\ast} \right)
		\bigg\}  +	{\mathcal O}(\epsilon^3)		\\[.7em]
	&	 \phantom{\alpha_{\epsilon}(A)}
	\:	 =	
	\max_{\mu_0 \in \Lambda(A)}
		\bigg\{
			\mathrm{Re} \{ \mu_0 \}
				\; + \;
				\frac{\epsilon}{y_{\mu_0}^\ast x_{\mu_0}}
				\mathrm{Re}
				\bigg[
					\bigg\langle
					y_{\mu_0} x_{\mu_0}^\ast
						\; +		\hskip 33ex		\\[.3em]
	&
		\hskip 18.8ex
		\frac{\epsilon}{2}
		\left\{
			(\widetilde{y}^{\, \mu_0}_{p}) x_{\mu_0}^\ast	 + y_{\mu_0} (\widetilde{x}^{\, \mu_0}_{p})^\ast	+	
				(\widetilde{\beta}^{\, \mu_0}_{p})	y_{\mu_0} x_{\mu_0}^\ast 
		\right\}		,
																\widetilde{\Delta}_{\mu_0 , \ast}
	\bigg\rangle
	\bigg]	
						\bigg\}		
								+	
						{\mathcal O}(\epsilon^3)			,
\end{split}
\end{equation}
where $x_{\mu_0}$, $y_{\mu_0}$ denote a unit right eigenvector, a unit left eigenvector, respectively, 
of $A$ corresponding to its eigenvalue $\mu_0$ normalized such that $y_{\mu_0}^\ast x_{\mu_0}$ is real and positive,
\begin{equation*}
\begin{split}
	&
	\widetilde{x}^{\, \mu_0}_{p}
				:=
			\frac{x(h; y_{\mu_0} x_{\mu_0}^\ast, \mu_0) - x_{\mu_0}}{h}	\:	,	\quad
	\widetilde{y}^{\, \mu_0}_{p}
				:=
			\frac{y(h; y_{\mu_0} x_{\mu_0}^\ast, \mu_0) - y_{\mu_0}}{h}	\:	,	\\[.3em]
	&
	\widetilde{\beta}^{\, \mu_0}_{p} :=	\frac{-1}{y_{\mu_0}^\ast x_{\mu_0}}
	\left\{	(\widetilde{y}^{\, \mu_0}_{p})^\ast x_{\mu_0} + y_{\mu_0}^\ast (\widetilde{x}^{\, \mu_0}_{p}) \right\}	\:	
\end{split}
\end{equation*}
for some positive $h = {\mathcal O}(\epsilon)$, while
$x(\eta; \Delta,\mu_0)$, $y(\eta; \Delta,\mu_0)$ denote analytic unit right, unit left 
eigenvectors of $A + \eta \Delta$
such that $y(\eta; \Delta,\mu_0)^\ast x(\eta; \Delta,\mu_0)$ is real and positive,
$x(0; \Delta,\mu_0) = x_{\mu_0}$, $y(0; \Delta,\mu_0) = y_{\mu_0}$, and
\begin{equation}\label{eq:so_perturbation}
	\widetilde{\Delta}^{\mu_0}_{\ast}	
		\;	:=	\;
		\frac{
		y_{\mu_0} x_{\mu_0}^\ast
			+
		\displaystyle
		\frac{\epsilon}{2}
		\left\{
			(\widetilde{y}^{\, \mu_0}_{p}) x_{\mu_0}^\ast	 + 
					y_{\mu_0} (\widetilde{x}^{\, \mu_0}_{p})^\ast	+	
								\widetilde{\beta}^{\, \mu_0}_{p}	y_{\mu_0} x_{\mu_0}^\ast \right\}
			\:
		 \;\; }
		{\left\|
		y_{\mu_0} x_{\mu_0}^\ast
			+
		\displaystyle
		\frac{\epsilon}{2}
		\left\{
			(\widetilde{y}^{\, \mu_0}_{p}) x^\ast	 + y (\widetilde{x}^{\, \mu_0}_{p})^\ast	+	
								\widetilde{\beta}^{\, \mu_0}_p	y_{\mu_0} x_{\mu_0}^\ast \right\}
			\:
			\right\|_F}	\:	.
\end{equation}
\end{theorem}
\begin{proof}
The proof of the first equality in (\ref{eq:sec_ord_thm1}) is the same as the proof of the first equality in 
Theorem \ref{thm:fo_estimate} specialized for a matrix $A$ in place of the matrix-valued function $T$, 
and the 2-norm replaced by the Frobenius norm. The second and fourth equalities in (\ref{eq:sec_ord_thm1})
are immediate from the derivation in this subsection, while the third equality in (\ref{eq:sec_ord_thm1})
follows from the inequalities
\[
	\alpha_{\epsilon}(A)
		\;	\geq	\;
	\max_{\mu_0 \in \Lambda(A)}
		\bigg\{
		\alpha\left(
			A + \epsilon  \widetilde{\Delta}^{\mu_0}_{\ast} \right)
		\bigg\} 
			\;	\geq	\;
		\max_{\mu_0 \in \Lambda(A)}
		\bigg\{
		\mathrm{Re} \{ \mu(\epsilon; \widetilde{\Delta}^{\mu_0}_{\ast} , \mu_0) \}
		\bigg\}
\]
combined with the second equality in (\ref{eq:sec_ord_thm1}).
\end{proof}

\begin{remark}\label{rmk:so_approx_prac}
Arguably the most useful approximation formula from Theorem \ref{thm:mat_sec_ord_approx} in practice is
\begin{equation}\label{eq:so_approx_formula}
	\alpha_{\epsilon}(A)
		\;	\approx	\;
	\max_{\mu_0 \in \Lambda(A)}
		\bigg\{
		\alpha\left(
			A + \epsilon  \widetilde{\Delta}^{\mu_0}_{\ast} \right)
		\bigg\}  
\end{equation}
with an approximation error ${\mathcal O}(\epsilon^3)$.
This requires the computation of the rightmost eigenvalue of $A + \epsilon  \widetilde{\Delta}^{\mu_0}_\ast$
for every $\mu_0 \in \Lambda(A)$, which can be achieved by an iterative method such as a Krylov
subspace method, e.g., {\sf eigs} in MATLAB. 


From the arguments that give rise to Theorem \ref{thm:mat_sec_ord_approx}, 
assuming $A$ has simple eigenvalues and $\epsilon$ is small enough,
the eigenvalue $\mu_0 \in \Lambda(A)$ maximizing the right-hand side
of (\ref{eq:so_approx_formula}) is the eigenvalue $\mu_0 \in \Lambda(A)$ maximizing 
$\widetilde{{\mathcal R}}(\epsilon ; \mu_0)$, that is the eigenvalue leading to
the rightmost point in $\Lambda_{\epsilon}(A)$.
Moreover, for this choice of $\mu_0 \in \Lambda(A)$, the arguments before
Theorem \ref{thm:mat_sec_ord_approx} show that 
$\alpha_{\epsilon}(A) = \widetilde{{\mathcal R}}(\epsilon ; \mu_0) \approx 
		\mathrm{Re} \{ \mu(\epsilon; \widetilde{\Delta}^{\mu_0}_{\ast}, \mu_0)  \}$,
		again provided $\epsilon$ is small enough.
Also, $\mu(\epsilon; \widetilde{\Delta}^{\mu_0}_{\ast}, \mu_0)$ is an eigenvalue of 
$A + \epsilon  \widetilde{\Delta}^{\mu_0 }_{\ast}$, so is contained in the $\epsilon$-pseudospectrum,
and, as a result, is an estimate for the globally rightmost point in $\Lambda_{\epsilon}(A)$.
Hence, the rightmost eigenvalue of $A + \epsilon  \widetilde{\Delta}^{\mu_0}_{\ast}$
for $\mu_0$ maximizing the right-hand side
of (\ref{eq:so_approx_formula}) could possibly provide a good estimate for a
globally rightmost point in $\Lambda_{\epsilon}(A)$.


Note also that the computation of $\widetilde{\Delta}^{\mu_0}_{\ast}$ requires an additional
eigenvector computation for every $\mu_0 \in \Lambda(A)$, namely the right, left eigenvectors 
$x(h; y_{\mu_0} x_{\mu_0}^\ast, \mu_0)$, $y(h; y_{\mu_0} x_{\mu_0}^\ast, \mu_0)$
of $A + h y_{\mu_0} x_{\mu_0}^\ast$. Since $h$ can be chosen small (e.g., half of the double machine precision), 
this can be achieved by computing the eigenvalue of $A + h y_{\mu_0} x_{\mu_0}^\ast$ closest to $\mu_0$,
and corresponding eigenvectors. Recalling $\mu_0$ is an eigenvalue of $A$, and nearly 
an eigenvalue of $A + h y_{\mu_0} x_{\mu_0}^\ast$, this task is likely to be cheaper 
than computing a rightmost eigenvalue.
\end{remark}

\begin{example}\label{ex:psa_mat1}
\rm{To illustrate the accuracy of the first-order and second-order approximations for $\alpha_{\epsilon}(A)$
with respect to $\epsilon$, we experiment with two random matrices. These matrices are 
$100\times 100$, $200\times 200$, and generated by typing 
\begin{center}
\texttt{randn(100)+0.5*sqrt(-1)*randn(100)},
\texttt{0.5*randn(200)+2*sqrt(-1)*randn(200)}, 
\end{center}
respectively, in MATLAB.
Letting 
\[
	\mathcal{R}_{\epsilon}(A)	:=	
	\max_{\mu_0 \in \Lambda(A)}
						\left\{
			\mathrm{Re}(\mu_0)		+
			\epsilon	\left(	  \frac{1}{| y_{\mu_0}^\ast x_{\mu_0} |} 	\right)		
						\right\}
			\;\;\;	\text{and}		\;\;\;
	\mathcal{R}^{(2)}_{\epsilon}(A)	:=	
	\max_{\mu_0 \in \Lambda(A)}
	\bigg\{
		\alpha\left(
			A + \epsilon  \widetilde{\Delta}^{\mu_0}_{\ast} \right)
		\bigg\} 
\]
be first-order and second-order approximations, the errors $| \alpha_{\epsilon}(A) - \mathcal{R}_{\epsilon}(A) |$
and $| \alpha_{\epsilon}(A) - \mathcal{R}^{(2)}_{\epsilon}(A) |$ of these approximations for the two random
matrices are plotted in Figure \ref{fig:error_mat1} as a function of $\epsilon$. The plots are in logarithmic scale. 
The slopes of $| \alpha_{\epsilon}(A) - \mathcal{R}_{\epsilon}(A) |$ and 
$| \alpha_{\epsilon}(A) - \mathcal{R}^{(2)}_{\epsilon}(A) |$ in the plots appear to be two and three, respectively,
same as the slopes of $y = \epsilon^2$ and $y = \epsilon^3$ in the logarithmic scale (i.e., the slopes of
$\log \, y \: = \: 2 \log \, \epsilon$ and $\log \, y \: = \: 3 \log \, \epsilon$). The plots confirm that the approximation
errors of $\mathcal{R}_{\epsilon}(A)$ and $\mathcal{R}^{(2)}_{\epsilon}(A)$ for these two random matrices
are ${\mathcal O}(\epsilon^2)$ and ${\mathcal O}(\epsilon^3)$.}

\begin{figure}
	\begin{tabular}{cc}
		\hskip -3.5ex
			\subfigure[$100\times 100$ matrix example]{\includegraphics[width = .52\textwidth]{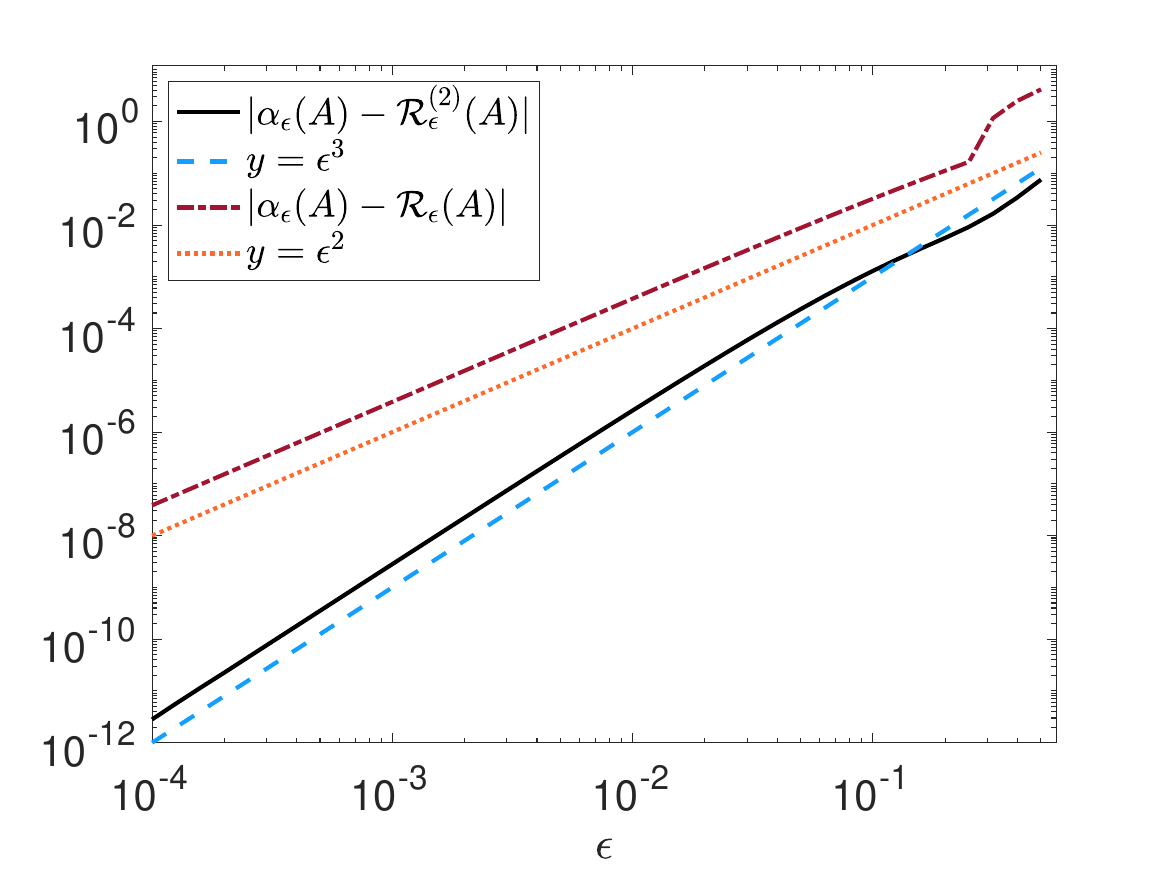}}   & 
			\hskip -1.5ex
			\subfigure[$200\times 200$ matrix example]{\includegraphics[width = .52\textwidth]{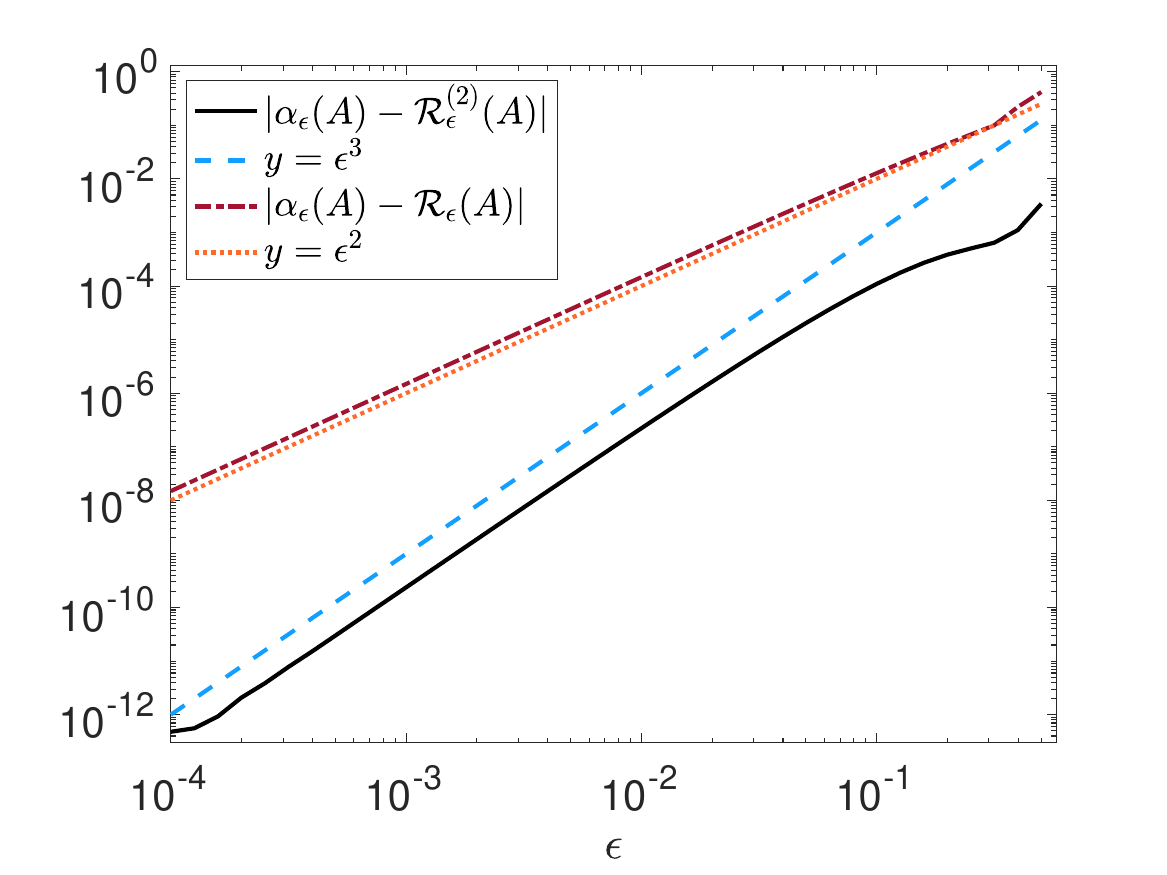}}	
	\end{tabular}
		\caption{The errors of the first-order approximation $\mathcal{R}_{\epsilon}(A)$
		and second-order approximation $\mathcal{R}^{(2)}_{\epsilon}(A)$
		for $\alpha_{\epsilon}(A)$ as a function of $\epsilon$ on two random matrices.}
		\label{fig:error_mat1}
\end{figure}
\end{example}

\begin{example}\label{ex:psa_mat2}
\rm{We consider a static output feedback stabilization problem $A + \nu BC^T$ with respect to the real 
parameter $\nu$ (controller), where $A \in {\mathbb R}^{1006\times 1006}$, $B, C \in {\mathbb R}^{1006}$ 
are taken from the NN18 example in the $COMPl_eib$ collection \cite{Lei04}. The aim is to find $\nu$ such that
$A + \nu BC^T$ has all of its eigenvalues on the left half of the complex plane. In \cite{AliM24}},
rather than minimizing the spectral abscissa, the $\epsilon$-pseudospectral abscissa of 
$A(\nu) := A + \nu BC^T$ for $\epsilon = 0.2$ is minimized over all $\nu \in [-1,1]$.
We depict the approximation of $\alpha_{\epsilon}(\nu)$ with $\mathcal{R}^{(2)}_{\epsilon}(\nu)$ for $\nu \in [-1,1]$
in Figure \ref{fig:error_mat2}. In the left-hand plot of the figure, it is not possible to distinguish
$\alpha_{\epsilon}(\nu)$ from its approximation $\mathcal{R}^{(2)}_{\epsilon}(\nu)$ for $\epsilon = 0.2$.
In the right-hand plot, the errors appear to be decreasing in accordance with ${\mathcal O}(\epsilon^3)$
as $\epsilon$ is reduced from 0.2 to 0.05.
\begin{figure}
	\begin{tabular}{cc}
		\hskip -3.5ex
			\includegraphics[width = .52\textwidth]{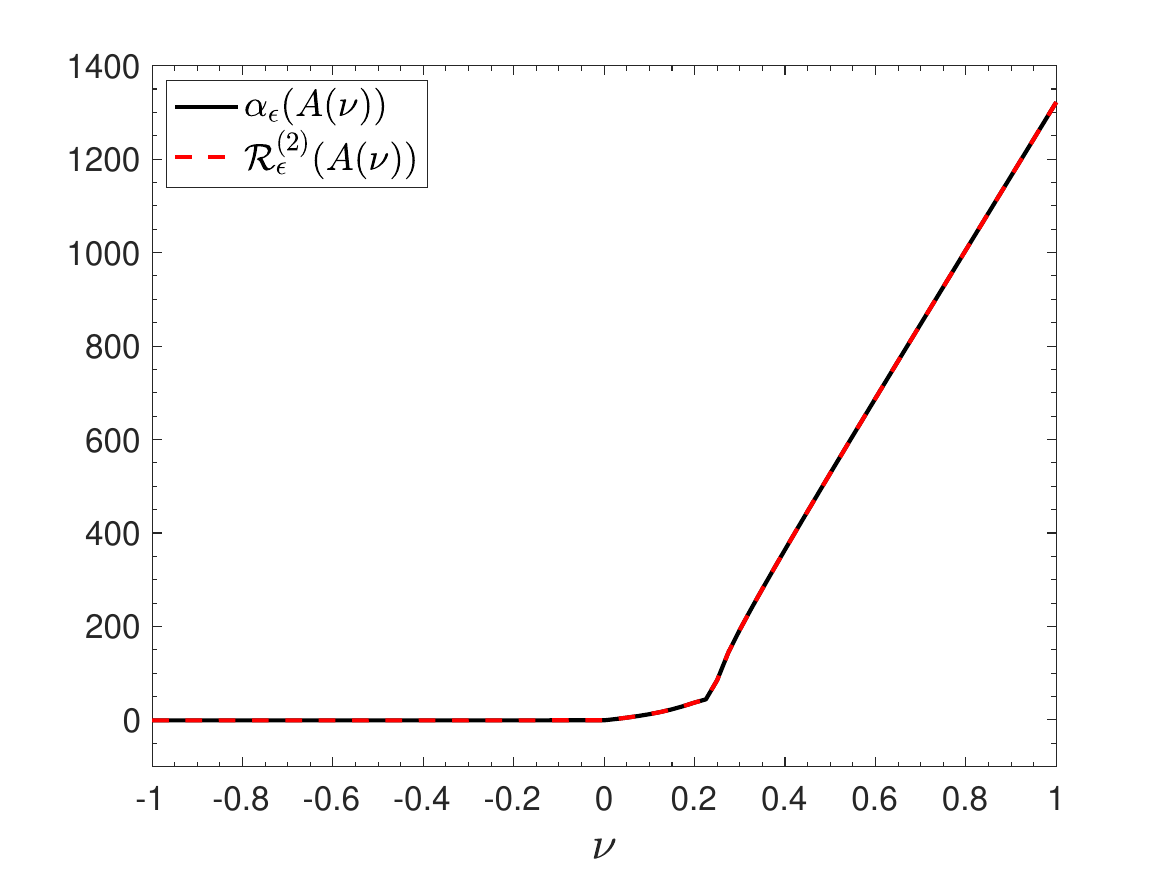} & 
			\hskip -1.5ex
			\includegraphics[width = .52\textwidth]{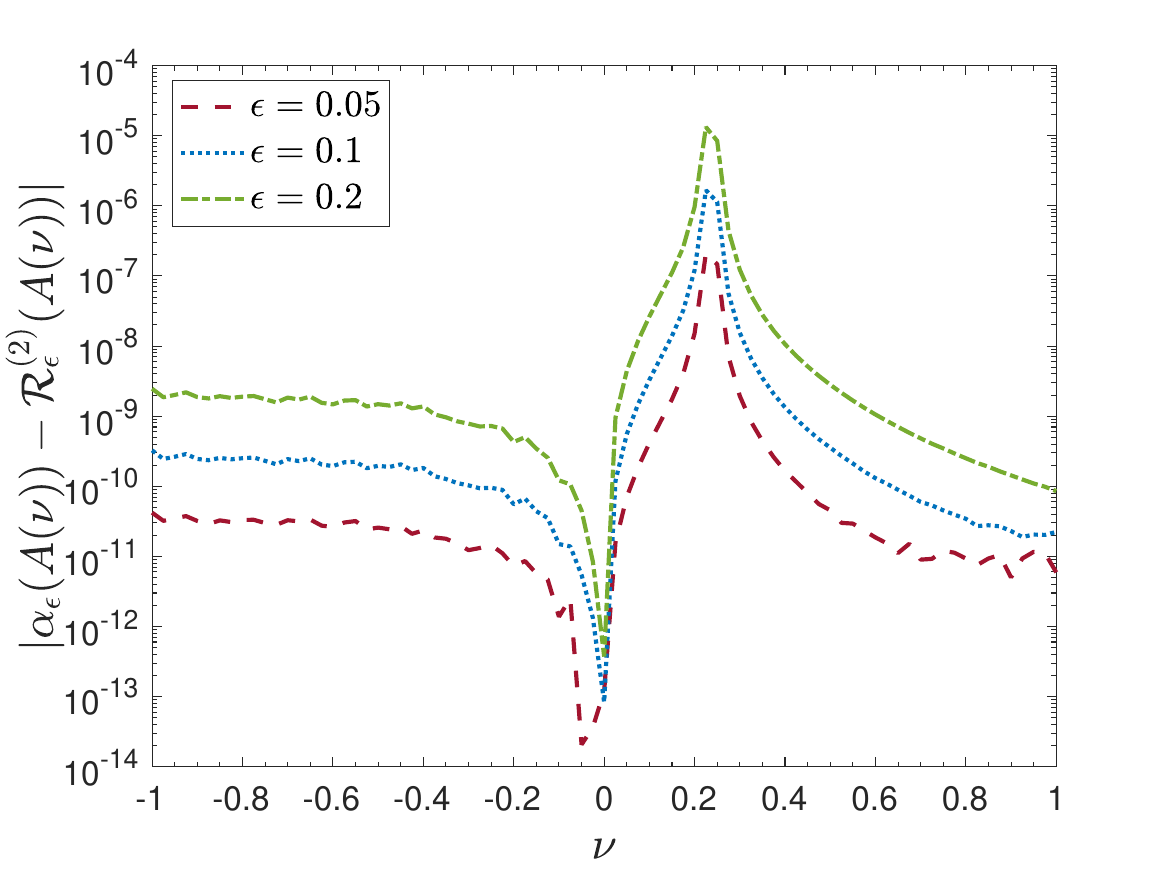} 
	\end{tabular}
		\caption{The approximation of $\alpha_{\epsilon}(A(\nu))$ with $\mathcal{R}^{(2)}_{\epsilon}(A(\nu))$
		for $\nu \in [-1,1]$,
		where $A(\nu) = A + \nu BC^T$ is the NN18 example from the $COMPl_eib$ collection. In the
		left-plot, $\epsilon = 0.2$.}
		\label{fig:error_mat2}
\end{figure}
\end{example}

\subsection{Fixed-point iteration for the $\epsilon$-pseudospectral abscissa of a matrix}
The first-order, second-order approximation of the previous two subsections yield estimates
of the $\epsilon$-pseudospectral abscissa of a matrix with errors 
${\mathcal O}(\epsilon^2)$, ${\mathcal O}(\epsilon^3)$, respectively.
For arbitrarily high accuracy, here, we briefly summarize the specialization of the fixed-point iterations 
in Section \ref{sec:fp_mat_val} to estimate the $\epsilon$-pseudospectral abscissa of a matrix-valued function
for a matrix $A \in {\mathbb C}^{n\times n}$. 
The approximation ideas of the previous two subsections can be employed to initialize
the fixed-point iteration below, but we postpone a proper discussion of this to Section \ref{sec:glob_opts}.


The matrix $A$ is initially perturbed by
$\Delta^{(0)} :=  y_0 x_0^\ast$, 
where $y_0, x_0$ are left, right
eigenvectors of $A$ corresponding to an eigenvalue $\lambda_R$ normalized so that
\[
	\|x_0 \|_2 = \| y_0 \|_2 = 1
	\quad	\text{and}		\quad
	y^\ast_0 x_0 \text{ is real and positive}	\:	,
\]
since
$\:
		\Delta^{(0)}	\, \in \,	
			\argmax
			\left\{
					\text{Re} \left\{ \mu'(0; \Delta, \lambda_R, A) \right\}
						\;	|	\;
					\Delta \in {\mathbb C}^{n\times n} \text{ s.t. } \| \Delta \|_2 \leq 1	
			\right\}
$.


Letting $z_1$ be the rightmost eigenvalue of $A + \epsilon \Delta^{(0)}$,
by the Eckart-Young-Mirsky theorem \cite[Theorem 2.5.3]{GolVL96},
the minimal perturbation $\underline{\Delta}^{(0)} \in {\mathbb C}^{n\times n}$ 
(with the smallest 2-norm possible)
such that $z_1$ is an eigenvalue of $A + \underline{\Delta}^{(0)}$ is given by
\[
	\underline{\Delta}^{(0)}	\;	:=	\;	-\gamma u_1 v_1^\ast
\] 
where $\gamma := \sigma_{\min}(z_1 I - A)$ and $u_1, v_1$ are corresponding
consistent unit left, right singular vectors, which also turn out to be left, right
eigenvectors of $A + \underline{\Delta}^{(0)}$ corresponding to its eigenvalue $z_1$. 
Normalizing the eigenvectors $u_1$, $v_1$ into  
$\widetilde{u}_1, \widetilde{v}_1$ so that
\[
	\| \widetilde{u}_1 \|_2 = \| \widetilde{v}_1 \|_2 = 1
	\quad	\text{and}		\quad
	\widetilde{u}_1^\ast \widetilde{v}_1 \text{ is real and positive}	,
\]
the matrix $\Delta^{(1)} := \widetilde{u}_1 \widetilde{v}_1^\ast$ satisfies
\[
		\Delta^{(1)}	 \in 
		\argmax
		\left\{	
				\text{Re} \left\{ \mu'(0; \Delta, z_1, A + \underline{\Delta}^{(0)}) \right\}		
					\:	|	\:
					\Delta \in {\mathbb C}^{n\times n} \text{ s.t. } \| \Delta \|_2 \leq 1	
		\right\}	\:	.
\]
Setting $z_2$ as the rightmost eigenvalue of $A + \epsilon \Delta^{(1)}$, we
repeat the procedure. The fixed-point iteration specialized
for the $\epsilon$-pseudospectral abscissa 
of a matrix is given in Algorithm \ref{alg:fixed_point_matrix} below.

\begin{algorithm}
\begin{algorithmic}[1]
	\REQUIRE{A matrix $A\in {\mathbb C}^{n\times n}$, a real number $\epsilon > 0$,
				tolerance for termination $\mathsf{tol} > 0$.}
	\ENSURE{Estimates $f$ for $\alpha_{\epsilon}(A)$ and $z$
						for globally rightmost point in $\Lambda_{\epsilon}(A)$.}	
	\vskip 1.2ex
	\STATE{$z_0 \; \gets$ an eigenvalue of $A$.}\label{line_alg4A:zm1}
	
	\vskip 1.2ex
	
	\STATE{$x, y \; \gets$ unit right, left eigenvectors corr. to
				rightmost eigenvalue of $A$.}
	\vskip 1.2ex
	\STATE{$y \; \gets \;\,  \{ (y^\ast x) / | y^\ast x| \} y $.}
	\vskip 1.2ex
	\STATE{$\Delta^{(0)} \; \gets \; y x^\ast$.}\label{line_alg4:D0}
	\vskip 1.2ex
	\FOR{$k=1,2,\dots$}
	\vskip 1.2ex
	\STATE{$z_k \; \gets \;$ rightmost eigenvalue of $A + \epsilon \Delta^{(k-1)}$.}
	\vskip 1.2ex
	\STATE{\textbf{If} $| z_{k} - z_{k-1} | < \mathsf{tol} \;$ \textbf{return} 
						$z \gets z_k$, $f \gets \text{Re}(z_k)$.}
	\vskip 1.2ex
	\STATE{$v_k, u_k \; \gets \;$ unit consistent right, left
							singular vectors corr. to $\sigma_{\min}(z_k I_n - A)$.} 
	\vskip 1.2ex
	\STATE{$u_k \; \gets \;\, \{ ( u_k^\ast v_k ) / | u_k^\ast v_k | \}  u_k $.}
	\vskip 1.2ex
	\STATE{$\Delta^{(k)} \gets u_k v_k^\ast$.}
	\vskip 1.5ex
	\ENDFOR
\end{algorithmic}
\caption{Fixed-point iteration for the pseudospectral abscissa of a matrix}
\label{alg:fixed_point_matrix}
\end{algorithm}

\subsubsection{Fixed-points of Algorithm \ref{alg:fixed_point_matrix}.}
Consider Algorithm \ref{alg:fixed_point} applied to $T(\lambda) = \lambda I - A$
(i.e., $t_1(\lambda) = \lambda$, $t_2(\lambda) = -1$ and $T_1 = I$, $T_2 = A$)
with the weights $w_1 = 0$, $w_2 = 1$. The fixed-point map $\zeta$ in 
Section \ref{sec:fp_analysis} in this special case, letting $\widetilde{z} = \zeta(z)$,
is defined as follows:
\begin{enumerate}
	\item Let $v, u$ be consistent unit right, left
							singular vectors corresponding to $\sigma_{\min}(zI - A)$.
							
	\item Set $\widetilde{u} := - \left\{ u^\ast  v  / | u^\ast v  | \right\} u $.
	\item Set  
			$\Delta  := -\widetilde{u} v^\ast$.
	\item $\widetilde{z}$ is the rightmost eigenvalue of 
	$T(\lambda) + \epsilon (- \Delta) = \lambda I - (A + \epsilon \Delta)$, that is the rightmost 
	eigenvalue of 
	$A - \epsilon \widetilde{u} v^\ast	= 	A  +   \epsilon \left\{ u^\ast  v  / | u^\ast v  | \right\} u v^\ast$.
\end{enumerate}
Now looking at one iteration of Algorithm \ref{alg:fixed_point_matrix},
for the associated fixed-point function $\zeta_4$,
the point $\widetilde{z} = \zeta_4(z)$ is the rightmost eigenvalue
of $A  +   \epsilon \left\{ u^\ast  v  / | u^\ast v  | \right\} u v^\ast$. Hence,
the algorithms have the same fixed-point functions. It can also be verified 
that $z_1$ for both algorithms are the same, provided that the algorithms 
are started with the same $z_0$. It follows that
Algorithm \ref{alg:fixed_point_matrix} and Algorithm \ref{alg:fixed_point}
applied to $T(\lambda) = \lambda I - A$ with $w_1 = 0$, $w_2 = 1$
generate the same sequence $\{ z_k \}$. The next result is now
an immediate corollary of Theorem \ref{thm:fp_main_result}, noting that
$\Lambda_{\epsilon}(A) = \Lambda_{\epsilon}(T)$ for 
$T(\lambda) = \lambda I - A$ with $w_1 = 0$, $w_2 = 1$, and the definition
of an rbvt point simplifies as follows.
\begin{definition}\label{defn:rbvt_point_matrix} 
\hskip -.45ex
We call $z \in {\mathbb C}$ an rbvt (right boundary with a vertical tangent) point in $\Lambda_{\epsilon}(A)$ if 
\begin{enumerate}
	\item the smallest singular value of $zI - A$ is simple,
	\item $z$ is on the boundary of $\Lambda_{\epsilon}(A)$, and
	\item $u^\ast v$ is real and positive, where $u$, $v$ denote a pair of consistent
	left, right singular vectors corresponding to $\sigma_{\min}(zI - A)$. 
\end{enumerate} 
\end{definition}
\begin{theorem}\label{thm:convergence_matrix_case}
Let $z \in \Lambda_{\epsilon}(A)$ be such that the smallest singular
value of $zI - A$ is simple, and, denoting with $u$, $v$ a pair of left,
right singular vectors corresponding to $\sigma_{\min}(zI - A)$, such that 
$u^\ast v \neq 0$.  Moreover, let $\zeta_4$ be the fixed-point
map associated with Algorithm \ref{alg:fixed_point_matrix}.
\begin{enumerate}
	\item If $z$ is not an rbvt point in $\Lambda_{\epsilon}(A)$,
	then $z$ is not a fixed-point of $\zeta_4$.
	\item If $z$ is the unique globally rightmost point in $\Lambda_{\epsilon}(A)$,
	then $z$ is a fixed-point of $\zeta_4$.
\end{enumerate}
\end{theorem}

\noindent
Analogous to the earlier conclusions made regarding the convergence 
of Algorithm \ref{alg:fixed_point}, we infer from Theorem \ref{thm:convergence_matrix_case} 
that if the sequence $\{ z_k \}$ by Algorithm \ref{alg:fixed_point_matrix} 
converges to a point $z_\ast$ such that 
\begin{enumerate}
	\item $\sigma_{\min}(z_\ast I - A)$ is simple, and 
	\item $u^\ast v \neq 0$ for the left, right singular vectors $u$, $v$ 
	corresponding to $\sigma_{\min}(z_\ast I - A)$, 
\end{enumerate}
the converged point $z_\ast$ is an rbvt point in $\Lambda_{\epsilon}(A)$,
probably a locally rightmost point in $\Lambda_{\epsilon}(A)$.

\begin{example}\label{eq:psa_mat3}
Let us again consider the matrix $A$ of size $1006\times 1006$
from the NN18 example (see Example \ref{ex:psa_mat2}).
We compute the globally rightmost point in $\Lambda_{\epsilon}(A)$ 
for $\epsilon = 40$ and $\epsilon = 80$ using Algorithm \ref{alg:fixed_point_matrix}
with tolerance ${\sf tol} = 10^{-10}$. Initially, $z_0$ is set equal to the 
eigenvalue with the largest imaginary part. The computed rightmost points
for $\epsilon = 40$, $\epsilon = 80$ are $39 + 40{\rm i}$, $79 + 40{\rm i}$,
respectively, and match with the points returned by the globally convergent
criss-cross algorithm \cite[Algorithm 3.1]{BurLO03}. These converged points are displayed
in Figure \ref{fig:mat_psa} together with the plots of the boundary of $\Lambda_{\epsilon}(A)$
for $\epsilon = 40$, $\epsilon = 80$. For both values of $\epsilon$,
the algorithm terminates after three iterations. The iterates of 
Algorithm \ref{alg:fixed_point_matrix}  are listed in Table \ref{tab:mat_iterates}.
\end{example}

\begin{figure}
\begin{floatrow}
\ffigbox{%
	\includegraphics[width = .38\textwidth]{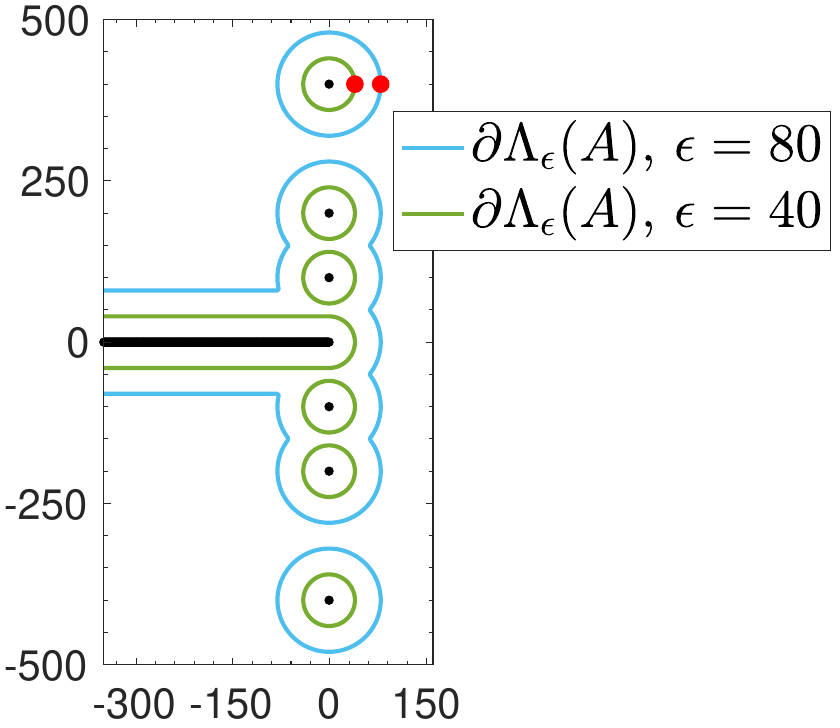} 
}{%
  \caption{  The boundary of $\Lambda_{\epsilon}(A)$ for the matrix $A$ from
  the NN18 example. Black and red dots correspond to eigenvalues
  and computed rightmost points in $\Lambda_{\epsilon}(A)$. }
  \label{fig:mat_psa}
}
\capbtabbox{%
\begin{tabular}{c}
    \begin{tabular}{|c||c|}
\hline
	$k$		&	Alg.~\ref{alg:fixed_point_matrix}			\\
\hline
\hline
	0		& \hskip -.5ex	$-1.00000 + \underline{400.00000}{\mathrm i}$		\\
	1		& \hskip -.5ex	$\: \underline{3}8.82335 + \underline{400.00000}{\mathrm i}$	\\
	2		& \hskip -.5ex	$\: \underline{39.00000} + \underline{400.00000}{\mathrm i}$	\\
\hline
    \end{tabular}		\\[3em]	
        \begin{tabular}{|c||c|}
	\hline
	$k$		&	Alg.~\ref{alg:fixed_point_matrix}		\\	
	\hline
	\hline
	0		& \hskip -.5ex	$-1.00000 + \underline{400.00000}{\mathrm i}$		\\
	1		& \hskip -.5ex	$\: \underline{7}8.64670 + \underline{400.00000}{\mathrm i}$	\\
	2		& \hskip -.5ex	$\: \underline{79.00000} + \underline{400.00000}{\mathrm i}$	\\
	\hline
   \end{tabular}		\\[3em]
\end{tabular} 
}{%
  \caption{ Concerns the application of Alg.~\ref{alg:fixed_point_matrix}
  		 to $A$ from the NN18 example.
  		(Top) Iterates $z_k$, $\epsilon = 40$. 
		 (Bottom) Iterates $z_k$, $\epsilon = 80$.}%
	\label{tab:mat_iterates}
}
\end{floatrow}
\end{figure}

\section{Initializations for global solutions of large-scale problems}\label{sec:glob_opts}
For the large-scale setting, the existing iterative
algorithms for the computation of $\alpha_{\epsilon}(T)$
(including the algorithms in \cite{MicG12, MeeMMV17}, and
Algorithms \ref{alg:fixed_point}--\ref{alg:fixed_point_matrix} below) 
suffer from local convergence. 
Initializing the locally convergent iterative algorithms with points close to the global
rightmost points is a remedy, and should usually result in convergence to a global rightmost point.


Some of the iterative algorithms remain throughout in the same connected component
of $\Lambda_{\epsilon}(T)$ from where it is initiated. Thus, to attain global convergence for an 
iterative algorithm that is guaranteed to 
converge only locally, it may be important 
to start with an initial point lying in the connected component that contains the rightmost 
point in $\Lambda_{\epsilon}(T)$ globally. To this end, suppose
$\mu_\ast \in \argmax_{\mu_0 \in \Lambda_{\epsilon}(T)} {\mathcal R}(\epsilon; \mu_0)$.
By Theorem \ref{thm:fo_estimate}, assuming its assumptions are satisfied,
the definition of ${\mathcal R}(\epsilon; \mu_\ast)$ and 
the compactness of ${\mathcal S}$, there is $\Delta_\ast \in {\mathcal S}$ such that
\begin{equation}\label{eq:est_rightmost1}
	\alpha_{\epsilon}(T) \; = \; {\mathcal R}(\epsilon; \mu_\ast)
					\; = \; \text{Re}\{ \mu(\epsilon; \Delta_\ast , \mu_\ast) \}	\:	.
\end{equation}
Moreover,
since $\epsilon \Delta_\ast \in {\mathcal S}_\epsilon$, we have 
$\mu(\epsilon; \Delta_\ast , \mu_\ast) \in \Lambda_{\epsilon}(T)$, so
$\mu(\epsilon; \Delta_\ast , \mu_\ast)$ is a globally rightmost point in $\Lambda_{\epsilon}(T)$.
There is a continuous map, namely ${\mathcal C}(\eta) : [0, \epsilon] \rightarrow \Lambda_{\epsilon}(T)$,
${\mathcal C}(\eta) := \mu(\eta; \Delta_\ast , \mu_\ast)$ such that ${\mathcal C}(0) = \mu_\ast$
and ${\mathcal C}(\epsilon) = \mu(\epsilon; \Delta_\ast , \mu_\ast)$. This shows that
$\mu_\ast$ and the rightmost point $\mu(\epsilon; \Delta_\ast , \mu_\ast)$ in $\Lambda_{\epsilon}(T)$
are in the same connected component of $\Lambda_{\epsilon}(T)$.
\begin{theorem}
Suppose that the assumptions of Theorem \ref{thm:fo_estimate} hold. Then
any eigenvalue $\mu_\ast \in \argmax_{\mu_0 \in \Lambda_{\epsilon}(T)} {\mathcal R}(\epsilon; \mu_0)$
is in a connected component of $\Lambda_{\epsilon}(T)$ that contains a rightmost point of 
$\Lambda_{\epsilon}(T)$ globally.
\end{theorem}

The arguments above suggest that the iterative algorithms can be initialized with
\begin{equation}\label{eq:cru_eig}
	\widetilde{\mu}_\ast \in
	\argmax_{\mu_0 \in \Lambda(P)}
						\left\{
							\mathrm{Re}(\mu_0)		+
		\epsilon	\left\{	 
		 \frac{ \sqrt{w_1^2 |t_1(\mu_0)|^2 + \dots + w_{\kappa}^2 | t_\kappa(\mu_0) |^2}}{| y^\ast_{\mu_0} T'(\mu_0) x_{\mu_0} |} 	
		 		\right\}		
						\right\}	\:	
\end{equation}
using the first-order approximation of ${\mathcal R}(\epsilon; \mu_0)$ in (\ref{eq:R_approx}). 
In the special setting of computing $\alpha_{\epsilon}(A)$ for a matrix $A \in {\mathbb C}^{n\times n}$, 
the condition above to choose the initial point simplifies as
\begin{equation}\label{eq:cru_eig3}
	\widetilde{\mu}_\ast \in
	\argmax_{\mu_0 \in \Lambda(A)}
						\left\{
							\mathrm{Re}(\mu_0)		+
		\epsilon	\left( 
		 \frac{ 1 }{| y^\ast_{\mu_0} x_{\mu_0} |} 	
		 		\right)		
						\right\}
\end{equation}
based on the first-order approximation of ${\mathcal R}(\epsilon; \mu_0)$ in (\ref{eq:RA_approx}).
\begin{remark}\label{rem:init_st1}
The fixed-point iterations, i.e., Algorithms \ref{alg:fixed_point} and \ref{alg:fixed_point2m} 
(Algorithm \ref{alg:fixed_point_matrix}), 
proposed above for the pseudospectral abscissa of a matrix-valued function 
(respectively, of a matrix)
can be initialized by setting $z_0$ equal to $\widetilde{\mu}_\ast$ as in (\ref{eq:cru_eig}) 
(respectively, as in (\ref{eq:cru_eig3})).
\end{remark}


Perhaps even a better option for initialization is using an accurate estimate
for the global maximizer sought, i.e., a point guaranteed to lie inside $\Lambda_{\epsilon}(T)$
and that is close to the globally rightmost point in $\Lambda_{\epsilon}(T)$.
Let $\mu_\ast \in \argmax_{\mu_0 \in \Lambda_{\epsilon}(T)} {\mathcal R}(\epsilon; \mu_0)$
and $\Delta_\ast \in {\mathcal S}$ be the optimal perturbation satisfying (\ref{eq:est_rightmost1}).
As argued above, $\mu(\epsilon; \Delta_\ast , \mu_\ast)$ is globally the rightmost
point in $\Lambda_{\epsilon}(T)$. Neither optimal $\mu_\ast$ nor the optimal perturbation 
$\Delta_\ast$ is known in practice. Yet, we can approximate $\mu_\ast$ with
$\widetilde{\mu}_\ast$ based on (\ref{eq:cru_eig}), and $\Delta_\ast$
with $\widetilde{\Delta}_\ast = (\underline{\Delta T}_1 , \dots , \underline{\Delta T}_\kappa)$ 
for $\underline{\Delta T}_1, \dots, \underline{\Delta T}_\kappa$
as in (\ref{eq:opt_perturbation}) but by replacing $\mu_0$ with $\widetilde{\mu}_\ast$.
Moreover,
\begin{equation*}
\begin{split}
	& \mu(\epsilon ; \Delta_\ast, \mu_\ast)	\;	 \approx	\;	
	\mu(\epsilon ; \widetilde{\Delta}_\ast, \widetilde{\mu}_\ast)	\;	\\ 	
	&	\hskip 4.5ex	\approx	\;
	\widetilde{\mu}_\ast	+	\epsilon \mu'(0; \widetilde{\Delta}_\ast, \widetilde{\mu}_\ast)	
					\;	=	\;	
			\widetilde{\mu}_\ast	+
							\epsilon
							\left\{
			\frac{\sqrt{w_1^2 |t_1(\widetilde{\mu}_\ast)|^2 + \dots + w_{\kappa}^2 |t_{\kappa}(\widetilde{\mu}_\ast)|^2}}{| y^\ast T'(\widetilde{\mu}_\ast) x |}
							\right\}	\;	=:	\;	\widetilde{\mu}_{\epsilon}	\:	.	
\end{split}											
\end{equation*}	
Thus, in practice, after computing $\widetilde{\mu}_\ast$, 
$\widetilde{\Delta}_\ast = (\underline{\Delta T}_1 , \dots , \underline{\Delta T}_\kappa) \in {\mathcal S}$ and 
$\widetilde{\mu}_{\epsilon}$, it seems plausible to approximate $\mu(\epsilon ; \Delta_\ast, \mu_\ast)$ 
with the eigenvalue of $(T + \epsilon \widetilde{\Delta T}_\ast)(\lambda)$ closest to 
$\widetilde{\mu}_{\epsilon}$, where $\widetilde{\Delta T}_\ast(\lambda) := t_1(\lambda) w_1 \underline{\Delta T}_1 
						+ \dots + t_{\kappa}(\lambda) \underline{\Delta T}_\kappa$,
which we refer as \emph{the first-order strategy}.


In the matrix setting, for approximating $\mu(\epsilon; \Delta_\ast , \mu_\ast)$
with higher accuracy, the second-order formulas in Section \ref{sec:psa_so_approx} are also available to our use:
\begin{itemize}
\item
\textbf{Second-order strategy.}
As argued in Remark \ref{rmk:so_approx_prac}, 
the rightmost eigenvalue of $A + \epsilon \widetilde{\Delta}^{\mu_0}_{\ast}$
for the maximizer $\mu_0 \in \Lambda(A)$ of the maximization problem on right-hand side 
of (\ref{eq:so_approx_formula}) and 
for the corresponding matrix $\widetilde{\Delta}^{\mu_0}_{\ast}$ as in (\ref{eq:so_perturbation})
is possibly a good estimate for $\mu(\epsilon; \Delta_\ast , \mu_\ast)$. In particular,
the real part of this rightmost eigenvalue and $\alpha_{\epsilon}(A)$ differ by ${\mathcal O}(\epsilon^3)$
under simplicity assumptions on the eigenvalues of $A$ and nearby matrices.
However, this requires forming $\widetilde{\Delta}^{\mu_0}_{\ast}$, as well as
computing the rightmost eigenvalue of $A + \epsilon \widetilde{\Delta}^{\mu_0}_{\ast}$
for every $\mu_0 \in \Lambda(A)$. 
\vskip 1ex
\item
\textbf{Hybrid first-order, second-order strategy.}
To reduce the computational cost, 
the rightmost eigenvalue $\mu(\epsilon; \Delta_\ast , \mu_\ast)$ can be estimated by
the rightmost eigenvalue of $A + \epsilon \widetilde{\Delta}^{\mu_0}_{\ast}$ 
for $\widetilde{\Delta}_{\mu_0, \ast}$ as in (\ref{eq:so_perturbation}) only for a particular 
$\mu_0 \in \Lambda(A)$ from which the rightmost point in $\Lambda_{\epsilon}(A)$ is likely to originate, e.g.,
$\mu_0 = \widetilde{\mu}_\ast$ with $\widetilde{\mu}_\ast$ as in (\ref{eq:cru_eig3}). 
\end{itemize}


\begin{remark}\label{rem:init_psa}
Algorithm \ref{alg:fixed_point_matrix} to compute $\alpha_{\epsilon}(A)$ for a matrix $A$
can also be initialized as follows by using the estimate of $\mu(\epsilon; \Delta_\ast, \mu_\ast)$ 
from the hybrid first-order, second-order strategy above:
\begin{enumerate}
	\item[\bf (1)] 
	$z_0$ can be set equal to the estimate
	from the hybrid first-order, second-order strategy.
	\item[\bf (2)]
	$\Delta^{(0)}$ in line \ref{line_alg4:D0} of Algorithm \ref{alg:fixed_point_matrix} 
	can be formed as $\Delta^{(0)} = u_0 v_0^\ast$ in terms of unit left, right singular vectors $u_0$, $v_0$ 
	corresponding to $\sigma_{\min}(z_0 I - A)$ normalized so that $u_0^\ast v_0$ is real and positive.
\end{enumerate}
Algorithms \ref{alg:fixed_point} and \ref{alg:fixed_point2m} to compute $\alpha_{\epsilon}(T)$
for a matrix-valued function $T$ can also be initialized similarly, for instance by using the estimate 
of $\mu(\epsilon; \Delta_\ast, \mu_\ast)$ from the first-order strategy.
\end{remark}

\section{Software}\label{sec:software}
The MATLAB implementations of (i) Algorithm \ref{alg:fixed_point} and Algorithm \ref{alg:fixed_point2m}
to compute the $\epsilon$-pseudospectral abscissa of a quadratic matrix polynomial, and (ii) Algorithm \ref{alg:fixed_point_matrix} to compute the $\epsilon$-pseudospectral abscissa of a matrix
are publicly available \cite{AhmM25}. 

\medskip

\noindent
\textbf{Initializations.} 
The implementations of Algorithms \ref{alg:fixed_point} and  
\ref{alg:fixed_point2m} for a quadratic matrix-polynomial are initialized by default
as in Remark \ref{rem:init_st1}, i.e., by setting $z_0$ equal to an eigenvalue $\mu_\ast$ 
satisfying (\ref{eq:cru_eig}). 
On the other hand, the implementation of Algorithm \ref{alg:fixed_point_matrix}
for a matrix $A$ is initialized by default according to Remark \ref{rem:init_psa},
i.e., by setting $z_0$ equal to the estimate of the globally rightmost point in $\Lambda_{\epsilon}(A)$ 
from the hybrid first-order, second-order strategy.

\medskip

\noindent
\textbf{Restarts.}
All implementations have an optional parameter ${\mathcal N}$, which controls the number
of times the fixed-point iteration (i.e., Algorithm \ref{alg:fixed_point}, Algorithm \ref{alg:fixed_point2m} or
Algorithm \ref{alg:fixed_point_matrix}) is executed, with each execution starting from a different $z_0$.
The largest value returned by these executions is taken as the computed value of the
$\epsilon$-pseudospectral abscissa. To give the specifics
when ${\mathcal N} > 1$, 
for Algorithm \ref{alg:fixed_point} or Algorithm \ref{alg:fixed_point2m},
the algorithm is executed ${\mathcal N}$ times, each time with $z_0$ equal to
one of the ${\mathcal N}$ eigenvalues yielding the largest value of the metric on the right-hand 
side of (\ref{eq:cru_eig}). \\[.4em]
\noindent
For Algorithm \ref{alg:fixed_point_matrix} with ${\mathcal N} > 1$ to compute
$\alpha_{\epsilon}(A)$ for a matrix $A$,  first ${\mathcal N}$ eigenvalues of
$A$ yielding the largest values of the metric on the right-hand side of (\ref{eq:cru_eig3}) are computed.
Then, for each such eigenvalue $\mu_0$, an initialization point $z_0$ is 
obtained by using the second-order strategy. \\[.4em]
\noindent
With ${\mathcal N} = 1$, these strategies
reduce to the default initialization strategies explained above.

\section{Numerical results}\label{sec:num_examp}
Here we report numerical results with publicly available MATLAB implementations
described in the previous section.

In subsection \ref{subsec:numexp_poly}, we first provide a numerical comparison
of Algorithms \ref{alg:fixed_point} and \ref{alg:fixed_point2m} on quadratic matrix 
polynomials arising from damping optimization. We also illustrate their accuracy
by comparing the results returned by them with those by the criss-cross 
algorithm \cite[Algorithm 2.1]{MehM24}.

In subsection \ref{subsec:numexp_matrix}, we focus on 
Algorithm \ref{alg:fixed_point_matrix} for the matrix setting,
in particular present numerical results with this approach on a parameter-dependent
matrix related to stabilization by static output feedback. We also 
provide comparisons with the criss-cross algorithm for matrices \cite[Algorithm 3.1]{BurLO03}, 
as well as \cite[Algorithm PSA1]{GugO11}.

All numerical experiments in this section are performed using 
MATLAB 2024b on a MacBook Air 
with Mac OS~15.4 operating system, Apple M3 CPU and 16GB RAM.

The termination tolerance used for every execution of 
Algorithms \ref{alg:fixed_point}, \ref{alg:fixed_point2m}, \ref{alg:fixed_point_matrix}
and \cite[Algorithm PSA1]{GugO11} is $\mathsf{tol} = 10^{-8}$.
The condition that we employ for termination 
inside Algorithm \ref{alg:fixed_point_matrix}
in the experiments in subsection \ref{subsec:numexp_matrix}
is $| \text{Re}(z_{k+1}) - \text{Re}(z_{k}) | < 
				\mathsf{tol} \cdot \max\{1 \, , \,  | \text{Re}(z_{k}) | \}$,
which is the termination condition employed by the implementation of \cite[Algorithm PSA1]{GugO11}.
We rely on this condition rather than 
 $| \text{Re}(z_{k+1}) - \text{Re}(z_{k}) | < \mathsf{tol}$ 
 for the sake of a fair comparison with \cite[Algorithm PSA1]{GugO11}.

\subsection{Results on quadratic matrix polynomials}\label{subsec:numexp_poly}

\subsubsection{One-parameter damping example.}\label{sec:numexp_cont_param}
Consider the quadratic matrix polynomial 
$P(\lambda; \nu) = \lambda^2 M + \lambda C(\nu) + K$ with
$M, C(\nu), K \in {\mathbb R}^{20\times 20}$ from Example \ref{ex:damp1}
as in (\ref{eq:damping_ex_20by20}), where
$C_{\mathrm{int}} \, = \, 2\xi  M^{1/2} \sqrt{M^{-1/2} K M^{-1/2}} M^{1/2}$
and $\xi = 0.005$. Recall that this problem corresponds to a mass-spring-damper system
with an external damper with viscosity $\nu$ positioned on the second mass
\cite[Example 5.2]{MehM24}. We compute $\alpha_{\epsilon}(P(\cdot;\nu))$
for $\epsilon = 0.4$, weights $(w_1, w_2, w_3) = (1,1,1)$ and $\nu \in [0,100]$ 
(i.e., for every $\nu$ on a uniform grid for $[0,100]$)
using Algorithms \ref{alg:fixed_point} 
and \ref{alg:fixed_point2m} initialized according to (\ref{eq:cru_eig}), as well as 
the globally convergent criss-cross algorithm. They all return the same results 
up to prescribed tolerances throughout $\nu \in [0,100]$. In particular, a plot
of $\alpha_{\epsilon}(P(\cdot;\nu))$ computed by Algorithm \ref{alg:fixed_point}
over $\nu \in [0,100]$ is given on the left in Figure \ref{fig:error_poly_fp_NN18}
with the solid black curve, and
the plots of the values computed by Algorithm \ref{alg:fixed_point2m} and the 
criss-cross algorithm are indistinguishable. The absolute value of the difference
of the computed values by Algorithm \ref{alg:fixed_point} and the criss-cross 
algorithm, as well as the difference of the values by Algorithm \ref{alg:fixed_point2m} 
and the criss-cross algorithm are depicted on the right in Figure \ref{fig:error_poly_fp_NN18}
as a function of $\nu \in [0,100]$.
Both of the the differences are less than the prescribed error tolerance $10^{-8}$
throughout $[0,100]$, yet Algorithm \ref{alg:fixed_point} appears to be more accurate.

Additionally, we report the total times in seconds taken by the three algorithms
in Table \ref{tab:poly_fp_NN18_times}, as well as the average number of iterations
required by Algorithms \ref{alg:fixed_point}, \ref{alg:fixed_point2m} 
until termination in Table \ref{tab:poly_fp_NN18_times}. 
Algorithms \ref{alg:fixed_point}, \ref{alg:fixed_point2m} are significantly faster
than the criss-cross algorithm, yet return the same values as the criss-cross algorithm
up to prescribed error tolerance.
Among Algorithms \ref{alg:fixed_point}, \ref{alg:fixed_point2m}, the former requires
fewer iterations to satisfy the prescribed tolerance, leading also to slightly 
faster runtime.

We have alternatively initiated Algorithms \ref{alg:fixed_point} and \ref{alg:fixed_point2m}
with the rightmost eigenvalues. The computed values by both algorithms are the same
up to tolerance, and plotted on the left in Figure \ref{fig:error_poly_fp_NN18} with the dashed
orange curve. The plot shows that the computed values of $\alpha_{\epsilon}(P(\cdot;\nu))$
are now much smaller than the actual values, and this occurs due to convergence
to wrong locally rightmost points that are not rightmost globally. Especially, for
general matrix-valued functions including matrix polynomials, it seems essential
to initialize the iterative algorithms properly, possibly based on (\ref{eq:cru_eig}).
Perturbations of rightmost eigenvalues especially in the nonlinear eigenvalue setting
may not to lead to the globally rightmost 
points in the $\epsilon$-pseudospecrtum; see, e.g.,  Figure \ref{fig:poly_psa}, where
the rightmost eigenvalue is the eigenvalue closest to the origin, yet the least
sensitive eigenvalue.


\begin{figure}
	\begin{tabular}{cc}
		\hskip -1.5ex
		\includegraphics[width = .5\textwidth]{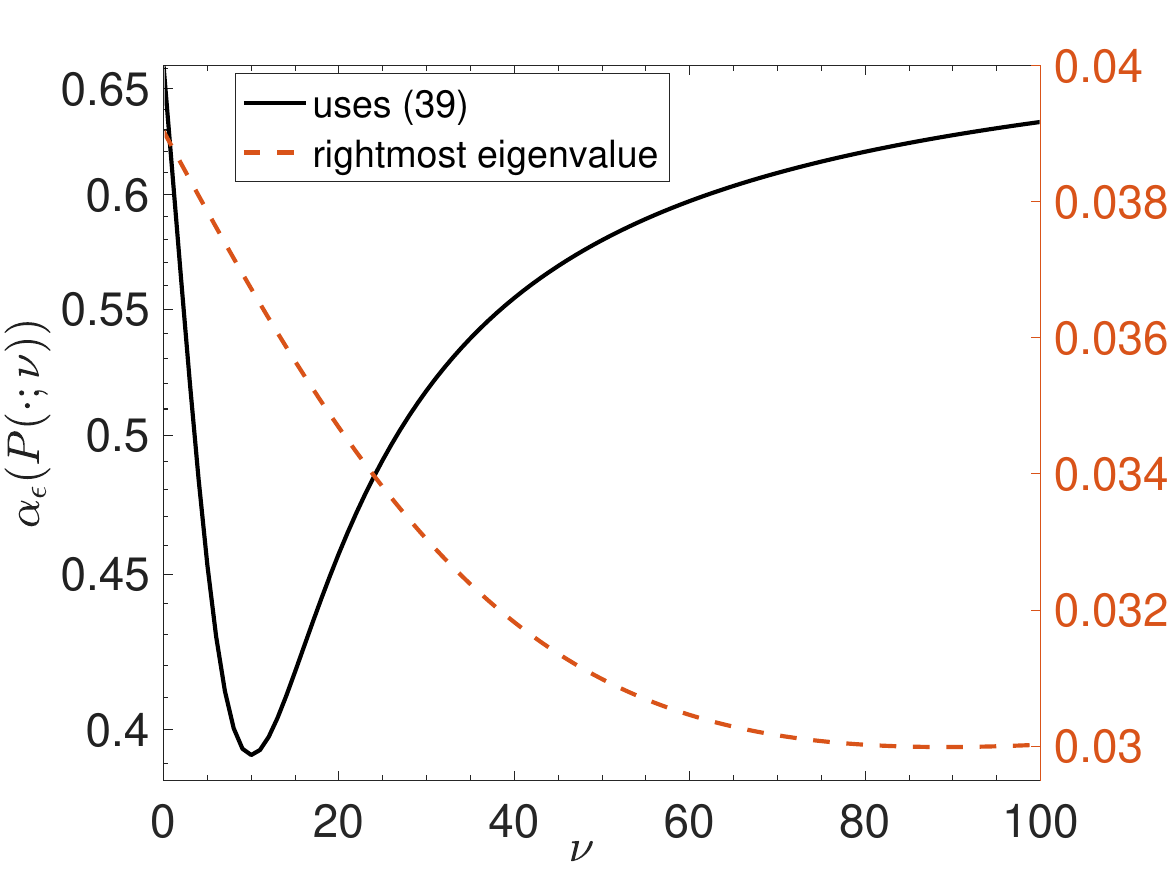} 		 & 
			\hskip .5ex
			\includegraphics[width = .5\textwidth]{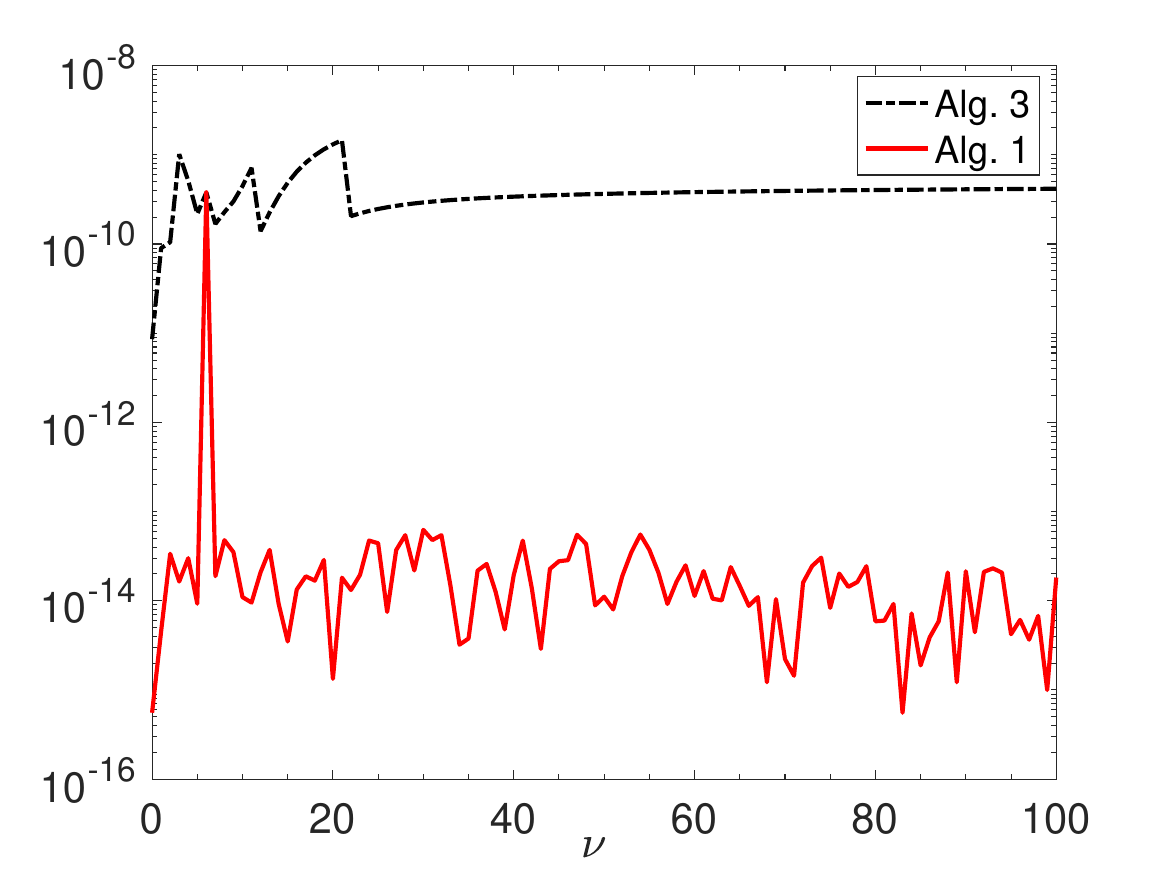}
	\end{tabular}
		\caption{
		This figure concerns the one-parameter damping problem
		in \S\ref{sec:numexp_cont_param}.  
		(Left) The computed value of $\alpha_{\epsilon}(P(\cdot ; \nu))$ 
		by Algorithm \ref{alg:fixed_point} as a function of $\nu \in [0,100]$.
		The plots of $\alpha_{\epsilon}(P(\cdot ; \nu))$ 
		by Algorithm \ref{alg:fixed_point2m} and the criss-cross algorithm
		are indistinguishable.
		(Right) The errors of Algorithms \ref{alg:fixed_point} and \ref{alg:fixed_point2m}, 
		that is the 
		absolute values of the differences between the values returned by these algorithms
		and the criss-cross algorithm.}
		\label{fig:error_poly_fp_NN18}
\end{figure}

\begin{table}
	\begin{tabular}{|c||cc|}
		\hline
								&	time		&	iterations		\\
		\hline
		\hline
		Alg. \ref{alg:fixed_point}		&	0.9		&	7.07			\\
		Alg. \ref{alg:fixed_point2m}	&	1.2 		&	10.75		\\
		criss-cross				&	38.1		&	---			\\
		\hline
	\end{tabular}
		\caption{  The total runtime
				in seconds of Algorithms \ref{alg:fixed_point}, 
				\ref{alg:fixed_point2m} and the criss-cross algorithm, and
				the average number of iterations until termination
				by Algorithms \ref{alg:fixed_point}, \ref{alg:fixed_point2m} for the
				example in \S\ref{sec:numexp_cont_param}.}
		\label{tab:poly_fp_NN18_times}
\end{table}

\subsubsection{Two-parameter damping example.}\label{sec:numexp_cont_param2}
Let us now consider the damping problem above with the addition of
a second damper on the nineteenth mass. The mass and stiffness
matrices $M$ and $K$ are as before, but the damping matrix becomes
\[
	C(\nu_1, \nu_2)
		=
	C_{\mathrm{int}}	\,	+	\,	\nu_1 e_2 e_2^T	
						\,	+	\,	\nu_2 e_{19} e_{19}^T  \:
\]
where we constrain the damping parameters $\nu_1$ and $\nu_2$
to lie in the intervals $[0,50]$ and $[0, 100]$, respectively. We compute
$\alpha_{\epsilon}(\cdot ; \nu)$, $\nu := (\nu_1, \nu_2)$ for $\epsilon = 0.2$
and the weights $(w_1,w_2,w_3) = (1,1,1)$
on the box $[0,50]\times [0,100]$ (i.e., on a uniform grid for this box)
using Algorithms \ref{alg:fixed_point} and \ref{alg:fixed_point2m} initialized according to (\ref{eq:cru_eig}),
and the criss-cross algorithm. All three return the same values of 
$\alpha_{\epsilon}(\cdot ; \nu)$ on the whole box up to the prescribed error 
tolerance. The plot of the computed $\alpha_{\epsilon}(\cdot ; \nu)$ by
 Algorithm \ref{alg:fixed_point} as a function of $\nu$ in the box 
 is given on the left in Figure \ref{fig:error_poly_fp_multi1}. In the same
 figure on the right, the error of  Algorithm \ref{alg:fixed_point}, that is the absolute
 value of the difference between the computed values by Algorithm \ref{alg:fixed_point}
 and the criss-cross algorithm, is shown on the box. 


\begin{figure}
	\begin{tabular}{cc}
		\hskip -3.5ex
			\includegraphics[width = .5\textwidth]{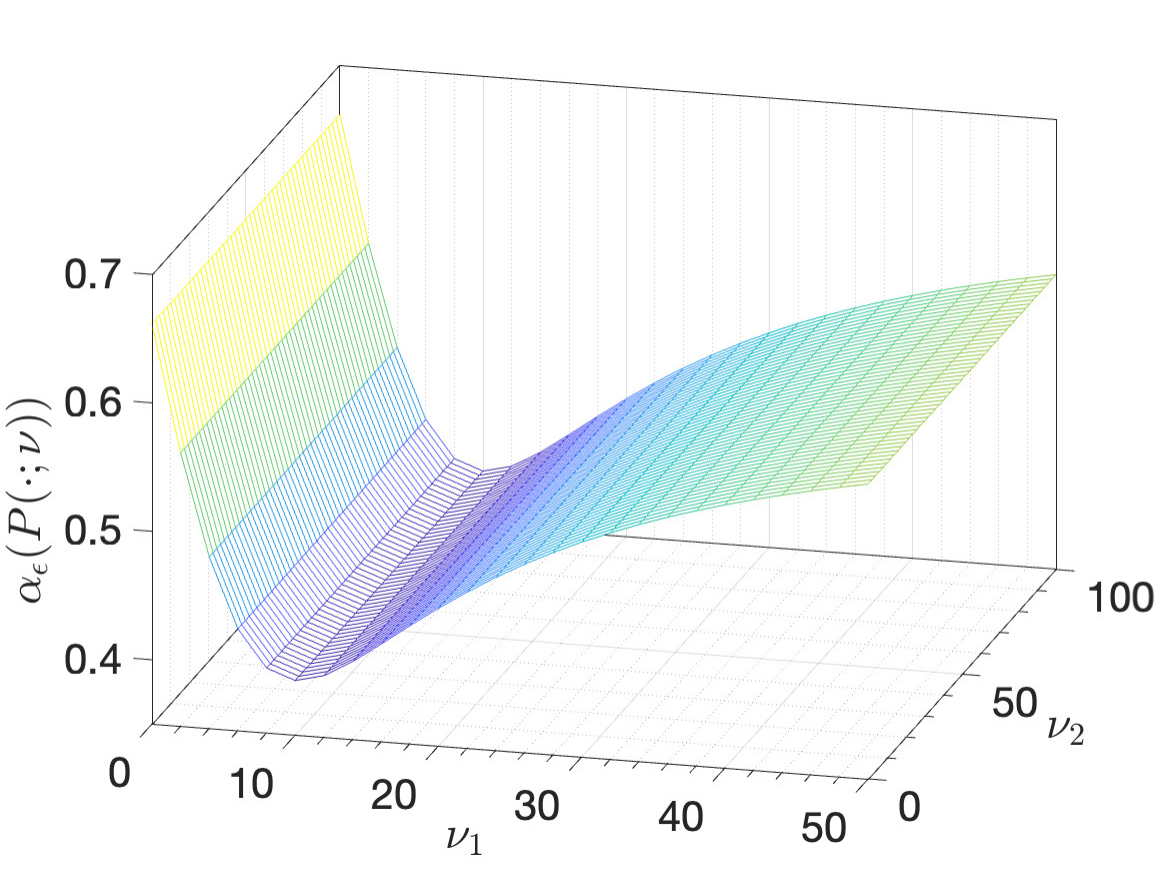} & 
			\hskip -1.5ex
			\includegraphics[width = .5\textwidth]{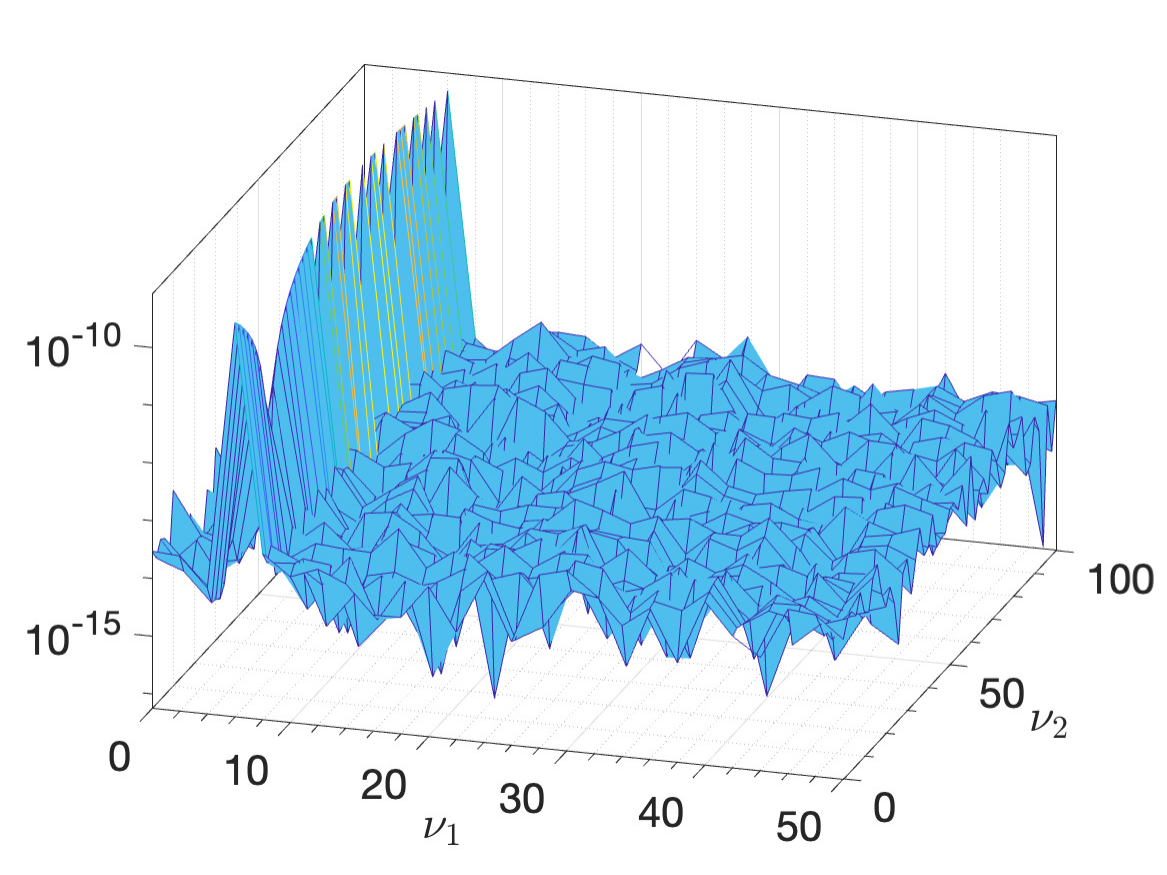} 
	\end{tabular}
		\caption{ The figure concerns the two-parameter damping problem 
		in \S\ref{sec:numexp_cont_param2}.  (Left) Plot of
		$\alpha_{\epsilon}(P(\cdot ; \nu))$ computed by Algorithm \ref{alg:fixed_point}
		on the box $[0,50] \times [0,100]$. (Right) Error of Algorithm \ref{alg:fixed_point}, 
		that is the absolute value of the gap between
		the computed values of $\alpha_{\epsilon}(P(\cdot ; \nu))$ by Algorithm \ref{alg:fixed_point} 
		and the criss-cross algorithm.}
		\label{fig:error_poly_fp_multi1}
\end{figure}

We additionally list the time taken by the algorithms, and the average number of
iterations by \ref{alg:fixed_point}, \ref{alg:fixed_point2m} 
until termination in Table \ref{tab:poly_fp_multi1_times} for this two-parameter
problem on the box. Once again, Algorithms \ref{alg:fixed_point}, \ref{alg:fixed_point2m} 
are significantly faster than the criss-cross algorithm, even though all three algorithms return 
the same values throughout the box up to the prescribed error tolerances. Similar to the
case for the one-parameter example above, Algorithm \ref{alg:fixed_point}
performs fewer iterations compared to Algorithm \ref{alg:fixed_point2m}, and,
as a result, Algorithm \ref{alg:fixed_point} has a lower runtime.

\begin{table}
	\begin{tabular}{|c||cc|}
		\hline
								&	time		&	iterations		\\
		\hline
		\hline
		Alg. \ref{alg:fixed_point}		&	13.3		&	8.31			\\
		Alg. \ref{alg:fixed_point2m}	&	18.8 		&	12.62		\\
		criss-cross				&	398.5		&	---			\\
		\hline
	\end{tabular}
		\caption{ The runtimes in seconds of the three algorithms and the number of iterations of 
				Algorithms \ref{alg:fixed_point}, \ref{alg:fixed_point2m} 
				for the two-parameter example in \S\ref{sec:numexp_cont_param2}. }
		\label{tab:poly_fp_multi1_times}
\end{table}

\subsubsection{Comparison of Algorithms \ref{alg:fixed_point} and \ref{alg:fixed_point2m} 
for varying values of $\epsilon$}
In this section, we take a closer look at the convergence of 
Algorithms \ref{alg:fixed_point} and \ref{alg:fixed_point2m} 
initialized according to (\ref{eq:cru_eig}) for varying values of $\epsilon$. In particular,
let us consider the computation of $\alpha_{\epsilon}(P(\cdot ; \nu))$ using these algorithms
for the one-parameter damping example in \S\ref{sec:numexp_cont_param}
for various values of $\nu \in [0,100]$, $(w_1, w_2, w_3) = (1,1,1)$ and $\epsilon = 0.2, 0.4, 0.8$.
The $\epsilon$-pseudospectrum is unbounded, and
$\alpha_{\epsilon}(P(\cdot ; \nu))$ is not bounded above regardless 
of $\nu$ for every $\epsilon > \sigma_{\min}(M) = 1$.  As $\epsilon < 1$
gets closer to one, it seems reasonable to expect that $\alpha_{\epsilon}(P(\cdot ; \nu))$
becomes more ill-conditioned and harder to compute.

\begin{table}
\begin{tabular}{c}
\begin{tabular}{|c||c|cc|ccc|cc|}
\hline
		&					& 	 \multicolumn{2}{c}{error}		&	\multicolumn{3}{|c}{time}	&	\multicolumn{2}{|c|}{iterations}			\\	
$\nu$	& $\alpha_{\epsilon} (P(\cdot ; \nu))$ &	Alg. \ref{alg:fixed_point}	 &	Alg. \ref{alg:fixed_point2m}		& Alg. \ref{alg:fixed_point}	& 	Alg. \ref{alg:fixed_point2m}		&	CCQ
							&  Alg. \ref{alg:fixed_point}		&  Alg. \ref{alg:fixed_point2m}	 \\[.2em]

\hline
\hline
0	&6.6147e-01	&	\phantom{-}5.6e-16	&	\phantom{-}8.6e-12	&	0.03	&0.04	&0.43	&7	&11	\\
10	&3.9242e-01	&	\phantom{-}1.1e-14	&	-4.5e-10	&0.03	&0.03		&0.14	&7	&9	\\
40	&5.5478e-01	&	-1.9e-14	&	-3.4e-10	&0.02	&0.04	&0.44	&7	&11	\\
100	&6.3385e-01	&	\phantom{-}1.8e-14	&	-4.2e-10	&0.02	&0.03	&0.43	&7	&11	\\
\hline 
\end{tabular}		\\
\phantom{aa} \\[-.1em]
\begin{tabular}{|c||c|cc|ccc|cc|}
\hline
		&					& 	 \multicolumn{2}{c}{error}		&	\multicolumn{3}{|c}{time}	&	\multicolumn{2}{|c|}{iterations}			\\	
$\nu$	& $\alpha_{\epsilon} (P(\cdot ; \nu))$ &	Alg. \ref{alg:fixed_point}	 &	Alg. \ref{alg:fixed_point2m}		& Alg. \ref{alg:fixed_point}	& 	Alg. \ref{alg:fixed_point2m}		&	CCQ
							&  Alg. \ref{alg:fixed_point}		&  Alg. \ref{alg:fixed_point2m}	 \\[.2em]

\hline
\hline
0	&1.4750e00	&	-2.2e-16	&	-1.5e-09	&	0.03	&0.04	&0.47	&12	&17	\\
10	&1.2856e00	&	\phantom{-}2.4e-10	&	-1.1e-09	&0.03	&0.04		&0.45	&11	&17	\\
40	&1.3947e00	&	\phantom{-}2.6e-14	&	-3.1e-09	&0.03	&0.04	&0.44	&11	&18	\\
100	&1.4632e09	&	-3.1e-15	&	-3.5e-09	&0.03	&0.04	&0.44	&11	&18	\\
\hline 
\end{tabular}	\\
\phantom{aa}	\\[-.1em]
\begin{tabular}{|c||c|c|cc|c|}
\hline
		&					& 	 \multicolumn{1}{c}{error}		&	\multicolumn{2}{|c}{time}	&	\multicolumn{1}{|c|}{iterations}			\\	
$\nu$	& $\alpha_{\epsilon} (P(\cdot ; \nu))$ &		 Alg. \ref{alg:fixed_point2m}			& 	Alg. \ref{alg:fixed_point2m}		&	CCQ
									&  Alg. \ref{alg:fixed_point2m}	 \\[.2em]

\hline
\hline
0	&4.6728e00	&	\phantom{-}9.7e-05	&	0.10	&0.21	&72	\\
10	&4.5928e00	&	-2.1e-08	&	0.10		&0.22		&74	\\
40	&4.6042e00	&	-1.7e-08	&	0.10		&0.24		&77	\\
100	&4.6455e00	&	\phantom{-}1.1e-06	&	0.10		&0.25	&77	\\
\hline 
\end{tabular}	\\
\end{tabular}
\caption{Comparison of the algorithms in terms of errors, runtimes in seconds and number of iterations
on the damping example in \S\ref{sec:numexp_cont_param}. 
CCQ stands for the criss-cross algorithm for 
a quadratic matrix polynomial \cite[Algorithm 2.1]{MehM24}.
(Top) $\epsilon =0.2$ (Middle) $\epsilon =0.4$ (Bottom) $\epsilon = 0.8$.}
\label{tab:epsvarious_alg1_vs_alg3}
\end{table}

In Table \ref{tab:epsvarious_alg1_vs_alg3}, the errors and number of iterations
of Algorithms \ref{alg:fixed_point} and \ref{alg:fixed_point2m}, as well as the
runtimes in seconds for these two algorithms and the criss-cross algorithm for 
a quadratic matrix polynomial (CCQ) \cite[Algorithm 2.1]{MehM24} are listed for
several values of the damping parameter $\nu$. Here, by error for one of these
two algorithms, we mean the difference between the computed values by the
algorithm and the criss-cross algorithm. A positive error means that 
Algorithm \ref{alg:fixed_point} or \ref{alg:fixed_point2m} returns a larger estimate
for $\alpha_{\epsilon}(P(\cdot ; \nu))$, while a negative error means the criss-cross
algorithm returns a larger value. For $\epsilon = 0.8$, Algorithm \ref{alg:fixed_point} does 
not converge, so the error, runtime, number of iterations for Algorithm \ref{alg:fixed_point}
are not listed at the bottom in Table \ref{tab:epsvarious_alg1_vs_alg3}. For $\epsilon = 0.2, 0.4$,
both Algorithm \ref{alg:fixed_point} and Algorithm \ref{alg:fixed_point2m} return
the correct values of $\alpha_{\epsilon}(P(\cdot;\nu))$ up to the prescribed error
tolerance. For these two values of $\epsilon$, they both have significantly lower runtime 
compared to the criss-cross algorithm, and the number of iterations required by both 
algorithms until termination is fewer than 20 in all runs. As also
observed in the previous subsection, Algorithm \ref{alg:fixed_point} requires fewer
number of iterations, and have slightly smaller runtimes compared to Algorithm \ref{alg:fixed_point2m} 
for these choices of $\epsilon$. This is the typical behavior we observe by the two
algorithms as long as $\epsilon$ is not close $\sigma_{\min}(M)$. 

As $\epsilon$ increases the number of iterations until termination by both algorithms 
also increase. For $\epsilon = 0.8$, Algorithm \ref{alg:fixed_point} does not converge 
anymore, while Algorithm \ref{alg:fixed_point2m} still converges even if it requires 
more than 70 iterations. For $\epsilon = 0.8$
and $\nu = 0, 100$, Algorithm \ref{alg:fixed_point2m} returns values even more accurate 
than the criss-cross algorithm as apparent from the results at the bottom of Table 
\ref{tab:epsvarious_alg1_vs_alg3}. In these examples, the criss-cross algorithm
is not capable of returning highly accurate results probably because the eigenvalue
problems it solves are ill-conditioned.
We observe in general that for larger values of $\epsilon$ close to $\sigma_{\min}(M)$,
Algorithm \ref{alg:fixed_point2m} still converges to $\alpha_{\epsilon}(P(\cdot;\nu))$
robustly even if it is at the expense of additional iterations, while Algorithm \ref{alg:fixed_point}
may possibly not converge.

\subsubsection{Comparison of the algorithms on larger matrix polynomials}\label{sec:numexp_wrt_size}
We next compare Algorithms \ref{alg:fixed_point}, \ref{alg:fixed_point2m},
and the criss-cross algorithm on larger quadratic matrix polynomials taken from
\cite[Example 5.3]{MehM24}. These quadratic matrix polynomials still originate 
from mass-damping-spring systems, and are of the form 
$P(\lambda) = \lambda^2 M + \lambda C_{\mathrm{int}} + K$ with $n\times n$ mass, 
stiffness, internal damping matrices of the form
\begin{equation*}
\begin{split}
	&	
	M
		=
	\mathrm{diag}(     
	1,2, \dots, n)	\:	,		\;\;
	K
		=
	\mathrm{tridiag}(-400,800,-400) \, ,		\\
	&
	C_{\mathrm{int}}
	\, = \, 2 \xi M^{1/2} \sqrt{M^{-1/2} K M^{-1/2}} M^{1/2}	,
\end{split}
\end{equation*}
and $\: \xi = 0.005 \:$ for $\: n = 80, 200, 400$.	We compute $\alpha_{\epsilon}(P)$
for $\epsilon = 0.5$ and two different choices of weights,
specifically for $(w_1, w_2, w_3)$ equal to $(1, 1, 1)$ and $(0.7, 1, 0)$, 
by using the three algorithms.

Table \ref{tab:nvarious_alg1_vs_alg3} reports the accuracy of the computed results, 
runtimes of the algorithms in seconds, and the number of iterations
until termination. For the $400\times 400$ example, we have not
provided the results of the criss-cross algorithm, because the computations take
excessive amount of time. It is clear from the table that 
Algorithms \ref{alg:fixed_point}, \ref{alg:fixed_point2m} take considerably
less time compared to the criss-cross algorithm on all examples, yet the
results returned by all three algorithms are the same up to the prescribed
tolerance for problems with size $n = 80, 200$. For $n = 400$, only the
results from Algorithms \ref{alg:fixed_point} and \ref{alg:fixed_point2m}
are available, and they seem to be the same up to the prescribed tolerance.
An observation aligned with those in the previous two subsections is that
Algorithm \ref{alg:fixed_point} seemingly requires fewer iterations that
result in better runtimes as compared to Algorithm \ref{alg:fixed_point2m}.

\begin{table}
\begin{tabular}{|r||c|cc|ccc|cc|}
\hline
		&					& 	 \multicolumn{2}{c}{error}		&	\multicolumn{3}{|c}{time}	&	\multicolumn{2}{|c|}{iterations}			\\	
$n$, $\: w$	& $\alpha_{\epsilon} (P)$ &	Alg. \ref{alg:fixed_point}	 &	Alg. \ref{alg:fixed_point2m}		& Alg. \ref{alg:fixed_point}	& 	Alg. \ref{alg:fixed_point2m}		&	CCQ
							&  Alg. \ref{alg:fixed_point}		&  Alg. \ref{alg:fixed_point2m}	 \\[.2em]

\hline
\hline
$80$, $\mathbf{1}$	&7.8362	&	\phantom{-}2.2e-10	&	-4.7e-09	&	0.6	&0.7	&64.7	&17	&23	\\
$80$, $\mathbf{u}$	&4.9734	&	\phantom{-}4.9e-11	&	-1.6e-09	&0.4		&0.5		&85.6	&11	& 16		\\
$200$, $\mathbf{1}$	&7.8362	&	-6.9e-11	&	-5.0e-09	&6.4		&8.9		&1028	&17	&23	\\
$200$, $\mathbf{u}$	&4.9734	&	\phantom{-}7.1e-11	&	-1.6e-09	&4.3		&6.2		&874		&11	&16	\\
$400$, $\mathbf{1}$	&7.8362	&	---	&	---	& 46.6	& 63.8	& ---	&17	&23	\\
$400$, $\mathbf{u}$	&4.9734	&	---	&	---	&44	&57.9	& ---	&11	&16	\\ 
\hline 
\end{tabular}	
\caption{Comparison of the algorithms for the quadratic matrix polynomials in
\S\ref{sec:numexp_wrt_size} of size $n = 80, 200, 400$, and with the
weights $w = (w_m, w_c, w_k)$ equal to ${\mathbf 1} := (1, 1, 1)$ and
${\mathbf u} := (0.7,1,0)$.}
\label{tab:nvarious_alg1_vs_alg3}
\end{table}

\subsection{The pseudospectral abscissa of a matrix}\label{subsec:numexp_matrix}
\subsubsection{Stabilization by output feedback example}\label{subsec:numexp_cont_param_M}
We experiment here with a stabilization by output feedback problem 
that involves a  matrix $A(\nu) = A + \nu_1 b c_1^T + \nu_2 b c_2^T$ with $\nu := (\nu_1, \nu_2)$
for prescribed $A \in {\mathbb R}^{130\times 130}$, $b , c_1, c_2 \in {\mathbb R}^{130}$ 
dependent on two real parameters $\nu_1, \nu_2$. This problem corresponds to the HF1
example in the $COMPl_{e}ib$ collection. We compute $\alpha_{\epsilon}(A(\nu))$
for $\epsilon = 0.8$ and $\nu$ on the box $[-1,1]\times[-1,1]$ 
(i.e., for every $\nu$ on a uniform grid for this box) using Algorithm \ref{alg:fixed_point_matrix}
and using the criss-cross algorithm \cite[Algorithm 3.1]{BurLO03}. The plot
of the computed $\alpha_{\epsilon}(A(\nu)$ by Algorithm \ref{alg:fixed_point_matrix}
is depicted on the left in Figure \ref{fig:error_mat_fp}, whereas the error of 
Algorithm \ref{alg:fixed_point_matrix} (i.e., the absolute value of the difference
between the computed values by Algorithm \ref{alg:fixed_point_matrix} and the
criss-cross algorithm) is shown on the right in the same figure. The error 
of  Algorithm \ref{alg:fixed_point_matrix} appears
to be less than the prescribed tolerance throughout the box.

The total runtimes in seconds by both algorithms, and the average 
number of iterations until termination by Algorithm \ref{alg:fixed_point_matrix}
are given in Table \ref{tab:matrix_fp_multi_times}. Algorithm \ref{alg:fixed_point_matrix} 
requires significantly smaller computation time compared to the criss-cross algorithm
to compute $\alpha_{\epsilon}(A(\nu))$ with the same accuracy up to the prescribed 
tolerance. 

\begin{table}
	\begin{tabular}{|c||cc|}
		\hline
								&	time		&	iterations		\\
		\hline
		\hline
		Alg. \ref{alg:fixed_point_matrix}		&	48.1		&	2.15			\\
		criss-cross				&	319.6		&	---			\\
		\hline
	\end{tabular}
		\caption{ The runtimes in seconds of Algorithm \ref{alg:fixed_point_matrix},
		the criss-cross algorithm, and the number of iterations of Algorithm \ref{alg:fixed_point_matrix} 
				for the HF1 example considered in \S\ref{subsec:numexp_cont_param_M}. }
		\label{tab:matrix_fp_multi_times}
\end{table}




\begin{figure}
	\begin{tabular}{cc}
		\hskip -3.5ex
			\includegraphics[width = .5\textwidth]{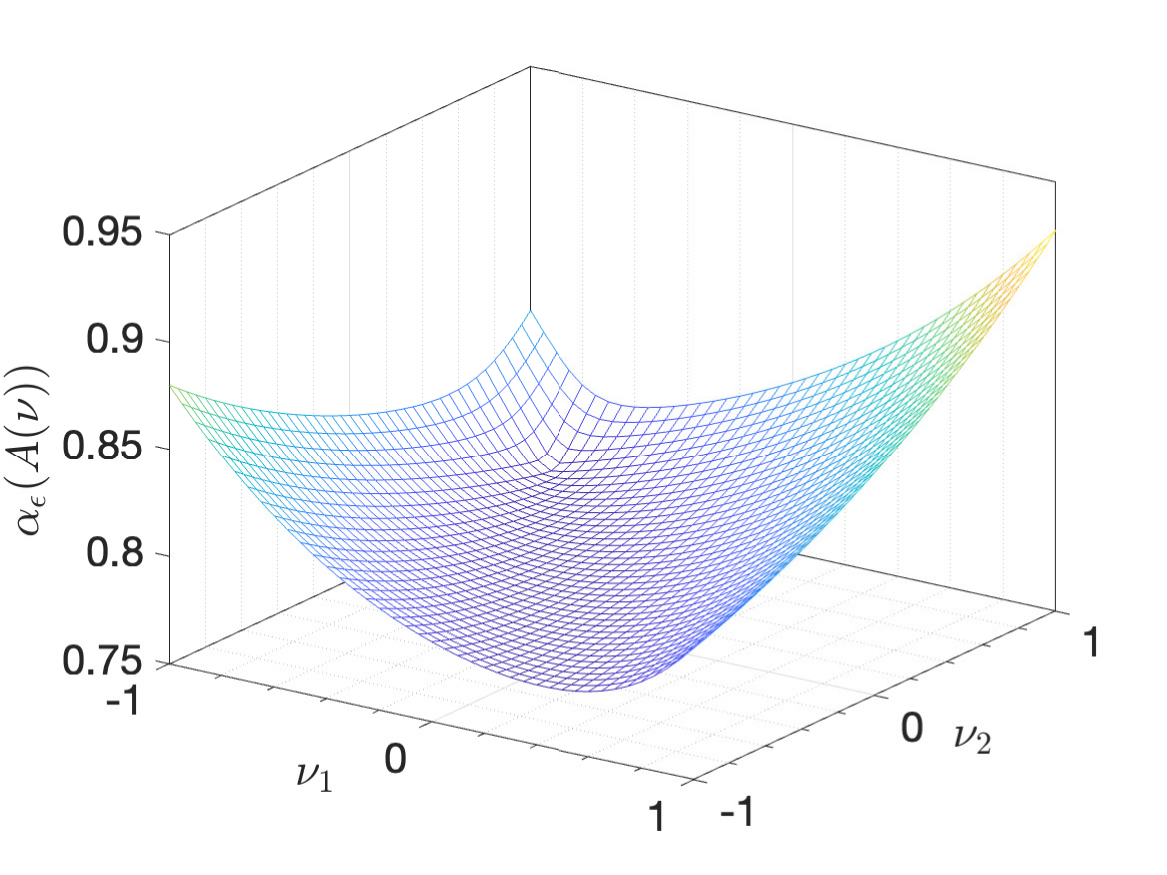} & 
			\hskip -1.5ex
			\includegraphics[width = .5\textwidth]{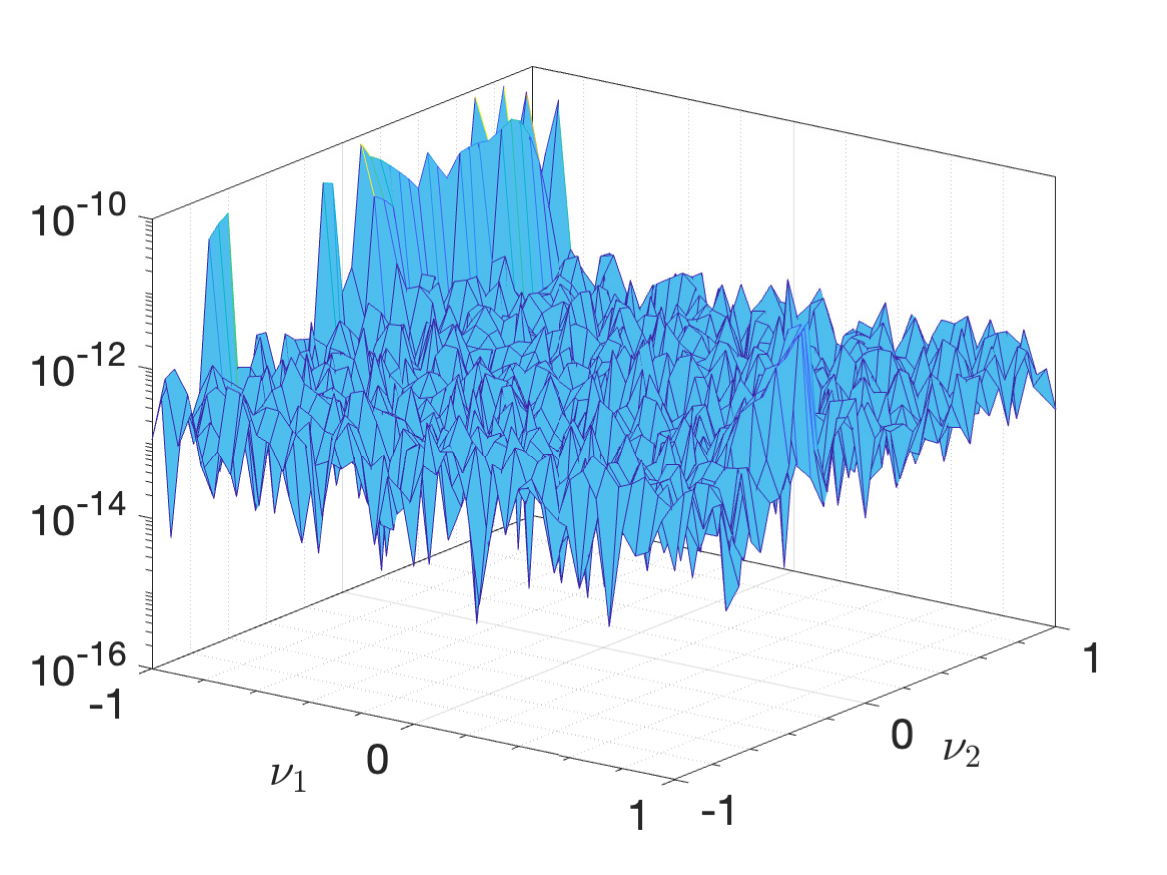} 
	\end{tabular}
		\caption{  The plot of computed $\alpha_{\epsilon}(\nu)$ by
		Algorithm \ref{alg:fixed_point_matrix} (left plot) and its error (right plot)
		as a function of $\nu = (\nu_1, \nu_2) \in [-1,1]\times[-1,1]$ 
		for the HF1 example in \S\ref{subsec:numexp_cont_param_M}.}
		\label{fig:error_mat_fp}
\end{figure}

\subsubsection{Comparing the performances of the algorithms}\label{sec:bench}
The experiments here are categorized into two: 
\begin{enumerate}
\item[(1)] small dense examples, 
where experiments are performed on various matrices of size $100 \times 100$ from 
the MATLAB package {\sf EigTool} \cite{eigtool}; 
\item[(2)] larger sparse examples, which are also obtained from {\sf EigTool}. 
\end{enumerate}
We compare the accuracy and computation time of the algorithms on these two classes. 
For small dense matrices, the accuracy of Algorithm \ref{alg:fixed_point_matrix} and 
the Guglielmi and Overton algorithm (GO) \cite[Algorithm PSA1]{GugO11}
are determined by comparing their estimates for the $\epsilon$-pseudospectral abscissa 
with those returned by the the criss-cross algorithm (CC) \cite[Algorithm 3.1]{BurLO03}.
For larger sparse examples, 
CC is no more computationally feasible, and we only provide 
the difference between the estimates for the $\epsilon$-pseudospectral 
abscissa returned by Algorithm \ref{alg:fixed_point_matrix} and by GO.

\medskip

\noindent
\textbf{Small dense examples.} \\
The experiments here are conducted for $\epsilon=0.2$, and the results are
reported in Table \ref{tab:table1}. 
Recall that error in this table for either one of Algorithm \ref{alg:fixed_point_matrix} or GO,
refers to the difference between the computed value of $\alpha_{\epsilon}(A)$
by the algorithm and the value by CC. For both values of $\epsilon$, 
the computed values for $\alpha_{\epsilon}(A)$ by
Algorithm \ref{alg:fixed_point_matrix} and GO are 
nearly the same as those returned by the criss-cross algorithm. Morover,
Algorithm \ref{alg:fixed_point_matrix} has smaller runtime compared to GO in all cases.
Even on these smaller examples, excluding the Grcar matrix, 
Algorithm \ref{alg:fixed_point_matrix} is faster than the criss-cross algorithm,


\begin{table} 
\begin{tabular}{|c||c|cc|ccc|cc|}
\hline
		&					& 	 \multicolumn{2}{c}{error}		&	\multicolumn{3}{|c}{time}	&	\multicolumn{2}{|c|}{iterations}			\\	
Example	& $\alpha_{\epsilon} (A)$ &	Alg. \ref{alg:fixed_point_matrix}	 &	GO		& Alg. \ref{alg:fixed_point_matrix}	& 	GO		&	CC
							&  Alg. \ref{alg:fixed_point_matrix}		&  GO	 \\[.2em]

\hline
\hline
grcar		&	3.1252	&	-2.0e-07	&	-5.0e-07	&	0.38		&2.43		&0.12	&88	&164		\\
hatano	&	3.2077	&	-1.8e-15	&	-1.0e-09	&0.02	&0.09		&0.09	&4	&8	\\
kahan	&	1.2795	&	\phantom{-}2.0e-15	&	-3.2e-08	&0.02	&0.58	&0.07	&5	& 43	\\
landau	&	1.1990	&	-2.2e-15	&	-1.6e-08	&0.04	&0.23	&0.19	&4	& 26		\\
riffle		&	1.2387	&	\phantom{-}2.0e-15	&	-6.3e-09	&0.02	&0.20	&0.07	&5	&22	\\
transient		&	0.4731	&	-4.1e-11	&	-1.4e-08	&0.08	&0.60	&0.09	&6	&35	\\
twisted		&	2.1719	&	-1.2e-10	&	-8.0e-09	&0.04	&0.24	&0.32	&6	&15	\\
\hline 
\end{tabular}
\caption{Accuracy, runtimes (in seconds), iterations of Algorithm \ref{alg:fixed_point_matrix} and GO
on dense examples of size $n = 100$ from {\sf EigTool} for  $\epsilon = 0.2$. CC and GO stand for the
criss-cross algorithm for a matrix \cite[Algorithm 3.1]{BurLO03} and 
the Guglielmi, Overton algorithm \cite[Algorithm PSA1]{GugO11}, respectively.}
\label{tab:table1}  
\end{table}

\medskip

\noindent
\textbf{Larger sparse examples.} \\
The results of numerical experiments on sparse matrices from {\sf EigTool}  
for $\epsilon =0.2$ are provided in Table \ref{tab:table5}. 
The column of `gap' refers to the difference of the computed value
of $\alpha_{\epsilon}(A)$ by Algorithm \ref{alg:fixed_point_matrix} and by GO.
This gap is less than $10^{-6}$ in all examples indicating that the algorithms return
similar values. The first two examples (`olmstead' and `supg') are modest in size,
so we have computed $\alpha_{\epsilon}(A)$ for these two examples by also 
employing the criss-cross algorithm. Indeed, the results returned by 
Algorithm \ref{alg:fixed_point_matrix} for these two examples differ from 
those returned by the criss-cross algorithm by amounts smaller than $2\times 10^{-8}$.
GO requires more iterations compared to Algorithm \ref{alg:fixed_point_matrix} in these examples.
This translates into a longer computation time for GO in the majority of the examples.
The convergence of GO is notably slow for the `supg', `markov', `pde' examples.
\begin{table} 
\begin{tabular}{|c||c|c|ccc|cc|}
\hline
		&					& 		&	\multicolumn{3}{|c}{time}	&	\multicolumn{2}{|c|}{iterations}			\\	
Example, $n$	& $\alpha_{\epsilon} (A)$ &	gap	&	Alg. \ref{alg:fixed_point_matrix}	& 	GO		&	CC
							&  Alg. \ref{alg:fixed_point_matrix}		&  GO	 \\[.2em]
\hline
\hline
olmstead, 500	&	4.7175	&	\phantom{-}4.5e-11		&	0.11	&0.07	&1.66	&2	&3	\\	
supg, 400	&	0.2942	&	\phantom{-}8.7e-07	&	0.09	&0.78		&29.99	&6	&721		\\	
dwave, 2048	&	1.1788	&	\phantom{-}5.4e-08	&	0.08	&0.16	& ---	&2	&17	\\
markov, 5050	&	1.2457	&	\phantom{-}4.1e-08	&	1.12		&2.60	& ---	& 19	& 66	\\
pde, 2961 	&	10.3775	&	\phantom{-}2.6e-07	&	0.78	&1.28	& ---	&40	& 85	\\
rdbrusselator, 3200	&	0.6037	&	-8.6e-12	&	0.25	& 0.11	& ---	& 4	&6	\\
\hline 
\end{tabular}
\caption{Comparison of the computed values of $\alpha_{\epsilon}(A)$, runtimes (in seconds), number of iterations
of Algorithm \ref{alg:fixed_point_matrix} and GO 
				on sparse large examples from {\sf EigTool} for $\epsilon = 0.2$.}
\label{tab:table5}							
\end{table}

\subsubsection{Accuracy of the algorithms on a large dataset}\label{sec:num_exp_acc}
We generate a dataset consisting of 1000 matrices by means of
the MATLAB command \\[.35em]
\phantom{aa}
\texttt{A = c1*randn(n) + c2*sqrt(-1)*randn(n)} \\[.35em] 
where the size $n \in [200,400]$ and the coefficients $c1 , c2 \in [0.2,4]$ are also chosen randomly
by using uniform distributions. 
Here, we test the accuracy of Algorithm \ref{alg:fixed_point_matrix} on this dataset, 
by varying the parameter ${\mathcal N}$ (see the last part of Section \ref{sec:software}
for a detailed explanation of this parameter), controling the number of times 
Algorithm \ref{alg:fixed_point_matrix} is executed, with each execution starting from 
a different $z_0$, chosen based on eigenvalue perturbation theory. 
Recall that the largest value returned by these executions
is taken as the computed value of $\alpha_{\epsilon}(A)$. We also provide a comparison
with the accuracy of GO. For each algorithm, we accept the 
computed value of the $\epsilon$-pseudospectral abscissa correct if this estimate and
the $\epsilon$-pseudospectral abscissa returned by the CC algorithm do not differ 
by more than $2 \cdot 10^{-6}$.  

For various choices of ${\mathcal N}$, the percentages of correct values of $\alpha_{\epsilon}(A)$
returned by Algorithm \ref{alg:fixed_point_matrix} with ${\mathcal N}$ restarts 
are presented in Table \ref{tab:table7} along with the percentages for GO for $\epsilon = 0.01, 0.2, 0.5$. 
For  smaller $\epsilon$ values, Algorithm \ref{alg:fixed_point_matrix} even with ${\mathcal N} = 1$
is usually more accurate than GO. However, the accuracy of Algorithm \ref{alg:fixed_point_matrix}
degrades for larger $\epsilon$ values, which can be attributed to fact that the eigenvalue
perturbation theory ideas used to initialize Algorithm \ref{alg:fixed_point_matrix} (i.e., to choose $z_0$) 
are not as accurate anymore.
Still, with a modest number of restarts (e.g., with ${\mathcal N} = 3$), a high
accuracy is reached even for larger values of $\epsilon$, such as $\epsilon = 0.5$.
For all three values of $\epsilon$ in this example,
Algorithm \ref{alg:fixed_point_matrix} with ${\mathcal N} = 7$ restarts is as accurate 
as the criss-cross algorithm. 
\begin{table}
    \begin{tabular}{|c||c||c|c|c|c|c|c|c|}
        \hline
        $\;\; \epsilon$ 
        & $\;\:$GO & ${\mathcal N} = 1$  &	${\mathcal N} = 2$ & ${\mathcal N} = 3$ & ${\mathcal N} = 4$ & ${\mathcal N} = 5$ & ${\mathcal N} = 6$ & ${\mathcal N} = 7$		\\
        \hline
	\hline
	0.01 & 99.8	& 99.9 & 100 &100	& 100	& 100	& 100	& 100		\\
	0.2 & 96.6 & 96.7 & 99.6 & 99.9  	& 99.9	& 100	&	100	&	100		\\
	0.5 & 92.9 & 92.5	& 98.3	& 99.3		& 99.6	& 99.7	& 99.8	& 100		\\
	 \hline
    \end{tabular}
\caption{The percentages of correctly computed $\epsilon$-pseudospectral abscissa on the dataset consisting
of 1000 random matrices for Algorithm \ref{alg:fixed_point_matrix} with ${\mathcal N}$ restarts, 
as well as for the GO algorithm for $\epsilon = 0.01, 0.2, 0.5$.}
\label{tab:table7}
\end{table}


\section{Concluding remarks}
By employing eigenvalue perturbation theory,
we have derived first-order approximation formulas with errors ${\mathcal O}(\epsilon^2)$
to approximate the $\epsilon$-pseudospectral abscissa of an analytic matrix-valued function,
and more specifically of a matrix. In the matrix setting, we have additionally deduced a
second-order approximation formula with ${\mathcal O}(\epsilon^3)$ error. 
The derived formulas are cheap to compute, and provide remarkably accurate 
approximations for small $\epsilon$.

For larger values of $\epsilon$, we have tailored two fixed-point iterations
to compute the $\epsilon$-pseudospectral abscissa in the general matrix-valued
function setting by making use of the deduced first-order formulas. We have shown
that, under a nondegeneracy assumption, the fixed-point iterations can only converge  
to a right boundary point of the $\epsilon$-pseudospectrum with a vertical tangent line,
likely to be a locally rightmost point.
The specialization of the fixed-point iterations to the matrix setting is also presented.

When these locally convergent fixed-point iterations are initialized with points that
are close to globally rightmost points in the $\epsilon$-pseudospectrum, which 
is obtained as a by-product of the approximation formulas based on eigenvalue
perturbation theory, they converge to the globally rightmost point in a vast majority of 
the cases. 

We make MATLAB implementations of the resulting fixed-point iterations 
with proper initializations publicly available \cite{AhmM25}.
The computational bottleneck for our current implementations is that they
compute all eigenvalues and eigenvectors once in order to benefit from
formula (\ref{eq:cru_eig}) or (\ref{eq:cru_eig3}) that yields an estimate
of the eigenvalue moving furthest to right under perturbations of
norm at most $\epsilon$. The eigenvalue $\mu_0$ satisfying (\ref{eq:cru_eig}) 
or (\ref{eq:cru_eig3}) can possibly be estimated by a subspace approach,
such as a Krylov subspace approach, without going through the burden
of computing all eigenvalues. This appears to be a direction worth exploring.

\appendix

\section{Proof of Theorem \ref{thm:berror_NEP}}\label{sec:proof_bw_nle_error}
\begin{proof}
Observe that
\begin{equation*}
\begin{split}
	\left\{ (T + \underline{\Delta T})(z)	\right\}	v
			\;	=	\;
	T(z) v	+	\Delta T(z) v			
			\; &	=	\;
	\sigma_{\min}(T(z)) u		+	\sum_{j=1}^\kappa t_j(z) w_j \underline{\Delta T}_j v		\\
			&	=	\;
	\sigma_{\min}(T(z)) u		+	
		\sum_{j=1}^\kappa t_j(z) w_j \left\{ \frac{ -\varphi(z) w_j \overline{t_j(z)} u }  
								{ \sqrt{\sum_{\ell = 1}^\kappa w^2_{\ell} | t_{\ell}(z) |^2} } \right\}		\\
			&	=	\;
	\sigma_{\min}(T(z)) u		-	\left\{ \sqrt{\sum_{\ell = 1}^\kappa w^2_{\ell} | t_{\ell}(z) |^2} \right\}   \varphi(z) u
		\;	=	\;	0	\:	,
\end{split}	
\end{equation*}
where the second equality is due to the fact $u$, $v$ are consistent unit left, unit right singular vectors
corresponding to $\sigma_{\min}(T(z))$, and we employ the definition of $\underline{\Delta T}_j$
in (\ref{eq:defn_Tj}) for the third equality. Consequently, $\mathrm{det}\left\{ (T + \underline{\Delta T})(z) \right\} = 0$. 
By analogous calculations $u^\ast \left\{ (T + \underline{\Delta T})(z)	\right\} = 0$ as well. 
Furthermore,
$
	\left\|
			\left[
			\begin{array}{ccc}
				\underline{\Delta T}_1	&	\dots		&	\underline{\Delta T}_{\kappa}
			\end{array}
			\right]
	\right\|_2	=	\varphi(z)
$
so that $\underline{\Delta T} \in {\mathcal P}_{\varphi(z)}$ 
from the definition of ${\mathcal P}_{\varphi(z)}$ in (\ref{eq:defn_NEP_pspec}).
Combining $\underline{\Delta T} \in {\mathcal P}_{\varphi(z)}$ with 
$\mathrm{det}\left\{ (T + \underline{\Delta T})(z) \right\} = 0$ yield
\begin{equation}\label{eq:berror_ub}
	\inf\{ \epsilon \; | \; \exists \Delta T \in {\mathcal P}_{\epsilon}	
			\text{ s.t. }	\mathrm{det}\left\{ (T + \Delta T)(z) \right\} = 0	\}
				\;	\leq		\;	\varphi(z)	\:	.
\end{equation}

Take any 
$\Delta T(\lambda) 
			= t_1(\lambda) w_1\Delta T_{\, 1} + \dots 
						+ t_\kappa(\lambda) w_\kappa \Delta T_{\, \kappa}$
such that $\mathrm{det}\left\{ (T + \Delta T)(z) \right\} = 0$.
By the Eckart-Young-Mirsky theorem \cite[Theorem 2.5.3]{GolVL96}, we have
\begin{equation*}
\begin{split}
		\sigma_{\min}(T(z))	
			\;	\leq		\;
		\| \Delta T(z) \|_2 	\;	&	=	\;
		\left\|
		\left[
			\begin{array}{ccc}
				\Delta T_1		&	\dots		&	\Delta T_\kappa
			\end{array}
		\right]
		\left[
			\begin{array}{c}
				w_1 \, t_1(z) I		\\
					\vdots		\\
				w_\kappa  \, t_\kappa(z) I
			\end{array}
		\right]	
		\right\|_2			\\
			\;	&		\leq		\;
			\left\|
			\left[
				\begin{array}{ccc}
					\Delta T_1		&	\dots		&	\Delta T_\kappa
				\end{array}
			\right]
		\right\|_2
		\sqrt{w_1^2 | t_1(z) |^2 +  \dots  + w_\kappa^2 | t_\kappa(z) |^2}	\:	.
\end{split}
\end{equation*}
This shows that 
$
		\left\|
			\left[
				\begin{array}{ccc}
					\Delta T_1		&	\dots		&	\Delta T_\kappa
				\end{array}
			\right]
		\right\|_2	\geq \varphi(z)
$, that is $\Delta T \not\in {\mathcal P}_{\epsilon}$ for any $\epsilon < \zeta(z)$, implying
\begin{equation}\label{eq:berror_lb}
	\inf\{ \epsilon \; | \; \exists \Delta T \in {\mathcal P}_{\epsilon}	
			\text{ s.t. }	\mathrm{det}\left\{ (T + \Delta T)(z) \right\} = 0	\}
				\;	\geq		\;	\varphi(z)	\:	.
\end{equation}
Combining (\ref{eq:berror_ub}) and (\ref{eq:berror_lb}), we deduce
$\inf\{ \epsilon \; | \; \exists \Delta T \in {\mathcal P}_{\epsilon}	
			\text{ s.t. }	\mathrm{det}\left\{ (T + \Delta T)(z) \right\} = 0	\} = \varphi(z)$.
Moreover, this infimum is attained for $\Delta T = \underline{\Delta T}$,
so can be replaced by minimum.
\end{proof}

\section{Proof of Lemma \ref{Lem:char_rightp}}\label{sec:opt_rightmost_pt}
\begin{proof}
Suppose $z \in {\mathbb C}$ is a locally rightmost point in $\Lambda_{\epsilon}(T)$.
By the continuity of the smallest singular value
$\sigma_{\min}(M(\cdot))$, assuming $\sigma_{\min}(M(z)) = 0$ implies that
for every ball ${\mathcal B}_r(z) = \{ \widetilde{z} \; | \; |\widetilde{z} - z | < r \}$ 
with positive radius $r$ sufficiently small, we have $\sigma_{\min}(M(\lambda)) \leq \epsilon \:$
$\forall \lambda \in {\mathcal B}_r(z)$, that is $\lambda \in \Lambda_{\epsilon}(T) \:$
$\forall \lambda \in {\mathcal B}_r(z)$. But this contradicts with the fact that
$z$ is a locally rightmost point. Thus, $\sigma_{\min}(M(z)) > 0$. 
Recall also that $\sigma_{\min}(T(z))$ is simple, so $\sigma_{\min}(M(z))$ is simple as well.
Letting $(x,y) = (\text{Re}(z), \text{Im}(z)) \in {\mathbb R}^2$, and
by employing continuity once again, $\sigma_{\min}({\mathcal M}(u))$ 
is positive and simple for all $u$
in an open neighborhood of ${\mathcal U}$ of $(x,y)$. Then $\sigma_{\min}({\mathcal M}(u))$ 
is real analytic (so is
differentiable infinitely many times)
in ${\mathcal U}$ \cite[Lemma 2.5]{MehM24}, in particular at $(x,y) \in {\mathcal U}$.


Now, letting $\sigma(s_1, s_2) :=  \sigma_{\min}({\mathcal M}(s_1 , s_2))$ 
for $s_1, s_2 \in {\mathbb R}$, consider the optimization problem
\begin{equation}\label{eq:pspa_opt_char}
	\max \{  s_1 \; | \; s = (s_1, s_2) \in {\mathbb R}^2 \text{ s.t. } \sigma(s_1, s_2) \leq \epsilon \}
	\:	,
\end{equation}
for which $(x,y) = (\text{Re}(z), \text{Im}(z))$ is a local maximizer.
From the analytical formula for the derivative of a singular value function \cite[Lemma 2.5]{MehM24}, we have
\begin{equation}\label{eq:sval_par_ders}
	\nabla \sigma(x,y)
		\;	=	\;
	\left(
		\frac{\partial \sigma}{\partial s_1}(x,y)	,
		\frac{\partial \sigma}{\partial s_2}(x,y)
	\right)	
		\;		=	\;
	\bigg(
		\text{Re}\{ u^\ast {\mathcal M}_{s_1}(x,y) v \} \:	,  \:	\text{Re} \{ u^\ast {\mathcal M}_{s_2}(x,y) v \}
	\bigg)	\:	,
\end{equation}
where we assume without loss of generality that the singular vectors $u$, $v$ are of unit norm
(i.e., as $u^\ast {\mathcal M}_{s_1}(x,y) v$, $u^\ast {\mathcal M}_{s_2}(x,y) v$ are positive multiples 
of $\widehat{u}^\ast {\mathcal M}_{s_1}(x,y) \widehat{v}$, $\widehat{u}^\ast {\mathcal M}_{s_2}(x,y) \widehat{v}$,
respectively, for any other
pair of consistent left, right singular vectors $\widehat{u}$, $\widehat{v}$ corresponding to
$\sigma_{\min}(T(z))$, it suffices to show $u^\ast M'(z) v$ is positive, real).
Thus, it follows from ${\mathcal S}(z) \neq 0$ (by the nondegeneracy of $z$) that
$\nabla \sigma(x,y) \neq 0$. As a result, the linear independence constraint
qualification holds for the optimization problem in (\ref{eq:pspa_opt_char}) at $(x,y)$.
From the first-order necessary conditions \cite[Theorem 12.1]{NocW00} applied
to (\ref{eq:pspa_opt_char}), there exists a nonnegative Lagrange multiplier $\mu$ such that
\begin{equation}\label{eq:opt_cond}
		1	-	\mu \frac{\partial \sigma}{\partial s_1}(x,y)	\;	=	\;	0	\:	
		\quad\quad\;	\text{and}		\quad\quad\;\;
		\mu	\frac{\partial \sigma}{\partial s_2}(x,y)	\;	=	\;	0	\:	.
\end{equation}
It follows from the left-hand equality in (\ref{eq:opt_cond}) that $\mu \neq 0$, so $\mu > 0$.
Now plugging the expressions in (\ref{eq:sval_par_ders}) for the partial derivatives of $\sigma(s_1, s_2)$
at $(x,y)$ in (\ref{eq:opt_cond}), we obtain
\[
	\text{Re}\{ u^\ast {\mathcal M}_{s_1}(x,y) v \}	=	1/\mu
	\quad\quad\;	\text{and}		\quad\quad\;\;
	\text{Re} \{ u^\ast {\mathcal M}_{s_2}(x,y) v \}	=	0	\:	.	
\]
Consequently, ${\mathcal S}(z)$ is real and positive as claimed. 
\end{proof}

\bigskip

\noindent
\textbf{Acknowledgments.}
Part of this work has been carried out while EM was visiting Gran Sasso Science
Institute (GSSI) in L'Aquila, Italy. EM is grateful to GSSI for hosting him,
as well as Nicola Guglielmi for providing financial support and insightful discussions 
during this visit.

\bibliography{AhmM25}

\end{document}